\documentclass[10pt]{article}

\PassOptionsToPackage{dvipsnames,x11names}{xcolor}
\usepackage{amsmath,amssymb,amsthm}
\usepackage{mathtools}
\usepackage{microtype}
\usepackage{hyperref}
\usepackage[capitalize,nameinlink,noabbrev]{cleveref}
\usepackage{booktabs}
\usepackage{tikz}
\usetikzlibrary{matrix,calc,fit,decorations.pathreplacing}
\usepackage{nicematrix}
\usepackage{graphicx}
\usepackage[margin=1in]{geometry}
\usepackage{enumitem}
\usepackage{bm}
\let\AMSmathbb\mathbb
\usepackage{bbold}
\let\bboldmathbb\mathbb
\let\mathbb\AMSmathbb
\usepackage[backend=biber,style=alphabetic]{biblatex}
\newtheorem{theorem}{Theorem}[section]
\newtheorem{lemma}[theorem]{Lemma}
\newtheorem{proposition}[theorem]{Proposition}
\newtheorem{corollary}[theorem]{Corollary}

\theoremstyle{definition}

\newtheorem{example}[theorem]{Example}

\theoremstyle{remark}
\newtheorem{remark}[theorem]{Remark}

\newcommand{\eg}{e.g.}
\newcommand{\R}{\mathbb{R}}
\newcommand{\Q}{\mathbb{Q}}
\newcommand{\N}{\mathbb{N}}
\newcommand{\T}{\mathbb{T}}
\newcommand{\TP}{\mathbb{TP}}
\newcommand{\TPfin}{\TP_{\mathrm{fin}}}
\newcommand{\TPpmfin}{\TP_{\pm,\mathrm{fin}}}
\newcommand{\h}{\mathcal{H}}
\newcommand{\W}{\mathcal{W}}
\renewcommand{\P}{\mathbb{P}}
\newcommand{\Ppm}{\P_{\pm}}
\newcommand{\Z}{\mathbb{Z}}
\newcommand{\x}[0]{\bm{x}}
\newcommand{\y}[0]{\bm{y}}
\newcommand{\s}[0]{\bm{s}}
\newcommand{\p}[0]{\bm{p}}
\renewcommand{\a}{\bm{a}}
\renewcommand{\u}{\bm{u}}
\renewcommand{\v}{\bm{v}}
\newcommand{\w}{\bm{w}}
\newcommand{\q}{\bm{q}}
\newcommand{\trop}[0]{\mathbb{T}}
\newcommand{\strop}[0]{\trop_{\!\pm}}
\DeclareMathOperator{\sgn}{sign}
\DeclareMathOperator{\tdet}{tdet}
\DeclareMathOperator{\tper}{tper}
\DeclareMathOperator{\CT}{CT}
\DeclareMathOperator{\CMF}{CMF}
\DeclareMathOperator{\RT}{RT}
\newcommand{\OT}{\ensuremath{\ell\text{-}\mathrm{OT}}}
\newcommand{\CMFT}{\CMF}
\newcommand{\CMFP}{\CMF_{\P}}
\newcommand{\CMFTbullet}{\CMF^{\bullet}}
\newcommand{\Sym}{\mathfrak{S}}
\newcommand{\CTT}{\CT}
\newcommand{\OTR}{\OT_{\R}}
\newcommand{\OTT}{\OT_{\strop}}
\newcommand{\OTP}{\OT_{\Ppm}}
\newcommand{\card}[1]{\#\!\left(#1\right)}
\newcommand{\Oh}{\mathcal{O}}
\newcommand{\A}{\mathcal{A}}
\newcommand{\Amax}{\mathcal{A}^{\max}}

\newcommand{\Hmax}{H^{\max}}
\newcommand{\Hmin}{H^{\min}}
\newcommand{\Hsmax}{H^{\max}_{\pm}}

\newcommand{\Smax}{S^{\max}}
\newcommand{\Smin}{S^{\min}}
\newcommand{\Fmax}{F^{\max}}
\newcommand{\Fmin}{F^{\min}}
\newcommand{\tzero}{\bboldmathbb{0}}
\newcommand{\tone}{\bboldmathbb{1}}

\newcommand{\pth}[1]{\left(#1\right)}

\hypersetup{colorlinks=true, linkcolor=NavyBlue, citecolor=OliveGreen, urlcolor=NavyBlue}

\newcommand{\affilmark}[1]{\textsuperscript{\,#1}}
\newcommand{\affiltext}[2]{\textsuperscript{#1}\,\textit{#2}}

\title{The Number of Tropical Order Types and Coherent Matching Fields}

\author{%
  Xavier Allamigeon\affilmark{1}
  \and Yiyuan Chen\affilmark{1}
  \and St\'ephane Gaubert\affilmark{1}
  \and Xavier Goaoc\affilmark{2}
}

\date{%
  \normalsize
  \affiltext{1}{CMAP, \'Ecole polytechnique \& INRIA Saclay, France}\\[4pt]
  \affiltext{2}{LORIA, Universit\'e de Lorraine \& INRIA Nancy, France}
}

\begin{document}

\maketitle

\begin{abstract}
  We investigate counting problems for discrete (geometric) structures in tropical settings, motivated by the comparison of the expressivity of linear programs and mean-payoff games.

  First, we give an upper bound on the number of coherent matching fields of type $(n,r)$ and a lower bound on the number of combinatorial types of $n$-hyperplane sequences in the max-plus tropical space of dimension~$r-1$. These structures are closely related, and our bounds establish that as $n \to \infty$, for every fixed $r$ their number equals $n^{(r-1)^2n+o(n)}$. This in turn implies that the number of regular triangulations of the products of simplices $\Delta^{n-1}\times\Delta^{r-1}$ is also $n^{(r-1)^2n+o(n)}$, thereby answering a question of Santos.

  Second, we obtain that the number of labeled order types over the reals and over the {\em signed} tropical numbers  coincide at first order, establishing a comparison of the expressivity of mean-payoff games and linear programs as measured by labeled order types. 

Last, the subclass of mean-payoff games known as parity games arise from linear programs by a ``double tropicalization.'' We show that the number of labeled order types obtained after this double tropicalization collapses.  

\end{abstract}

\section{Introduction}\label{sec:intro}

\subsection{Motivation: complexity of linear programming and games}

Smale's ninth problem~\cite[Problem~9]{Smale1998} asks for the existence of a strongly polynomial algorithm to decide the feasibility
problem in linear programming. This requires the number of arithmetic operations to be polynomially
bounded in the number of variables and constraints, independently of the bit-size of the coefficients. Another 
open problem concerns the complexity of mean-payoff games. Gurvich, Karzanov and Khachiyan~\cite{GurvichKarzanovKhachiyan1988} asked
in 1988 whether these games can be solved in polynomial time. The two questions are linked:
a strongly polynomial algorithm of a combinatorial nature (this covers a very broad class
of pivoting algorithms) for linear programming would solve mean-payoff games in polynomial
time~\cite{AllamigeonBenchimolGaubertJoswig2014}. Parity games constitute an intriguing subclass of
mean-payoff games~\cite[Section~2.5.1]{Fearnley2010}.  
Being amenable to combinatorial and automata-theory methods~\cite{Calude2017,parity_lower_bounds}, parity games might seem computationally simpler than general mean-payoff games. No polynomial-time reduction from mean-payoff games to parity games is known.

\medskip

This raises the question of comparing the \emph{expressivity} of
linear programs, mean-payoff games, and parity games. We do so by
comparing some underlying combinatorial structures. To do so, we adopt
the tropical perspective: mean-payoff games are equivalent to tropical
linear programs~\cite{AkianGaubertGuterman2012}, which are obtained by
deformation limits of classical linear programs, hence, they
constitute a subclass of linear programs. The standard priority
encoding of parity games as mean-payoff games (see
Section~\ref{sec:parity-numbers}) can also be viewed as a
tropicalization in the large-base limit. Thus, parity games can be
viewed through a double tropicalization of linear programming.

\medskip

As shown by Bland, pivoting algorithms in linear programming can be expressed
in terms of a combinatorial structure, namely oriented matroids~\cite{Bland1977Abstraction}. 
Furthermore, mean-payoff
games and parity games correspond to special classes of realizable oriented matroids,
called tropically and parity realizable, respectively.
Therefore, counting
the numbers of classically realizable, tropically realizable, and parity realizable, oriented matroids provides quantitative measures allowing one to compare the expressivity of the associated
optimization and game problems, thought of as inputs of combinatorial (pivoting) algorithms.
This, in turn, sheds light on the problem of counting regular subdvisions
of products of simplices~\cite{DevelinSturmfels2004,Santos2004}. In this paper, we investigate the counting of these structures in the language of chirotopes and labeled order types.

\subsection{Contributions}\label{sec:contributions}

Our main contributions concern several related combinatorial structures.

\paragraph{Tropical Hyperplane sequences and coherent matching fields.}

We write $\N_0$ for $\N \cup \{0\}$ and for $n \in \N$ we write $[n] \coloneqq \{1,2, \ldots, n\}$. The \emph{(max-plus) tropical semifield} is the structure \(\T \coloneqq (\R \cup \{-\infty\}, \oplus, \odot)\), i.e., the set $\R \cup \{-\infty\}$ equipped with the tropical addition \(u \oplus v \coloneqq \max\{u, v\}\) and the tropical multiplication \(u \odot v \coloneqq u + v\), with the convention $u + (-\infty)=(-\infty)+u=-\infty$. We use $\R$ to mean $\T\setminus \{-\infty\}$, that is the set of reals equipped with the tropical operations.

The multiplicative group \(\T^\times = \T \setminus \{- \infty\} \) acts on \(\T^r\) by tropical multiplication. We can thus define the \emph{$d$-dimensional projective tropical space} as $\TP^d \coloneqq \bigl(\T^{d+1} \setminus \{(-\infty,\ldots,-\infty)\}\bigr)/\T^\times$. We also work in the affine space $\R^d$ over the tropical semiring. We identify it with the finite part of the projective space, that is $\R^{d} \simeq \TPfin^{d}  \coloneqq  \R^{d+1}/\T^\times$.
The notion of hyperplane carries over to the tropical setting.
The \emph{max-tropical hyperplane} with apex $u\in\R^d$ is
\begin{equation}\label{eq:max-tropical-hyperplane}
  \Hmax(u)\;\coloneqq\;
  \Bigl\{x \in \TPfin^{d} \colon \max_{i\in[d+1]}(x_i-u_i)
  \text{ is attained at least twice}\Bigr\}.
\end{equation}
A tropical hyperplane in $\R^d$ divides the space into $d+1$ sectors (\cref{fig:intro-hyperplane-sequence}(a)). Given a sequence of tropical hyperplanes, we can record the position of a point with respect to each of them, that is, which sectors of each hyperplane contain it; in \cref{fig:intro-hyperplane-sequence}(b), for instance, the point labelled $(\{1\},\{2\},\{3\})$ lies in sector~1 of the first hyperplane, sector~2 of the second and sector~3 of the third. The set of all possible sequences of positions is the \emph{combinatorial type} of the hyperplane sequence (see Section~\ref{sec:pointed-cmf} for a formal definition).

\begin{figure}[t]
  \centering
  \begin{tikzpicture}[
    x=0.7cm, y=0.7cm,
    line cap=round, line join=round,
    maxK/.style={black!75, line width=1pt},
    maxB/.style={RoyalBlue, line width=1pt},
    maxR/.style={BrickRed!85!black, line width=1pt},
    maxG/.style={ForestGreen!55!black, line width=1pt},
    apex/.style={circle, fill, inner sep=0pt, minimum size=3pt},
    pt/.style={circle, fill=black!85, inner sep=0pt, minimum size=3.2pt},
    lab/.style={font=\footnotesize, fill=white, inner sep=1pt, outer sep=1pt},
    slabel/.style={font=\scriptsize\itshape, text=black!70, inner sep=0.5pt},
    panellab/.style={font=\itshape, anchor=south}
  ]

  \begin{scope}[scale=1.4]
    \node[panellab] at (0,3.1) {(a) \(\Hmax(u)\)};
    \begin{scope}
      \clip (-3,-3) rectangle (3,3);
      \draw[maxK] (0,0) -- (-3,0);
      \draw[maxK] (0,0) -- (0,-3);
      \draw[maxK] (0,0) -- (3,3);
    \end{scope}
    \node[apex, black!75] at (0,0) {};
    \node[lab, anchor=north east] at (-0.05,-0.05) {$u$};
    \node[slabel] at (1.6,-1.4) {sector 1};
    \node[slabel] at (-1.2,1.6) {sector 2};
    \node[slabel] at (-1.6,-1.6) {sector 3};
  \end{scope}

  \begin{scope}[shift={(6.6,-3.0)}]
    \node[panellab] at (3.5,7.34)
      {(b) \(\Hmax(u^1),\Hmax(u^2),\Hmax(u^3)\)};
    \begin{scope}
      \clip (-0.8,-1.2) rectangle (7.8,7.2);
      \draw[maxB] (2,4) -- ++(-20,0) (2,4) -- ++(0,-20) (2,4) -- ++(20,20);
      \draw[maxR] (4,1) -- ++(-20,0) (4,1) -- ++(0,-20) (4,1) -- ++(20,20);
      \draw[maxG] (5,5) -- ++(-20,0) (5,5) -- ++(0,-20) (5,5) -- ++(20,20);
    \end{scope}
    \node[apex, RoyalBlue] at (2,4) {};
    \node[lab, text=RoyalBlue, anchor=south east] at (1.93,4.07) {$u^1$};
    \node[apex, BrickRed!85!black] at (4,1) {};
    \node[lab, text=BrickRed!85!black, anchor=north west] at (4.07,0.93) {$u^2$};
    \node[apex, ForestGreen!55!black] at (5,5) {};
    \node[lab, text=ForestGreen!55!black, anchor=south east] at (4.93,5.07)
      {$u^3$};

    \node[pt] at (0.6,6.3) {};
    \node[lab, font=\scriptsize, anchor=west] at (0.7,6.3) {$(\varnothing,\{1,2,3\},\varnothing)$};
    \node[pt] at (2.2,2.6) {};
    \node[lab, font=\scriptsize, anchor=west] at (2.3,2.6) {$(\{1\},\{2\},\{3\})$};
    \node[pt] at (5.6,1.3) {};
    \node[lab, font=\scriptsize, anchor=west] at (5.7,1.3) {$(\{1,2,3\},\varnothing,\varnothing)$};
  \end{scope}
\end{tikzpicture}
  \caption{\emph{(a)} A max-tropical hyperplane in \(\R^2\) and its sectors.
  \emph{(b)} A sequence of three max-tropical hyperplanes; each marked point is
  labelled by its type.}
  \label{fig:intro-hyperplane-sequence}
\end{figure}
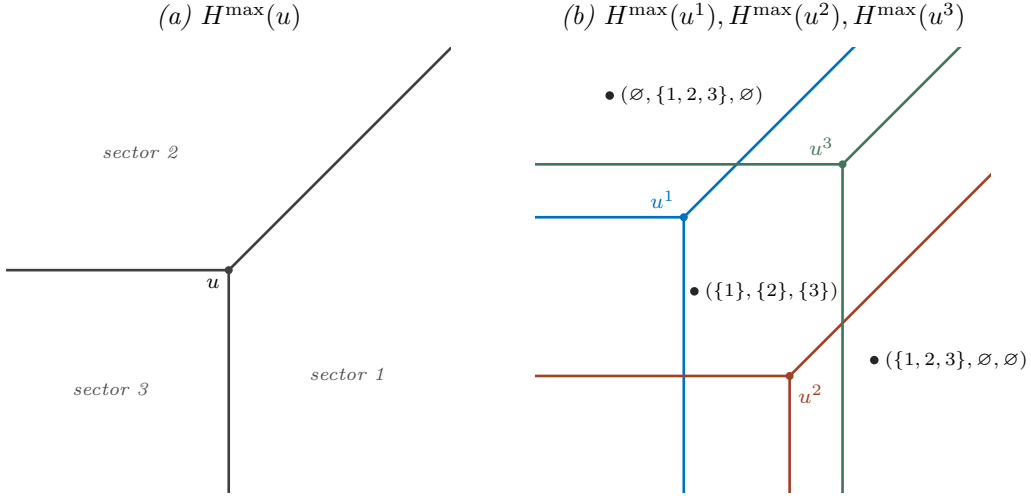
We count only sequences of max-tropical hyperplanes in general position; see Section~\ref{sec:pointed-cmf} for a formal definition. We denote by $\CTT(n,d)$ the set of labeled combinatorial types of such sequences of $n$ hyperplanes in dimension $d$.

\medskip

These combinatorial types of max-tropical hyperplanes in general position are in correspondence with another combinatorial structure: the coherent matching fields~\cite{SturmfelsZelevinski93}. For $I\in\binom{[n]}{r}$, let $\mathcal M_I$ denote the set of the $r!$ perfect matchings of the complete bipartite graph $K_{I,[r]}$. A \emph{matching field} of type $(n,r)$ is a collection $\Lambda=(M_I)_{I\in\binom{[n]}{r}}$ with $M_I\in\mathcal M_I$ for every $I\in\binom{[n]}{r}$. A matching field $\Lambda$ is \emph{coherent} if there exists a matrix $W\in\T^{n\times r}$ such that
\[
  \text{for all } I \in \binom{[n]}{r}, \quad  \text{for all }M\in\mathcal M_I\setminus\{M_I\}, \qquad
  \sum_{(i,k)\in M_I}W_{i,k}
  >
  \sum_{(i,k)\in M}W_{i,k}.
\]
The coherent matching fields have a natural geometric interpretation as they are in bijection with the extreme points of the Minkowski sum of the Newton polytopes of the maximal minors of a general $r\times n$ matrix~\cite[$\mathsection 1$]{SturmfelsZelevinski93}.

\medskip

We bound from above the size of the set $\CMFT(n,r)$ of coherent matching fields of type $(n,r)$ using a connection to Minkowski sums of Birkhoff polytopes due to Sturmfels and Zelevinsky~\cite{SturmfelsZelevinski93}, and bound from below the size of $\CTT(n,d)$ by adapting techniques from affine combinatorial geometry. We then use the bijection between some of the coherent matching fields and combinatorial types of hyperplane sequences (see, \eg{}, \cite{FinkRincon2015}) to propagate these bounds both ways, obtaining that both are tight at first order:

\begin{theorem}\label{thm:main-unsigned}
  As $n \to \infty$ we have
  \begin{enumerate}[label=(\roman*)]
  \item\label{item:cmf} For every fixed $r\geq 2$, $\#\CMFT(n,r) = \phi(n,r)\, n^{(r-1)^2n}$,
    where\footnote{Throughout,
    $\Oh(\cdot)$ denotes a nonnegative quantity, while $\pm\Oh(\cdot)$ denotes
    an error of either sign; unless stated otherwise, implicit constants
    depend only on the fixed rank or dimension.} \[2^{-\Oh(n)}(\log_2 n)^{-(r-1)^2n}
    \le \phi(n,r) \le 2^{\Oh(1)}((r-1)!)^{-rn}.\]
  \item\label{item:ct} For every fixed $d\geq 1$, $\#\CTT(n,d) = \psi(n,d)\, n^{d^2n}$,
    where \[2^{-\Oh(n)}(\log_2 n)^{-d^2n}   \le \psi(n,d) \le 2^{\Oh(\log n)}(d!)^{-(d+1)n}.\]
  \end{enumerate}
\end{theorem}

\noindent
The upper bounds follow from direct counting arguments on integer points in polytopes. The lower bounds require more effort, including specific constructions in tropical geometry.

\paragraph{Regular triangulations of products of simplices.} 

A triangulation of a point set $P \subset \R^d$ is a partition of the convex hull of $P$ into simplices with disjoint interiors and vertices in $P$.
A {\em lifting} of $P = \{p^1,p^2, \ldots, p^n\} \subset \R^d$ to $\R^{d+1}$ is any point set of the form $\{(p^1,\omega^1), (p^2,\omega^2), \ldots, (p^n,\omega^n)\} \subset \R^{d+1}$ with $\omega^i \in \R$ for $i \in [n]$. A triangulation of $P \subset \R^d$ is \emph{regular} if it can be obtained as the projection of a lower envelope of the convex hull of some lifting of $P$ (see \cref{fig:rt-square}). Theorem~\ref{thm:main-unsigned} yields, via a bijection of Develin and Sturmfels~\cite{DevelinSturmfels2004}, a similar bound for the number $\RT(n,r)$ of regular triangulations of the products of simplices $\Delta^{n-1}\times\Delta^{r-1}$:

\begin{corollary}\label{cor:main-triangulations}
  For every fixed $r\geq 2$, as $n \to \infty$,  $\# \RT(n,r) = \tau(n,r) \, n^{(r-1)^2 n}$,
  where \[2^{-\Oh(n)}(\log_2 n)^{-(r-1)^2n}
  \le \tau(n,r) \le 2^{\Oh(\log n)}((r-1)!)^{-rn}.\]
\end{corollary}

\begin{figure}[t]
  \centering
  \begin{tikzpicture}[
    x=1.1cm, y=1.1cm,
    line cap=round, line join=round,
    vertex/.style={circle, fill=black!85, inner sep=0pt, minimum size=3pt},
    edge/.style={black!70, line width=0.75pt},
    lift/.style={black!40, dotted, line width=0.6pt},
    diagA/.style={RoyalBlue, line width=1pt},
    diagB/.style={BrickRed!85!black, line width=1pt},
    hlab/.style={font=\scriptsize, text=black!70, inner sep=1pt},
  ]

  \coordinate (A) at (0,0);
  \coordinate (B) at (1.9,0);
  \coordinate (C) at (0.55,0.7);
  \coordinate (D) at (2.45,0.7);

  \begin{scope}
    \coordinate (AL) at (0,1.5);
    \coordinate (BL) at (1.9,1.8);
    \coordinate (CL) at (0.55,2.5);
    \coordinate (DL) at (2.45,2.2);
    \foreach \P/\Q in {A/AL,B/BL,C/CL,D/DL} { \draw[lift] (\P) -- (\Q); }

    \fill[RoyalBlue!7]  (A) -- (B) -- (D) -- cycle;
    \fill[RoyalBlue!12] (A) -- (C) -- (D) -- cycle;
    \draw[edge] (A) -- (B) -- (D) -- (C) -- cycle;
    \draw[diagA] (A) -- (D);

    \fill[RoyalBlue!8]  (AL) -- (BL) -- (DL) -- cycle;
    \fill[RoyalBlue!14] (AL) -- (CL) -- (DL) -- cycle;
    \draw[edge] (AL) -- (BL) -- (DL) -- (CL) -- cycle;
    \draw[diagA] (AL) -- (DL);

    \foreach \P in {A,B,C,D,AL,BL,CL,DL} { \node[vertex] at (\P) {}; }
    \node[hlab, anchor=east]       at ($(AL)+(-0.06,0)$)    {$2$};
    \node[hlab, anchor=west]       at ($(BL)+(0.06,0)$)     {$3$};
    \node[hlab, anchor=south east] at ($(CL)+(-0.03,0.04)$) {$3$};
    \node[hlab, anchor=west]       at ($(DL)+(0.06,0)$)     {$2$};
  \end{scope}

  \begin{scope}[shift={(3.9,0)}]
    \coordinate (A) at (0,0);
    \coordinate (B) at (1.9,0);
    \coordinate (C) at (0.55,0.7);
    \coordinate (D) at (2.45,0.7);
    \coordinate (AL) at (0,1.8);
    \coordinate (BL) at (1.9,1.5);
    \coordinate (CL) at (0.55,2.2);
    \coordinate (DL) at (2.45,2.5);
    \foreach \P/\Q in {A/AL,B/BL,C/CL,D/DL} { \draw[lift] (\P) -- (\Q); }

    \fill[BrickRed!7]  (A) -- (B) -- (C) -- cycle;
    \fill[BrickRed!12] (B) -- (D) -- (C) -- cycle;
    \draw[edge] (A) -- (B) -- (D) -- (C) -- cycle;
    \draw[diagB] (B) -- (C);

    \fill[BrickRed!8]  (AL) -- (BL) -- (CL) -- cycle;
    \fill[BrickRed!14] (BL) -- (DL) -- (CL) -- cycle;
    \draw[edge] (AL) -- (BL) -- (DL) -- (CL) -- cycle;
    \draw[diagB] (BL) -- (CL);

    \foreach \P in {A,B,C,D,AL,BL,CL,DL} { \node[vertex] at (\P) {}; }
    \node[hlab, anchor=east]       at ($(AL)+(-0.06,0)$)    {$3$};
    \node[hlab, anchor=west]       at ($(BL)+(0.06,0)$)     {$2$};
    \node[hlab, anchor=south east] at ($(CL)+(-0.03,0.04)$) {$2$};
    \node[hlab, anchor=west]       at ($(DL)+(0.06,0)$)     {$3$};
  \end{scope}
\end{tikzpicture}
  \caption{The two regular triangulations of the square
  \(\Delta^1\times\Delta^1\), each the projection of the lower envelope of a
  lifting; the numbers are the heights.}
  \label{fig:rt-square}
\end{figure}
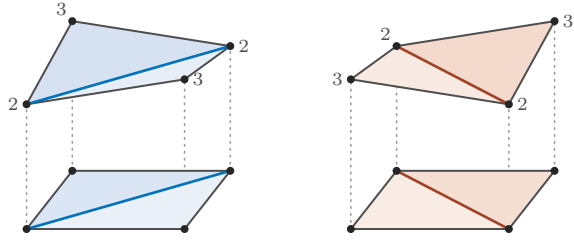

\noindent
Corollary~\ref{cor:main-triangulations} answers a question raised by Santos~\cite{Santos2004}, asking for a tight
estimate of \(\card{\RT(n,r)}\). Santos proved the upper bound
\[
  \log_2 \card{\RT(n,r)}
  \;\leq\;
  r(r-1)(n-1)\log_2\!\Bigl(\tfrac{e}{2}nr\Bigr)
  \;=\;
  r(r-1)\,n\log_2 n + \Oh(n),
\]
and stated that ``\emph{[he] believe[s] the actual number of regular subdivisions to be in $(nr)^{\Theta(nr)}$}'' (remark preceding~\cite[Corollary~5.5]{Santos2004}), that is that \(\log_2\card{\RT(n,r)}=\Theta(nr\log_2(nr))\). However, \cref{cor:main-triangulations} shows that, in a perhaps surprising way, for fixed \(r\) the leading term of \(\log_2\card{\RT(n,r)}\) is at least \((r-1)^2 n\log_2 n\). In particular, Santos's expectation \(\log_2\card{\RT(n,r)}=\Theta(nr\log_2(nr))\) cannot hold uniformly in both
parameters \(n\) and \(r\).

\paragraph{Signed tropicalization of order types.}

We then turn our attention to the \emph{signed tropical space}. Recall that in affine geometry, the chirotope of a point sequence in $\R^d$ is the map assigning to each $(d+1)$-subsequence of points the orientation of the simplex it forms.\footnote{These are a subclass of the chirotopes used to axiomatize oriented matroids~\cite{BjornerEtAl1999}.} Two point sequences are said to {\em have the same labeled order type} if and only if they have the same chirotope.\footnote{Similarly, two point sets are said to {\em have the same order type} if they can be ordered into sequences with the same labeled order type. Formally, (labeled) order types are equivalence classes for these relations. Transfering our bounds on the number of labeled order types to order types involves controlling their symmetries (as is done {\em e.g.} in~\cite{goaoc2023convex}), which is beyond the scope of this paper.} The same definition makes sense over the signed tropical numbers: just as real chirotopes are given by the signs of suitable determinants, tropical chirotopes are given by the signs of tropical determinants (see Section~\ref{sec:signed-chirotopes} for a formal definition). Similarly, two point sequences in the signed tropical space {\em have the same tropical labeled order type} if and only if they have the same tropical chirotope.

\medskip

\begin{figure}[t]
  \centering
  \begin{tikzpicture}[
    x=0.28cm, y=0.28cm,
    >=latex,
    line cap=round, line join=round,
    font=\small,
    declare function={
      D(\v)=(abs(\v)<1) ? (\v/ln(2)) : (sign(\v)*(ln(max(abs(\v),1))+1)/ln(2));
      X(\s)=sign(\s)*(pow(2,abs(\s))-1);
    },
    axis/.style={->, black!60, thin},
    lineB/.style={RoyalBlue,             line width=1.05pt},
    lineR/.style={BrickRed!85!black,     line width=1.05pt},
    lineG/.style={ForestGreen!55!black,  line width=1.05pt},
    viro/.style={samples=100, variable=\x, domain=-7.8:7.8},
    centerpt/.style={circle, fill=black!80, inner sep=0pt, minimum size=2.6pt},
    ptlabel/.style={font=\scriptsize, fill=white, inner sep=1pt, outer sep=2pt},
    inflabel/.style={font=\scriptsize, text=black!60, inner sep=1pt},
    panel/.style={font=\itshape, align=center, text=black!80},
  ]

  \newcommand{\drawaxes}[1]{%
    \draw[axis] (0,0) -- (8.6,0) node[right, font=\scriptsize] {$x_1$};
    \draw[axis] (0,0) -- (-7.8,0);
    \draw[axis] (0,0) -- (0,7.8) node[above, font=\scriptsize] {$x_2$};
    \draw[axis] (0,0) -- (0,-7.8);
    \node[font=\scriptsize, fill=white, inner sep=1pt] at (0,0) {#1};
  }

  \newcommand{\drawinf}{%
    \node[inflabel, anchor=south] at (7.6,0.3) {$\oplus\infty$};
    \node[inflabel, anchor=south] at (-6.9,0.3) {$\ominus\infty$};
    \node[inflabel, anchor=west] at (0.4,7.2) {$\oplus\infty$};
    \node[inflabel, anchor=west] at (0.4,-7.2) {$\ominus\infty$};
  }

  \begin{scope}
    \node[panel, anchor=south] at (0,9.0) {Real lines at \(t=2\)};
    \drawaxes{$0$}
    \begin{scope}
      \clip (-7.5,-7.5) rectangle (8.4,7.6);
      \draw[lineB] (-3,-10) -- (2,10);
      \draw[lineR] (-10,-1) -- (10,1.5);
      \draw[lineG] (-6,10) -- (10,-6);
    \end{scope}
    \node[RoyalBlue, ptlabel, anchor=south east] at (-1.15,-2.55) {$\bm c^1$};
    \node[BrickRed!85!black, ptlabel, anchor=north east] at (-4.4,-0.8) {$\bm c^2$};
    \node[ForestGreen!55!black, ptlabel, anchor=south west] at (2.05,2.25) {$\bm c^3$};
  \end{scope}

  \draw[->, semithick, black!65] (10.0,4.0) -- (11.2,4.0);

  \begin{scope}[shift={(19.3,0)}]
    \node[panel, anchor=south] at (0,9.0) {Signed logarithm};
    \drawaxes{$\tzero$}
    \drawinf
    \begin{scope}
      \clip (-7.5,-7.5) rectangle (8.4,7.6);
      \draw[lineB, viro] plot ({D(X(\x))},{D(4*X(\x)+16)});
      \draw[lineR, viro] plot ({D(X(\x))},{D((X(\x)+16)/8)});
      \draw[lineG, viro] plot ({D(X(\x))},{D(32-X(\x))});
    \end{scope}
  \end{scope}

  \draw[->, semithick, black!65] (29.3,4.0) -- (30.5,4.0)
    node[midway, above=3pt, font=\scriptsize, fill=white, inner sep=1pt]
      {$t\to\infty$};

  \begin{scope}[shift={(38.6,0)}]
    \node[panel, anchor=south] at (0,9.0) {Tropical limit};
    \drawaxes{$\tzero$}
    \drawinf
    \begin{scope}
      \clip (-7.5,-7.5) rectangle (8.4,7.6);
      \draw[lineB] (8.44,10.44) -- (3.44,5.44) -- (-3.44,5.44) -- (-3.44,-5.44)
        -- (-8.44,-10.44);
      \draw[lineR] (8.44,5.44) -- (5.44,2.44) -- (-5.44,2.44) -- (-5.44,-2.44)
        -- (-8.44,-5.44);
      \draw[lineG] (-8.44,8.44) -- (-6.44,6.44) -- (6.44,6.44) -- (6.44,-6.44)
        -- (8.44,-8.44);
    \end{scope}
    \node[centerpt] at (3.44,5.44) {};
    \node[centerpt] at (5.44,2.44) {};
    \node[centerpt] at (6.44,6.44) {};
    \node[RoyalBlue, ptlabel, anchor=north west] at (3.6,5.3) {$\ell_1$};
    \node[BrickRed!85!black, ptlabel, anchor=south east] at (5.3,2.6) {$\ell_2$};
    \node[ForestGreen!55!black, ptlabel, anchor=west] at (6.55,-3) {$\ell_3$};
  \end{scope}
\end{tikzpicture}
  \caption{Tropicalization via deformation limits, illustrated on lines for better readability.}
  \label{fig:signed-puiseux-to-tropical}
\end{figure}
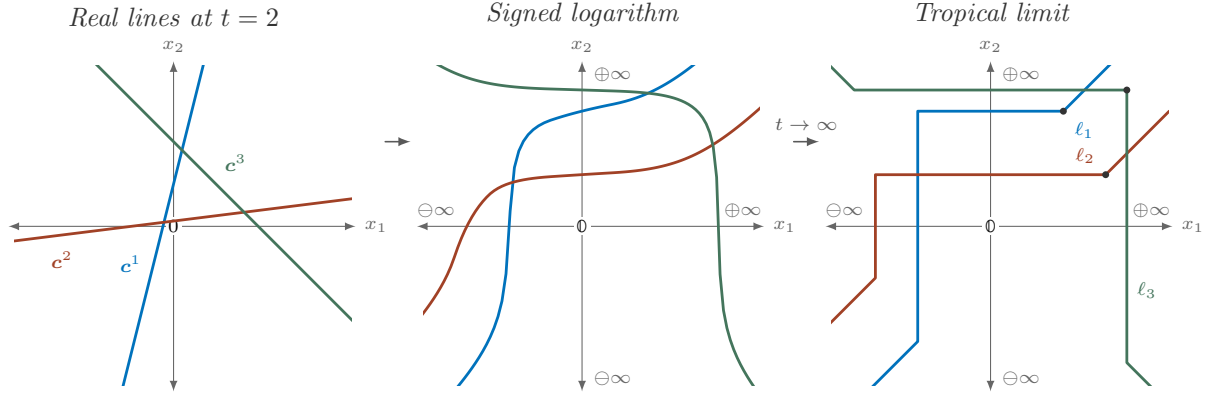
\medskip

Let $\OTT(n,d)$ denote the set of tropical labeled order types obtained from limits of
\(n\)-point sequences in \(\R^d\) with signed monomial $\pm t^a$
with generic exponents. These labeled order types are dual to combinatorial types of signed affine tropical hyperplane
sequences.
Although combinatorial types of hyperplane sequences in the signed and unsigned space are actually different geometric objects, we can adapt the strategy and the key ingredients for the proofs of Theorem~\ref{thm:main-unsigned} to prove:

\begin{theorem}\label{thm:main-signed}
  For every fixed $d \geq 1$, as $n\to\infty$, $\#\OTT(n,d) = \eta(n,d)\, n^{d^2n}$,
  where \[2^{-\Oh(n)}(\log_2 n)^{-d^2n}
  \le \eta(n,d) \le 2^{dn+\Oh(1)}(d!)^{-(d+1)n}.\]
\end{theorem}

\noindent
As we spell out in Section~\ref{sec:signed-real-limits}, every tropical labeled order type corresponds to a real labeled order type, in the sense that they share the same chirotope. The number of labeled order types in $\R^d$ is known to be $2^{d^2 n \log_2 n \pm O(n)}$ (see below), so Theorem~\ref{thm:main-signed} reveals that the numbers of real and tropical labeled order types coincide at first order. (The ratio could still be exponential, though!) 

\medskip

The connection between real and tropical labeled order types proceeds via parametric families of real point sequences whose coordinates are of the form $\pm t^a$, where $t$ is a real parameter and $a$ a real number. As $t \to \infty$, the chirotope of this parameterized sequence eventually stabilizes (see \cref{fig:signed-puiseux-to-tropical}); the limit chirotope corresponds to the chirotope of the tropical points given by the exponents $a$ and the signs of the monomials. The same connection relates mean-payoff games, or equivalently tropical linear programs~\cite{AkianGaubertGuterman2012}, to linear programs over non-archimedean fields~\cite{Allamigeon2015}.

\paragraph{Parity coherent matching fields.}

A remarkable subclass of tropical labeled order types arises from a ``double tropicalization'' construction, taking a deformation limit of tropical numbers. This deformation is key to the analysis of parity games. Following this, we call parity coherent matching fields and parity chirotopes the special coherent matching fields and chirotopes arising in this deformation limit (see Section~\ref{sec:parity} for a formal definition). Two point sequences over these numbers {\em have the same parity labeled order type} if and only if they have the same parity chirotope. We write $\CMFP(n,r)$ for the set of parity coherent matching fields of type $(n,r)$ and $\OTP(n,d)$ for the set of parity labeled order types of labeled $n$-point sequences in dimension $d$.

\begin{theorem}\label{thm:main-parity}
  As $n\to\infty$ we have:
  \begin{itemize}
  \item[(i)] For every fixed $r\geq 2$, $\#\CMFP(n,r) = \phi_{\P}(n,r)\, n^{(r-1)n}$, where $2^{-\Oh(n)} \le \phi_{\P}(n,r) \le 2^{\Oh(n)}$.
  \item[(ii)]For every fixed $d\geq 1$, $\#\OTP(n,d) = \eta_{\Ppm}(n,d)\, n^{dn}$, where $2^{-\Oh(n)} \le \eta_{\Ppm}(n,d) \le 2^{\Oh(n)}$.
  \end{itemize}
\end{theorem}

\noindent
Although tropicalization preserves the first-order growth of the number of real labeled order types (Theorem~\ref{thm:main-signed}), Theorem~\ref{thm:main-parity} shows that double tropicalization causes the numbers of coherent matching fields and labeled order types to collapse, for $r\geq 3$, respectively~$d\geq 2$.

\subsection{Context and related work}\label{sec:related-work}

Let us comment on how our work relates to various lines of research, then compare our techniques to the classical methods used to count combinatorial structures.

\paragraph{Oriented matroids and their relatives.}

Associating a combinatorial type to a geometric object is a classical idea, and oriented matroids are an important example. The development of oriented matroid theory started in the 1970's in the work of Bland~\cite{Bland1977PivotRules,Bland1977Abstraction}, of Folkman and Lawrence~\cite{FolkmanLawrence1978}, and of Las Vergnas~\cite{LasVergnas1975Matroides,LasVergnas1980Convexity}. They were rediscovered several times over, for instance as the {\em (labeled) order types} of Goodman and Pollack~\cite{GoodmanPollack1983Multidimensional} for higher-dimensional sorting, or as the CC-systems of Knuth~\cite{Knuth1992Axioms} for convex hull algorithms. The {\em topological representation theorem} of Folkman and Lawrence~\cite{FolkmanLawrence1978} asserts that every uniform oriented matroids of rank $r$ can be realized as a signed arrangement of pseudospheres in $\mathbb{S}^{r-1}$; for rank $3$ these are the {\em pseudoline arrangements}. The oriented matroids arising from point configurations or from hyperplane configurations in $\R^d$ are called {\em realizable}; they have {\em rank} $d+1$ and when the points or hyperplanes are in general position they are called {\em uniform}. The labeled order types of points in $\R^d$ are actually the realizable acyclic oriented matroids of rank $d+1$. All uniform oriented matroid of rank~$3$ with up to $8$ elements are realizable~\cite[Theorem~5.3.2]{felsner2017pseudoline}.

The oriented matroids arising from tropical point configurations or from hyperplane configurations in signed tropical spaces are called {\em tropically realizable}. An early exploration of tropical oriented matroids appeared in~\cite[Chap.~9]{Skomra2018}, in which it is shown that all uniform oriented matroids of rank $3$ with up to $7$ elements are tropically realizable. This was subsequently refined in~\cite{delambilly}, showing that the same is true up to $8$ elements, and that the corresponding order types are in fact parity realizable. An efficient enumeration algorithms for coherent matching fields has been developed in~\cite{AllamigeonChen2026}; as a by-product, it allows one to count tropical order types in rank $3$ up to $9$ elements and shows that some rank 3 with 9 elements classical realizable order types are not tropically realizable. We note that an alternative combinatorial approach to mean-payoff games has
been recently developed in~\cite{LohoMaatSkomra2025}, using a notion
of cycle patterns instead of order types.

\paragraph{Counting geometric structures.}

The problem of counting combinatorial types of geometric objects has a rich history, see for example~\cite{alon1986number} and the discussion therein for an account of the early history of the problem of counting face lattices of polytopes. The number of uniform oriented matroids of rank $d+1$ is not known, not even at first order. For rank~$3$, although the number is known to be~$2^{\Theta(n^2)}$ the precise asymptotics is not known: the best known lower bound is $2^{0.2721\,n^2}$~\cite{CortesKuhnastDallantFelsnerScheucher2024Pseudolines} and the best known upper bound is $2^{0.6496\,n^2}$~\cite{Dallant2025Zone}. The number of uniform {\em realizable} oriented matroids of rank $d+1$ is known to be $2^{d^2 n \log n \pm O(n)}$ (see below); in particular, already at first order one can see that an overwhelming majority of the oriented matroids are not realizable. Theorem~\ref{thm:main-signed} reveals that the contrast is not so sharp between realizable and tropically realizable oriented matroids.

\paragraph{Complications arising from realization spaces.}

Every combinatorial object that records geometric properties has an associated {\em realization space}, namely the set of all these geometric configurations giving rise to that combinatorial object. An important observation is that any given geometric realization may only give access to a fraction of the extensions of the combinatorial structure. For instance, adding one element to the labeled order type of $n$ points in convex position in the plane can produce $\Omega\pth{2^n/n}$ distinct labeled order types~\cite{pilz2020crossing} (the so-called \emph{wheelsets}); in contrast, adding one point to a fixed set of $n$ points in convex position can only produce point sequences with $O(n^4)$ different labeled order types (see Appendix~\ref{a:mistake-GP}). For tropically realizable oriented matroids, this problem arises as well.

A landmark result in realizable oriented matroid theory is Mnëv's universality theorem~\cite{Mnev1988Universality}, which asserts that already in the plane the set of point sequences that give rise to a given oriented matroid may be arbitrarily complicated. In computational term, Mnëv's arguments position the {\sc Membership problem} of deciding whether a given oriented matroid is realizable as an $\exists \R$-complete problem. In particular, it is unknown whether this  {\sc Membership problem} is in NP.\footnote{This may appear surprising at first glance: why not certify membership by providing a sequence of points realizing it? The catch is that although every uniform realizable oriented matroid is realizable with rational, and therefore integer, coordinates, some require Cartesian coordinates of order $2^{2^{\Omega(n)}}$ where $n$ is the number of points~\cite{GoodmanPollackSturmfels1990Spread}.} Note that a sampling algorithm, which turns a random bitstring into a random realizable oriented matroid, is akin to a verifier (with the bitstring acting as certificate); there is therefore no known algorithm to produce in time polynomial in $n$ a random realizable oriented matroid of size $n$ with full support (meaning that every realizable oriented matroid of size $n$ has nonzero probabilty), regardless of the bias!
\footnote{In particular, reading off the oriented matroid of a random point set from a standard measure (Gaussian or uniform on a square or disk) is known to produce asymptotically singular distributions at the level of oriented matroids~\cite[Theorem~1.1]{goaoc2023convex}.}

Regarding realization spaces, the tropical picture stands in sharp contrast. 
In particular, in the tropical case, the set of tropical realizations of a given oriented matroid
has a semi-linear structure. \cite{brandt2018incidence} provides a tropical variant of the universality theorem, in which semi-linear sets play the role of semi-algebraic sets.

\paragraph{Techniques for counting labeled order types.}

When it comes to determining the asymptotic number of real labeled order types, the techniques used are essentially of two types.

{\em Upper bounds} were obtained by recasting the space of $n$-point sequences in $\R^d$ to $\R^{dn}$, partitioning that space by the hypersurfaces corresponding to configurations where some $d+1$ points lie on a common hyperplane, and bounding from above the number of chambers. This last step can be achieved, with some effort~\cite{alon1986number}, using the Oleinik-Petrovsky-Milnor-Thom bound on the Betti numbers of real algebraic varieties~\cite{Milnor1964} or more directly using Warren's topological counting argument~\cite[see ``added in print'']{GoodmanPollack1986}; see also~\cite[$\mathsection 5$]{Alon1995}. This approach readily generalizes to any discrete structure defined by semi-algebraic relations, and plays for instance a role in sharpening the bounds on the Zarankiewicz problem for semi-algebraic graphs~\cite{fox2017semi}. For our purpose, these techniques give, for $r \ge 4$, worse dependency in $r$ than the upper bounds of Theorems~\ref{thm:main-unsigned} and~\ref{thm:main-signed} (we spell this out in Appendix~\ref{app:upper-bound-comparison}). Our proofs of upper bounds are of a completely different nature: we leverage the connections between tropically realizable oriented matroids, coherent matching fields, and integer points of certain transportation polytopes. 
 
{\em Lower bounds} were established by constructing many point configurations with pairwise distinct oriented matroids. This is made possible by observing that given a point sequence $P=(p^1, p^2, \ldots, p^n)$, the oriented matroid of an extension $(p^1, p^2, \ldots, p^n, p^{n+1})$ is determined by the cell that contains $p^{n+1}$ in the arrangement of the $\binom{n}{d+1}$ hyperplanes spanned by $P$. For points of $\R^d$ in general position, the number of such cells does not depend on the position of the points, Goodman and Pollack~\cite{GoodmanPollack1986} could analyze the inductive process of repeatedly adding one point at a time and prove that the upper bound obtained from Warren's theorem is tight at first order.\footnote{We note that the precise formula given by Goodman and Pollack for the number of cells in the arrangement induced by $n$ points in $\R^d$ evaluates to a negative value for $n=12$ and $d=10$. The asymptotic order of magnitude of this number of cells is, however, correct so only lower order terms are affected. See Appendix~\ref{a:mistake-GP}.} An alternative approach is to ``batch'' the insertions: argue that for arbitrary large number of $t$ there exist {\em one} configuration of size $t$ with many cells and then augments its by $n-t$ points; this was used by Sauermann~\cite{Sauermann2021Speed} to prove that Warren's bound is also tight at first order for generic semi-algebraic graphs. Scheucher~\cite{SCHEUCHER2023101924} also used this approach to obtain lower bounds on the number of labeled order types realizable on integer grids of polynomial size. In the tropical setting, the number of cells in the arrangement of the hyperplanes spanned by $n$ generic points depends on the position of the points, so that it is difficult to control a sequence of induction steps adding $1$ point each. We therefore follow the ``batched'' approach, and note that the construction of an initial configuration with many cells proves more challenging in the tropical setting.

\paragraph{Tropical geometry.}

The present work should be put in the broader perspective of real
tropicalization.  The tropical labeled order types considered here
constitute special cases of the matroids over
hyperfields studied by Baker and Bowler~\cite{BakerBowler2019} as an
extension of the valuated matroids introduced by Dress and
Wentzell~\cite{DressWenzel1992}, they are also special cases of the
matroids over systems considered
in~\cite[\S~8]{AkianGaubertRowen2024}.

The discrepancy between the classical and tropical labeled
order types is reminiscent of phenomena observed in Viro's
patchworking method~\cite{Viro1983,Viro2001,Viro2006}. This method
allows one in particular to construct multivariate polynomial systems
with a number of positive solutions controlled by combinatorial
means. Bihan~\cite{Bihan2002} showed that for any system built by
patchworking (transverse tropical intersections), the number of
positive solutions cannot exceed a bound conjectured by
Kouchnirenko. Haas's counterexample~\cite{Haas2002} to Kouchnirenko's
conjecture entails that there are counts of solutions that cannot be
attained by combinatorial patchworking.  El Hilany showed
in~\cite{ElHilany2018} that the transversality condition is
responsible for the loss of solutions. The genericity assumption used
in our definition of tropical chirotopes plays a role similar to
transversality in combinatorial patchworking, and by analogy, one
would expect a gap between the number of ordinary chirotopes and the
number of tropical chirotopes.

\section{The upper bounds of Theorem~\ref{thm:main-unsigned}}\label{sec:upper-unsigned}

In this section we prove our upper bounds on $\# \CMFT(n,r)$ and $\# \CTT(n,d)$. We first prove the bound for $\# \CMFT(n,r)$ then propagate it to $\# \CTT(n,d)$. The analysis for $\# \CMFT(n,r)$ relies on a correspondence between coherent matching fields and Minkowski sums of Birkhoff polytopes. The propagation leverages a bijection due to Fink and Rincón.

\subsection{Minkowski sums of Birkhoff polytopes}

Let $r \ge 2$. Let $I \in \binom{[n]}{r}$. For every matching $M\in\mathcal M_I$ we let $E(M)\in\{0,1\}^{n\times r}$ denote its incidence matrix, that is $E(M)_{i,k}=1$ if and only if $(i,k)\in M$. Now, to any coherent matching field $\Lambda = (M_I)_{I \in \binom{[n]}{r}} \in \CMFT(n,r)$ we associate a point $v(\Lambda) = \sum_{I\in\binom{[n]}{r}} E(M_I) \in {\N_0}^{n\times r}$.

\begin{lemma}\label{lem:cmf-vertices}
  The map $\Lambda \in \CMFT(n,r) \mapsto v(\Lambda) \in (\N_0)^{n\times r}$ is injective.
\end{lemma}
\begin{proof}
  For $I \in \binom{[n]}{r}$ let $B_I \;\coloneqq\;
  \operatorname{conv}\{E(M) \colon M \in \mathcal M_I\}$ denote the
  copy of the Birkhoff polytope $B_r \subset \R^{r \times r}$ in the
  subspace $\R^{I \times r}$ of $\R^{n \times r}$. For a fixed $I$,
  maximizing over matchings is the same as optimizing a linear
  functional over $B_I$:
  \[
    \max_{M\in\mathcal M_I}\sum_{(i,k)\in M}W_{i,k}
    =
    \max_{v\in\mathrm{Vert}(B_I)}\langle W,v\rangle .
  \]
  Hence, if $W$ realizes $\Lambda=(M_I)_{I\in\binom{[n]}{r}}$, then $E(M_I)$ is the unique maximizer on $B_I$ for every $I$. Now suppose that $v= \sum_{I\in\binom{[n]}{r}} E(M_I) = \sum_{I\in\binom{[n]}{r}} E(M_I')$ for two distinct coherent matching fields $(M_I)_{I \in \binom{[n]}{r}}$ and $(M_I')_{I \in \binom{[n]}{r}}$. Then $\langle W,E(M'_I)\rangle$ is at most the maximum of $W$ on $B_I$ for every $I$. Equality after summing forces equality for every $I$, and hence  $M'_I=M_I$ by uniqueness of the maximizer on $B_I$. The map is therefore injective.
\end{proof}

\begin{remark}
  The map $v$ actually defines a bijection between $\CMFT(n,r)$ and the vertices of the Minkowski sum \(P_{n,r} \coloneqq \sum_{I\in\binom{[n]}{r}} B_I
  \subset\R^{n\times r}\) of all these copies of $B_r$. This is a special case of the bijection between the extreme points of a polytope and the maximal cones of its normal fan. See Sturmfels and Zelevinsky~\cite{SturmfelsZelevinski93}.
\end{remark}

We can now prove our first upper bound:

\begin{proposition}\label{prop:cmf-upper-bound}
  For $n \ge r \ge 2$, we have $\# \CMFT(n,r) \le \binom{\binom{n-1}{r-1}+r-1}{r-1}^{\!n}$.
\end{proposition}

\noindent
The proof is elementary, bounding the number of coherent matching fields by
the number of integer points of a transportation polytope gives a bound which, surprisingly, is tight enough
for our purposes.
\begin{proof}
  Let $\Lambda \in \CMFT(n,r)$. Observe that $v(\Lambda)$ is the count matrix where for all $(j,k) \in [n] \times [r]$,
\[
v(\Lambda)_{j,k} \;=\;  \#\left\{I \in \binom{[n]}{r} \colon (j,k) \in M_I\right\},\]
with fixed row and column sums
\[
\text{for all } j\in[n], \quad \sum_{k\in[r]} v(\Lambda)_{j,k} \;=\; \binom{n-1}{r-1},
  \qquad \text{and for all } k\in[r], \quad
  \sum_{j\in[n]} v(\Lambda)_{j,k} \;=\; \binom{n}{r}.
\]
(Each $r$-subset containing $j$ contributes one edge to row $j$, and
each $r$-subset contributes one edge to column $k$.) Each row has
nonnegative entries with sum $\binom{n-1}{r-1}$, so it has at most
\[
  \binom{\binom{n-1}{r-1}+r-1}{r-1}
\]
possibilities, the number of compositions of $\binom{n-1}{r-1}$ into $r$ ordered nonnegative parts. Multiplying over all rows and using \cref{lem:cmf-vertices} gives
\[  \card{\CMFT(n,r)}
  \;\leq\;
  \binom{\binom{n-1}{r-1}+r-1}{r-1}^{\!n}\]
as announced.
\end{proof}

Here is a direct translation of the upper bound of Theorem~\ref{thm:main-unsigned}~(i):

\begin{corollary}\label{cor:cmf-log-upper-bound}
  As $n \to \infty$, for every fixed $r\geq 2$,
  \[ \log_2 \card{\CMFT(n,r)} \leq (r-1)^2\,n\log_2 n - rn\,\log_2((r-1)!) + \Oh(1).\]
\end{corollary}
\begin{proof}
  Let us write $N_0\coloneqq\binom{n-1}{r-1}$. Statement~(i) is merely a reformulation of Proposition~\ref{prop:cmf-upper-bound}, which asserts that
  \[
    \log_2 \card{\CMFT(n,r)}
    \;\leq\;
    n\log_2\binom{N_0+r-1}{r-1}.
  \]
  Since $r$ is fixed,
  \[
    \log_2\binom{N_0+r-1}{r-1}
    = (r-1)\log_2 N_0-\log_2((r-1)!)+\Oh(N_0^{-1}),
  \]
  while
  \[
    \log_2 N_0
    = (r-1)\log_2 n-\log_2((r-1)!)-\Oh(1/n).
  \]
  Multiplying by $n$ gives the claimed bound; here
  $nN_0^{-1}=\Oh(1)$ for every fixed $r\geq2$.
\end{proof}

\subsection{Pointed coherent matching fields and max-tropical hyperplanes in general position} \label{sec:pointed-cmf}

To propagate the upper bound to $\CTT(n,d)$, some terminology is in order.

\paragraph{Sectors and combinatorial types.}

Let \(u\in\TPfin^d\simeq\R^d\), and use homogeneous coordinates for \(u\) and \(x\). For each \(i\in[d+1]\), the \(i\)-th closed \emph{sector} of \(\Hmax(u)\) is
\[
  \Smax_i(u)\coloneqq
  \left\{x\in\TPfin^d \colon x_i-u_i\geq x_j-u_j
  \text{ for all }j\in[d+1]\setminus\{i\}\right\}.
\]
In \cref{fig:intro-hyperplane-sequence}(a), \(\Smax_i(u)\) is the region labelled ``sector \(i\)''.

Given a point sequence \(\u=(u^1,\ldots,u^n)\) in \(\TPfin^d\), write
\(
  \Amax(\u)\coloneqq(\Hmax(u^1),\ldots,\Hmax(u^n))
\)
for the associated sequence of max-tropical hyperplanes. The type of a point \(x\in\TPfin^d\) with respect to \(\A=\Amax(\u)\)
is the \((d+1)\)-tuple of subsets of \([n]\) given by
\(T^{\A}(x)=(T^{\A}_1(x),\ldots,T^{\A}_{d+1}(x))\), where for each \(j\in[d+1]\),
\[
  T^{\A}_j(x)\coloneqq
  \left\{k\in[n] \colon x_j-u_j^k=\max_{i\in[d+1]}(x_i-u_i^k)\right\}.
\]
The \emph{labeled combinatorial type} of \(\A\) is the set of types it realizes:
\[
  \operatorname{type}(\A)\coloneqq
  \left\{T^{\A}(x)\colon x\in\TPfin^d\right\}.
\]
Two hyperplane sequences are \emph{combinatorially equivalent} if they have
the same labeled combinatorial type.
Recall that $\CTT(n,d)$ denotes the set of labeled combinatorial
types of sequences of $n$ max-tropical hyperplanes in general
position in $\TPfin^d$. The general-position condition is
defined below.

\paragraph{Generic tropical matrices.}

Recall that $\Sym_q$ denotes the set of permutations of $[q]$. For a square tropical matrix \(A\in\T^{q\times q}\), the \emph{tropical permanent} is $\tper(A)  \coloneqq \bigoplus_{\sigma\in\Sym_q}\bigodot_{i=1}^q A_{i,\sigma(i)} = \max_{\sigma\in\Sym_q}\sum_{i=1}^q A_{i,\sigma(i)}$. We say that \(A\) is \emph{non-singular} if the maximum defining \(\tper(A)\) is finite and attained by exactly one permutation \(\sigma\in\Sym_q\). A matrix \(U\in\T^{n\times r}\) with \(n\geq r\) is \emph{generic} if every maximal square submatrix, that is, every \(r\times r\) submatrix, is non-singular. As an example, the \emph{tropical identity matrix} \(I_r\in\T^{r\times r}\), with \(\tone\) on the diagonal and \(\tzero\) off it, is non-singular, since the maximum defining \(\tper(I_r)\) is attained only by the identity permutation.

\paragraph{Pointed augmentation.}

The \emph{pointed augmentation} of a matrix \(B\in\T^{n\times r}\) is the matrix $\widehat B \in\T^{(n+r)\times r}$ defined by
\[
  \widehat B\coloneqq
  \begin{pmatrix} I_r\\ B\end{pmatrix}.
\]
Note that \(\widehat B\) is generic if and only if every square submatrix of \(B\) is
non-singular.\footnote{Indeed, for \(K\subset[n]\) and \(J\subset[r]\) of equal cardinality \(j\), write \(B_{K,J}\) for the corresponding \(j\times j\) submatrix of \(B\), and write \(\widehat{B}_{K,J}\) for the \(r\times r\) submatrix of \(\widehat B\) that uses the rows of \(B\) indexed by \(K\) together with the rows of \(I_r\) indexed by \([r]\setminus J\). Each of the latter has its unique finite entry in a column of \([r]\setminus J\), so any permutation with finite permanent must match that row there. The remaining columns \(J\) are then matched with the rows indexed by \(K\), so \(\widehat{B}_{K,J}\) is non-singular if and only if \(B_{K,J}\) is.}

\paragraph{Pointed coherent matching fields.}

The propagation of the upper bound to $\CTT(n,d)$ goes through a special kind of coherent matching fields. A coherent matching field $\Lambda\in\CMFT(n,r)$ is \emph{pointed} if  for every $j\in[r]$ and every $I\in\binom{[n]}{r}$ with $j\in I$, the diagonal edge $(j,j)$ lies in $M_I$. Equivalently, $\Lambda$ is pointed if and only if it can be realized by a matrix whose top $r$ rows form the tropical identity $I_r$, that is, if it can be realized by the pointed augmentation matrix $\widehat B$ of some $B\in\T^{(n-r)\times r}$. We write $\CMFTbullet(n,r)$ for the set of pointed coherent matching fields of type $(n,r)$.
 
The following lemma allows us to transfer counting bounds between ordinary
and pointed coherent matching fields at the cost of adding only \(r\) rows.

\begin{lemma}[Inclusion chain]\label{lem:cmf-inclusion}
  For all $n\geq r\geq 1$,
  \[
    \card{\CMFTbullet(n,r)}
    \;\leq\;
    \card{\CMFT(n,r)}
    \;\leq\;
    \card{\CMFTbullet(n+r,r)}.
  \]
\end{lemma}
\begin{proof}
  The first inequality is immediate from $\CMFTbullet(n,r)\subset\CMFT(n,r)$.
  For the second, let \(\Lambda\in\CMFT(n,r)\) and choose a matrix \(B\)
  realizing it. Then every \(r\times r\) submatrix of \(B\) is non-singular,
  but smaller square submatrices may be singular. Consequently, \(\widehat B\)
  need not be generic and need not realize a pointed coherent matching field.
  Since the realization conditions are finitely many strict inequalities, we
  may choose \(B\), after a sufficiently small generic perturbation, so that it
  still realizes \(\Lambda\) and every square submatrix of \(B\) is
  non-singular. Then \(\widehat B\) is generic and realizes a pointed coherent
  matching field
  \(\widehat\Lambda\in\CMFTbullet(n+r,r)\). Fix one such choice for each
  \(\Lambda\). Restricting \(\widehat\Lambda\) to the last \(n\) rows and
  shifting their indices by \(i\mapsto i-r\) recovers \(\Lambda\); hence
  \(\Lambda\mapsto\widehat\Lambda\) is injective.
\end{proof}

\paragraph{General position for points and max-tropical hyperplanes.}

Let $\u = (u^1,\ldots,u^n)$ be a point sequence in $\TPfin^d$, and put $r\coloneqq d+1$.
A matrix $U \in \R^{n\times r}$ is a \emph{matrix of} $\u$ if, for every $i\in [n]$, the $i$th row of $U$ represents $u^i$.
We say that $\u$ is \emph{in max-general position} if $\widehat{-U}$ is generic, where $U$ is a matrix of $\u$.
This condition depends neither on the choice of row representatives nor on the order of the points.
We say that a sequence of max-tropical hyperplanes is in \emph{general position} if the sequence of its apices is in max-general position.

\paragraph{Wrapping up.}

We can now complete the proof of the upper bounds of Theorem~\ref{thm:main-unsigned}:

\begin{corollary}\label{cor:unsigned-upper-bound}
  As $n \to \infty$, for every fixed $d\geq 1$,
  \[ \log_2 \card{\CTT(n,d)} \le d^2\,n\log_2 n - (d+1)n\ \log_2(d!) + \Oh(\log n).\]
\end{corollary}
\begin{proof}
  Fink and Rinc\'on~\cite[Theorem~4.4 and Example~4.8]{FinkRincon2015} constructed a bijection between the combinatorial types of $n$ max-tropical hyperplane arrangements in general position in $\TPfin^{r-1}$ and the pointed coherent matching fields of type $(n+r,r)$. Together with the Develin--Sturmfels correspondence~\cite{DevelinSturmfels2004}, this gives
  \begin{equation}\label{eq:ct-rt-cmf}
    \#\RT(n,r) = \#\CTT(n,r-1) = \#\CMFTbullet(n+r,r).
  \end{equation}
  Set $r=d+1$. By Lemma~\ref{lem:cmf-inclusion},
  $\#\CTT(n,d) \le \#\CMFT(n+d+1,d+1)$.
  Applying Corollary~\ref{cor:cmf-log-upper-bound} with $n+d+1$ in place of $n$
  gives the claimed bound, since $d$ is fixed and this shift changes
  the right-hand side by $\Oh(\log n)$.
\end{proof}

\section{The lower bounds of Theorem~\ref{thm:main-unsigned}}\label{sec:lower-unsigned}

In this section we prove our lower bounds on $\# \CMFT(n,r)$ and $\# \CTT(n,d)$. This time, we first prove the bound for $\# \CTT(n,d)$ then propagate it to $\# \CMFT(n,r)$, again via the correspondence between combinatorial types of tropical hyperplane arrangement and pointed coherent matching field established by Fink and Rincón.

\medskip

The analysis for $\# \CTT(n,d)$ has two main steps. The first is to construct an initial arrangement of $t$~tropical hyperplanes ($t \leq n$) that admits many 1-hyperplane extensions of distinct combinatorial types. We count these extensions using an \emph{auxiliary arrangement}. While in the usual affine setting any choice of hyperplanes in general position would be suitable, in the tropical setting a specific construction is required.\footnote{For the same reason, the argument of Goodman--Pollack~\cite[Remark~5]{GoodmanPollack1986} for the affine setting is ill-suited here.} The second step extends this initial arrangement by $n-t$ further tropical hyperplanes. The key point is that two sequences of 1-hyperplane extensions yield arrangements of distinct combinatorial types if they differ in at least one $1$-extension of the initial $t$ hyperplanes. We postpone the fine-tuning of $t$ to the last stage of the proof.

\subsection{Preparation: some background on tropical hyperplanes}\label{sec:unsigned-geometry}

Our proof of the lower bounds requires some additional background on tropical geometry: min-tropical hyperplane and the min/max duality, a better geometric description of tropical hyperplanes, and the tropical Cramer rule and its use for defining the tropical hyperplane spanned by $d$ points in $\TPfin^d$.

\paragraph{Min-tropical hyperplanes and sectors.}

Just like the max-tropical hyperplanes defined by~\cref{eq:max-tropical-hyperplane}, the \emph{min-tropical hyperplane} with apex $u\in\TPfin^{d}$ is defined by
\[\Hmin(u)\;\coloneqq\;  \Bigl\{x \in \TPfin^{d} \colon \min_{i\in[d+1]}(x_i-u_i) \text{ is attained at least twice}\Bigr\}.\]
The \(i\)-th closed \emph{sector} of \(\Hmin(u)\) in \(\TPfin^{d}\) is
\[\Smin_i(u)\coloneqq\{x \colon x_i - u_i \leq x_j - u_j \text{ for all } j \neq i\}.\]
this is similar to the definition of sectors of \(\Hmax(u)\) in Section~\ref{sec:pointed-cmf}, with the inequalities reversed.

We say that a point sequence \(\u=(u^1,\ldots,u^n)\) is in
\emph{min-general position} if \(\widehat U\) is generic, where
\(U\) is a matrix of \(\u\).
Equivalently, \(-\u\coloneqq(-u^1,\ldots,-u^n)\) is in max-general position.
We say that a sequence of min-tropical hyperplanes is in \emph{general position} if the sequence of its apices is in min-general position.

\paragraph{Facets.}

A max/min-tropical hyperplane in $\TPfin^{d}$ with apex $u$ is a polyhedral complex with a unique vertex $u$ and $\binom{d+1}{2}$ facets, one for each pair in \(\binom{[d+1]}{2}\). Specifically, for $\{i,j\} \in \binom{[d+1]}{2}$ we put
\begin{align}\label{eq:tropical-facet}
  \Fmax_{ij}(u)\coloneqq
  \{x \colon x_i - u_i = x_j - u_j \geq x_k - u_k \text{ for all } k \neq i,j\}
\end{align}
for the facet of $\Hmax(u)$ where the maximum is attained in the $i$th and
$j$th coordinates. Equivalently,
\(\Fmax_{ij}(u) = \Smax_i(u) \cap \Smax_j(u)\). The facets
\(\Fmin_{ij}(u)\) of $\Hmin(u)$ are defined similarly
to~\eqref{eq:tropical-facet}, with the inequalities reversed.
\Cref{fig:hyperplanes} illustrates these hyperplanes, their sectors, and their facets in
\(\TPfin^{2}\).

\begin{figure}[t]
  \centering
  \begin{tikzpicture}[
    x=0.82cm, y=0.82cm,
    >=latex, line cap=round, line join=round, font=\small,
    maxB/.style={RoyalBlue, line width=1pt},
    minline/.style={ForestGreen!70!black, line width=1pt},
    apex/.style={circle, fill, inner sep=0pt, minimum size=3pt},
    lab/.style={font=\footnotesize, fill=white, inner sep=1pt, outer sep=1pt},
    slabel/.style={font=\scriptsize\itshape, text=black!70, inner sep=0.5pt},
    flabel/.style={font=\scriptsize, inner sep=0.5pt, outer sep=1pt},
    panellab/.style={font=\itshape, anchor=south},
    sA/.style={fill=RoyalBlue!7, draw=none},
    sB/.style={fill=ForestGreen!9, draw=none},
    sC/.style={fill=BrickRed!5, draw=none},
  ]

  \begin{scope}[shift={(0,0)}]
    \node[panellab] at (0,3.0) {(a) \(\Hmax(u)\)};
    \begin{scope}
      \clip (-2.4,-2.4) rectangle (2.4,2.4);
      \fill[sA] (0,0) -- (0,-2.4) -- (2.4,-2.4) -- (2.4,2.4) -- cycle;
      \fill[sB] (0,0) -- (2.4,2.4) -- (-2.4,2.4) -- (-2.4,0) -- cycle;
      \fill[sC] (0,0) -- (-2.4,0) -- (-2.4,-2.4) -- (0,-2.4) -- cycle;
      \draw[maxB] (0,0) -- (-2.4,0);
      \draw[maxB] (0,0) -- (0,-2.4);
      \draw[maxB] (0,0) -- (2.4,2.4);
    \end{scope}
    \node[apex, RoyalBlue] at (0,0) {};
    \node[lab, text=RoyalBlue, anchor=north east] at (-0.04,-0.04) {$u$};
    \node[slabel] at (1.45,-1.3) {$\Smax_1(u)$};
    \node[slabel] at (-0.95,1.55) {$\Smax_2(u)$};
    \node[slabel] at (-1.3,-1.3) {$\Smax_3(u)$};
    \node[flabel, text=RoyalBlue, anchor=south] at (-1.30,0.12) {$\Fmax_{23}(u)$};
    \node[flabel, text=RoyalBlue, anchor=west] at (0.12,-1.95) {$\Fmax_{13}(u)$};
    \node[flabel, text=RoyalBlue, rotate=45, anchor=south] at (1.35,1.45) {$\Fmax_{12}(u)$};
  \end{scope}

  \begin{scope}[shift={(6.6,0)}]
    \node[panellab] at (0,3.0) {(b) \(\Hmin(u)\)};
    \begin{scope}
      \clip (-2.4,-2.4) rectangle (2.4,2.4);
      \fill[sA] (0,0) -- (0,2.4) -- (-2.4,2.4) -- (-2.4,-2.4) -- cycle;
      \fill[sB] (0,0) -- (-2.4,-2.4) -- (2.4,-2.4) -- (2.4,0) -- cycle;
      \fill[sC] (0,0) -- (2.4,0) -- (2.4,2.4) -- (0,2.4) -- cycle;
      \draw[minline] (0,0) -- (2.4,0);
      \draw[minline] (0,0) -- (0,2.4);
      \draw[minline] (0,0) -- (-2.4,-2.4);
    \end{scope}
    \node[apex, ForestGreen!70!black] at (0,0) {};
    \node[lab, text=ForestGreen!55!black, anchor=south west] at (0.05,0.06) {$u$};
    \node[slabel] at (-1.4,0.75) {$\Smin_1(u)$};
    \node[slabel] at (0.4,-1.55) {$\Smin_2(u)$};
    \node[slabel] at (1.3,1.3) {$\Smin_3(u)$};
    \node[flabel, text=ForestGreen!70!black, anchor=north] at (1.30,-0.12) {$\Fmin_{23}(u)$};
    \node[flabel, text=ForestGreen!70!black, anchor=east] at (-0.12,1.95) {$\Fmin_{13}(u)$};
    \node[flabel, text=ForestGreen!70!black, rotate=45, anchor=north] at (-1.35,-1.45) {$\Fmin_{12}(u)$};
  \end{scope}
\end{tikzpicture}
  \caption{A max-tropical hyperplane \(\Hmax(u)\) and a min-tropical hyperplane
  \(\Hmin(u)\), each with its three sectors and three facets, in the chart \(x_3=0\).}
  \label{fig:hyperplanes}
\end{figure}
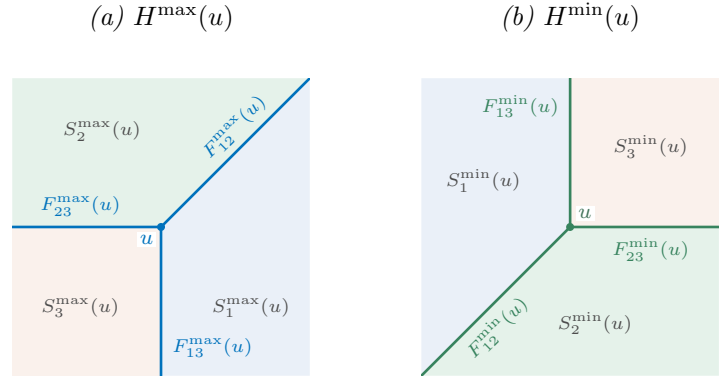

\paragraph{Tropical point-hyperplane duality.}

There is a duality between tropical points and tropical hyperplanes that preserves incidences similarly to the  classical point-hyperplane duality:
\begin{eqnarray}\label{eq:max-min-incidence-duality}
  x\in \Hmax(u) & \quad\Longleftrightarrow\quad & u\in \Hmin(x)\\
  x\in \Smax_i(u) & \quad\Longleftrightarrow\quad & u\in \Smin_i(x)
\end{eqnarray}
Consequently, for apices \(u^1,\ldots,u^{d}\) and a point
\(x\in\TPfin^{d}\) (see \cref{fig:incidence-duality}),
\[
  x\in \Hmax(u^1)\cap\cdots\cap \Hmax(u^{d})
  \quad\Longleftrightarrow\quad
  \{u^1, \ldots, u^d\} \subset \Hmin(x).
\]

\begin{figure}[t]
  \centering
  \begin{tikzpicture}[
    x=0.82cm, y=0.82cm,
    >=latex, line cap=round, line join=round, font=\small,
    maxB/.style={RoyalBlue, line width=1pt},
    maxR/.style={BrickRed!85!black, line width=0.95pt},
    minline/.style={ForestGreen!70!black, line width=1pt},
    apexB/.style={circle, fill=RoyalBlue, draw=black, line width=0.25pt,
      inner sep=0pt, minimum size=3.3pt},
    apexR/.style={circle, fill=BrickRed!85!black, draw=black, line width=0.25pt,
      inner sep=0pt, minimum size=3.3pt},
    pt/.style={circle, fill=black!85, inner sep=0pt, minimum size=3.2pt},
    lab/.style={font=\footnotesize, fill=white, inner sep=1pt, outer sep=1pt},
    tiny/.style={font=\scriptsize, fill=white, inner sep=1pt, outer sep=1pt},
    panellab/.style={font=\itshape, anchor=south},
  ]

  \begin{scope}[shift={(0,0)}]
    \node[panellab] at (0.45,3.0) {(a) \(x\in\Hmax(u^k)\)};
    \begin{scope}
      \clip (-2.3,-2.35) rectangle (3.2,2.6);
      \draw[maxB] (0,1) -- (-2.3,1);
      \draw[maxB] (0,1) -- (0,-2.35);
      \draw[maxB] (0,1) -- (1.6,2.6);
      \draw[maxR] (1,0) -- (-2.3,0);
      \draw[maxR] (1,0) -- (1,-2.35);
      \draw[maxR] (1,0) -- (3.2,2.2);
    \end{scope}
    \node[apexB] at (0,1) {};
    \node[lab, text=RoyalBlue, anchor=south east] at (-0.05,1.05) {$u^1$};
    \node[apexR] at (1,0) {};
    \node[lab, text=BrickRed!85!black, anchor=north west] at (1.05,-0.05) {$u^2$};
    \node[pt] at (0,0) {};
    \node[lab, anchor=north east] at (-0.05,-0.05) {$x$};
  \end{scope}

  \begin{scope}[shift={(6.6,0)}]
    \node[panellab] at (0.45,3.0) {(b) \(u^k\in\Hmin(x)\)};
    \begin{scope}
      \clip (-2.3,-2.35) rectangle (3.2,2.6);
      \draw[minline] (0,0) -- (3.2,0);
      \draw[minline] (0,0) -- (0,2.6);
      \draw[minline] (0,0) -- (-2.3,-2.3);
    \end{scope}
    \node[pt] at (0,0) {};
    \node[lab, anchor=north east] at (-0.05,-0.05) {$x$};
    \node[apexB] at (0,1) {};
    \node[lab, text=RoyalBlue, anchor=south east] at (-0.05,1.05) {$u^1$};
    \node[apexR] at (1,0) {};
    \node[lab, text=BrickRed!85!black, anchor=north west] at (1.05,-0.05) {$u^2$};
    \node[tiny, text=ForestGreen!70!black, anchor=south west] at (0.18,1.55)
      {$\Hmin(x)$};
  \end{scope}
\end{tikzpicture}
  \caption{The point--hyperplane incidence duality~\eqref{eq:max-min-incidence-duality}
  in the chart \(x_3=0\): \emph{(a)} a point \(x\) on the max-tropical hyperplanes
  with apices \(u^1,u^2\); \emph{(b)} equivalently, \(u^1,u^2\) on the min-tropical
  hyperplane \(\Hmin(x)\).}
  \label{fig:incidence-duality}
\end{figure}
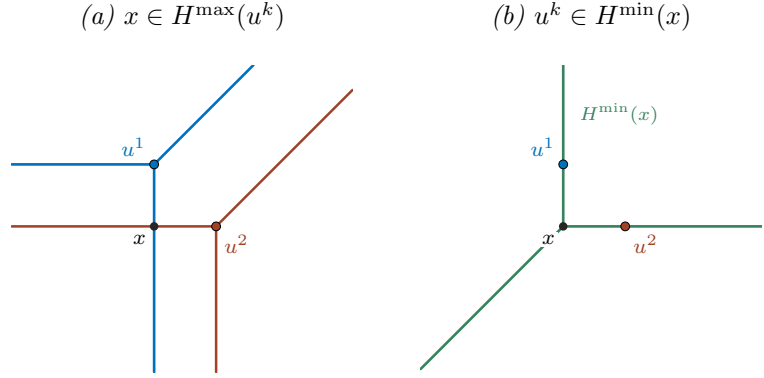

\paragraph{Tropical Cramer rule.}

Let \(A\in\R^{d\times(d+1)}\), and let \(A_{\widehat j}\) be
the square matrix obtained by deleting the $j$th column.
The \emph{Cramer vector} \(\operatorname{cr}(A)\in\R^{d+1}\)
is defined by
\[
  [\operatorname{cr}(A)]_j
  \coloneqq \tper(A_{\widehat j})
  =
  \max_{\sigma \colon [d]\xrightarrow{\sim}[d+1]\setminus\{j\}}
  \sum_{i=1}^{d} A_{i,\sigma(i)},
  \qquad j\in[d+1] \, ,
\]
where $\sigma \colon [d]\xrightarrow{\sim}[d+1]\setminus\{j\}$ means that $\sigma$ is a bijective map from $[d]$ to $[d+1]\setminus\{j\}$.

The tropical Cramer rule
(see, e.g.,~\cite[Cor.~5.4]{RichterGebertSturmfelsTheobald2003}
and~\cite[Th.~5.11]{AkianGaubertGuterman2014})
states that $\operatorname{cr}(A)$ solves the system
\[ \left\{\max_{j\in[d+1]}(A_{i,j}+y_j)  \quad\text{is attained at least twice}\right\}_{i\in[d]}\]
with indeterminates \(y\in\R^{d+1}\). 
When \(A\) is generic, \(\operatorname{cr}(A)\) is the unique
solution up to adding the same constant to all coordinates. Adding a constant to a row of a generic matrix \(A\) adds the same constant to every coordinate of \(\operatorname{cr}(A)\). Hence, if we interpret the rows of $A$ as the coordinates of a point sequence \(\u=(u^1,\ldots,u^d)\) in \(\TPfin^d\), then the projective class of \(\operatorname{cr}(A)\) only depends on that point sequence. We call this projective class of \(\operatorname{cr}(A)\), that is this point in \(\TPfin^d\), the \emph{Cramer point} of $\u$ and denote it by \(\operatorname{cr}(\u)\in\TPfin^d\). The tropical Cramer rule therefore implies that \(d\) tropical hyperplanes in general position intersect in a single point:

\begin{lemma}\label{lem:uniqueCramer}
Let \(\u=(u^1,\ldots,u^d)\) be a point sequence in \(\TPfin^d\).
\begin{itemize}
  \item[(i)] If \(\u\) is in max-general position, then $\Hmax(u^1)\cap\ldots\cap\Hmax(u^d) =\{\operatorname{cr}(-\u)\}$.
  \item[(ii)] If \(\u\) is in min-general position, then $\Hmin(u^1)\cap\ldots\cap\Hmin(u^d) =\{-\operatorname{cr}(\u)\}$.
\end{itemize}
\end{lemma}

\begin{proof}
For (i), let \(U\in\R^{d\times(d+1)}\)
be a matrix whose rows represent \(u^1,\ldots,u^d\). For each \(i\in[d]\),
\[
  x\in\Hmax(u^i)
  \quad\Longleftrightarrow\quad
  \max_{j\in[d+1]}\bigl((-U)_{i,j}+x_j\bigr)
  \text{ is attained at least twice}.
\]
The \(d\) tropical equations on the right have coefficient
matrix \(-U\). Since \(\u\) is in max-general position, the tropical Cramer rule therefore gives
\(\operatorname{cr}(-\u)\) as the unique solution of this system of \(d\) tropical equations in \(\TPfin^d\).

For (ii), \(-\u\) is in max-general position and, for each \(i\in[d]\),
\(x\in\Hmin(u^i)\) if and only if \(-x\in\Hmax(-u^i)\).
Applying (i) to \(-\u\) and taking negatives proves the claim.
\end{proof}

\paragraph{Spanned tropical hyperplane.}

Let \(\u=(u^1,\ldots,u^d)\) be a point sequence in \(\TPfin^d\),
and set \(v=\operatorname{cr}(-\u)\).
By the tropical Cramer rule, \(v\in\Hmax(u^i)\) for every \(i\in[d]\).
The incidence duality~\eqref{eq:max-min-incidence-duality}
then gives \(u^i\in\Hmin(v)\), so \(\Hmin(v)\) contains all
points of \(\u\). If \(\u\) is in max-general position, the preceding lemma gives
\(
  \Hmax(u^1)\cap\ldots\cap\Hmax(u^d)=\{v\}
\).
Thus \(\Hmin(v)\) is the unique
min-tropical hyperplane containing these points; we call it the
\emph{min-tropical hyperplane spanned by} \(\u\).
Similarly, if \(\u\) is in min-general position, the
\emph{max-tropical hyperplane spanned by} \(\u\) is
\(\Hmax({-\operatorname{cr}(\u)})\).

\subsection{The auxiliary arrangement}\label{sec:unsigned-extension-step}

We now turn our attention to the task of
controlling the combinatorial type of an extension.

\paragraph{Auxiliary arrangement.}

Let \(\u=(u^1,\ldots,u^n)\) be a point sequence in max-general position in
\(\TPfin^{r-1}\). We define its \emph{auxiliary arrangement} by
\[
  \mathcal B(\u)
  \coloneqq \pth{\Hmin(x_I) \colon I\in\binom{[n]}{r-1}},
  \qquad
  \{x_I\}:=\bigcap_{i\in I}\Hmax(u^i).
\]
Thus \(\mathcal B(\u)\) is an arrangement of \(\binom{n}{r-1}\)
min-tropical hyperplanes, illustrated in
\cref{fig:auxiliary-arrangement}. Although \(\Amax(\u)\) is in
general position, \(\mathcal B(\u)\) need not be. To use the notion of the type of a point with
respect to a hyperplane sequence, we may fix an arbitrary order on the
\((r-1)\)-element subsets of \([n]\), for example the lexicographic
order, and give \(\mathcal B(\u)\) the induced order.

\begin{figure}[t]
  \centering
  \begin{tikzpicture}[
    x=0.58cm, y=0.58cm,
    line cap=round, line join=round,
    maxold/.style={black!45, dashed, line width=0.65pt},
    minline/.style={ForestGreen!65!black, line width=1.05pt},
    oldpt/.style={circle, fill=black!75, inner sep=0pt, minimum size=3pt},
    centerpt/.style={circle, fill=Goldenrod, draw=black, line width=0.35pt,
      inner sep=0pt, minimum size=3.6pt},
    lab/.style={font=\footnotesize, fill=white, inner sep=1pt, outer sep=1pt},
    panellab/.style={font=\itshape, anchor=south}
  ]

  \begin{scope}
    \node[panellab] at (2.175,5.85) {(a) \(\Amax(\u)\)};
    \begin{scope}
      \clip (-2.0,-0.8) rectangle (6.35,5.7);
      \foreach \x/\y in {0/4,2/2.75,5.75/4.75,4.25/0} {
        \draw[maxold] (\x,\y) -- ++(-20,0);
        \draw[maxold] (\x,\y) -- ++(0,-20);
        \draw[maxold] (\x,\y) -- ++(20,20);
      }
    \end{scope}
    \node[oldpt] at (0,4) {};
    \node[lab, anchor=south east] at (-0.07,4.07) {$u^1$};
    \node[oldpt] at (2,2.75) {};
    \node[lab, anchor=south east] at (1.93,2.82) {$u^2$};
    \node[oldpt] at (5.75,4.75) {};
    \node[lab, anchor=south east] at (5.68,4.82) {$u^3$};
    \node[oldpt] at (4.25,0) {};
    \node[lab, anchor=south east] at (4.18,0.07) {$u^4$};
  \end{scope}

  \begin{scope}[shift={(9.45,0)}]
    \node[panellab] at (2.175,5.85) {(b) \(\mathcal B(\u)\)};
    \begin{scope}
      \clip (-2.0,-0.8) rectangle (6.35,5.7);
      \foreach \x/\y in {0/0,0.75/4.75,2/0,4/4.75,5.75/1.5} {
        \draw[minline] (\x,\y) -- ++(0,20);
      }
      \foreach \x/\y in {0/0,0/2.75,0.75/4.75,5.75/1.5} {
        \draw[minline] (\x,\y) -- ++(20,0);
      }
      \foreach \x/\y in {0/2.75,0.75/4.75,0/0,4/4.75,2/0,5.75/1.5} {
        \draw[minline] (\x,\y) -- ++(-20,-20);
      }
    \end{scope}
    \node[centerpt] at (0,2.75) {};
    \node[lab, anchor=north west] at (0.08,2.67) {$x_{12}$};
    \node[centerpt] at (0.75,4.75) {};
    \node[lab, anchor=north west] at (0.83,4.67) {$x_{13}$};
    \node[centerpt] at (0,0) {};
    \node[lab, anchor=north west] at (0.08,-0.08) {$x_{14}$};
    \node[centerpt] at (4,4.75) {};
    \node[lab, anchor=north west] at (4.08,4.67) {$x_{23}$};
    \node[centerpt] at (2,0) {};
    \node[lab, anchor=north west] at (2.08,-0.08) {$x_{24}$};
    \node[centerpt] at (5.75,1.5) {};
    \node[lab, anchor=south east] at (5.67,1.58) {$x_{34}$};
  \end{scope}
\end{tikzpicture}
  \caption{The arrangement \(\Amax(\u)\) and its
  auxiliary arrangement \(\mathcal B(\u)\) in \(\TPfin^2\).}
  \label{fig:auxiliary-arrangement}
\end{figure}
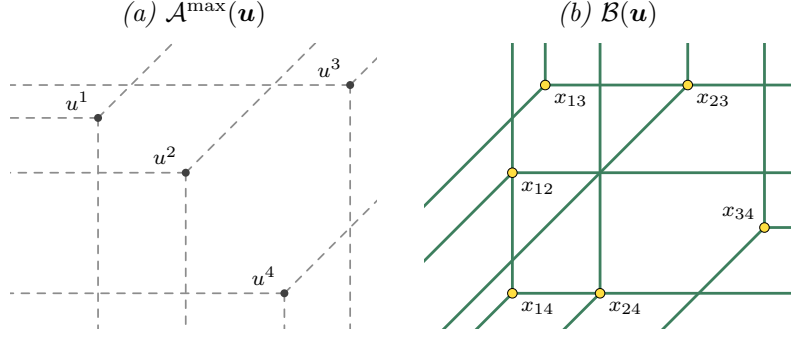

\paragraph{Generic extensions and extension lemma.}

We define a \emph{generic extension} of $\u$ as a sequence of points $\v = (u^1,\ldots,u^n, v^1, \ldots, v^k)$, for some $k \ge 1$, such that $\v$ is in max-general position.

\begin{lemma}\label{lem:generic-extension-existence}
  Let \(\u\) be a (possibly empty) point sequence in max-general position in
  \(\TPfin^d\). For every \(k\geq1\), every nonempty open set
  \(O\subset(\TPfin^d)^k\) contains a \(k\)-tuple giving a generic
  extension of \(\u\).
\end{lemma}
\begin{proof}
  With \(\u\) fixed, the tuples failing to give a generic
  extension lie in a finite union of proper \underline{affine}
  hyperplanes. These finitely many hyperplanes cannot cover \(O\),
  which therefore contains the required \(k\)-tuple.
\end{proof}

\begin{lemma}[Extension lemma]\label{lem:unsigned-extension}
  Let $t \ge r-1$, let $\u = (u^1,\ldots,u^t)$ be a sequence of points in max-general position in $\TPfin^{r-1}$, and let \(\mathcal B\) be its auxiliary arrangement. Let \(\v = (u^1,\ldots,u^t, v^1, \ldots, v^k)\) and \(\w = (u^1,\ldots,u^t, w^1, \ldots, w^k)\) be two generic extensions of $\u$. If the sequences of types $(T^{\mathcal B}(v^1), \ldots, T^{\mathcal B}(v^k))$ and $(T^{\mathcal B}(w^1), \ldots, T^{\mathcal B}(w^k))$ differ, then the combinatorial types of the arrangements $\Amax(\v)$ and $\Amax(\w)$ differ.
\end{lemma}
\begin{proof}
  Set $\A:=\Amax(\u)$, $\A_{\v}:=\Amax(\v)$ and $\A_{\w}:=\Amax(\w)$. Choose an index \(\ell\in[k]\) for which
  \(T^{\mathcal B}(v^\ell)\neq T^{\mathcal B}(w^\ell)\). There is then some
  \(I\in\binom{[t]}{r-1}\) such that \(v^\ell\) and
  \(w^\ell\) lie in different sectors of \(\Hmin(x_I)\in\mathcal B\),
  where $\{x_I\}:=\bigcap_{i\in I}\Hmax(u^i)$. By the point--hyperplane incidence duality
  \eqref{eq:max-min-incidence-duality}, \(x_I\) lies in different sectors of
  \(\Hmax(v^\ell)\) and \(\Hmax(w^\ell)\). These two hyperplanes have the same label \(t+\ell\) in \(\A_{\v}\) and \(\A_{\w}\), respectively, so \(T^{\A_{\v}}(x_I)\) and \(T^{\A_{\w}}(x_I)\) differ at
  label \(t+\ell\).

  We claim that no point \(y\in \TPfin^{r-1}\) has type $T^{\A_{\w}}(y)=T^{\A_{\v}}(x_I)$. Indeed, restricting to the first \(t\) labels would force \(T^{\A}(y)=T^{\A}(x_I)\), which implies that $y$, like $x_I$, lies on $\cap_{i \in I} \Hmax(u^i)$. Since $\u$ is in general position, however, $\cap_{i \in I} \Hmax(u^i)$ is a unique point by \cref{lem:uniqueCramer}, so that \(y=x_I\) and $T^{\A_{\w}}(y)=T^{\A_{\w}}(x_I) \neq T^{\A_{\v}}(x_I)$. The arrangements $\A_{\v}$ and $\A_{\w}$ therefore have different combinatorial types.
\end{proof}

\Cref{fig:apex-induction-step} illustrates the data at the separating label
\(t+\ell\) in the case \(r=3\) and \(I=\{1,2\}\). The apices \(u^1\) and
\(u^2\) span \(\Hmin(x_{12})\), while \(v^\ell\) and \(w^\ell\) lie in different
sectors of this hyperplane. Dually, \(x_{12}\) lies in different sectors of
\(\Hmax(v^\ell)\) and \(\Hmax(w^\ell)\). Only this separating pair
is drawn; the remaining \(k-1\) extension points of \(\v\) and of
\(\w\) are omitted.

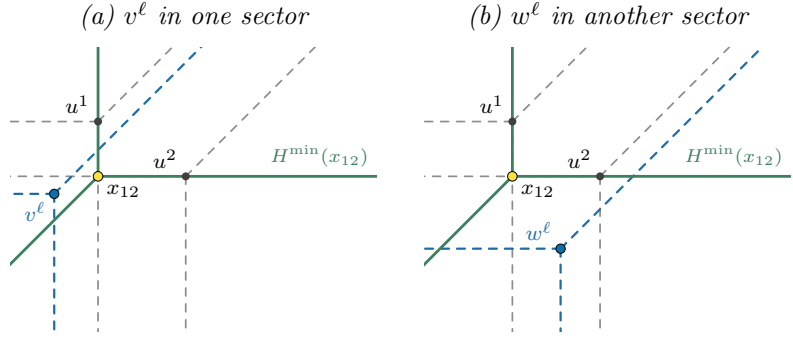
\begin{figure}[t]
  \centering
  \begin{tikzpicture}[
    x=0.58cm, y=0.58cm,
    line cap=round, line join=round,
    maxold/.style={black!45, dashed, line width=0.65pt},
    maxnew/.style={RoyalBlue!80!black, dashed, line width=0.85pt},
    minline/.style={ForestGreen!65!black, line width=1.05pt},
    oldpt/.style={circle, fill=black!75, inner sep=0pt, minimum size=3pt},
    newpt/.style={circle, fill=RoyalBlue!85!black, draw=black, line width=0.25pt,
      inner sep=0pt, minimum size=3.5pt},
    centerpt/.style={circle, fill=Goldenrod, draw=black, line width=0.35pt,
      inner sep=0pt, minimum size=3.6pt},
    lab/.style={font=\footnotesize, fill=white, inner sep=1pt, outer sep=1pt},
    tiny/.style={font=\scriptsize, fill=white, inner sep=1pt, outer sep=1pt},
    panellab/.style={font=\itshape, anchor=south}
  ]

  \newcommand{\drawmaxline}[2]{%
    \draw[#1] (#2) -- ++(-20,0);
    \draw[#1] (#2) -- ++(0,-20);
    \draw[#1] (#2) -- ++(20,20);
  }

  \newcommand{\drawminline}{%
    \draw[minline] (0,2.75) -- ++(20,0);
    \draw[minline] (0,2.75) -- ++(0,20);
    \draw[minline] (0,2.75) -- ++(-20,-20);
  }

  \newcommand{\drawcommon}{%
    \node[centerpt] at (0,2.75) {};
    \node[lab, anchor=north west] at (0.08,2.67) {$x_{12}$};
    \node[oldpt] at (0,4) {};
    \node[lab, anchor=south east] at (-0.07,4.07) {$u^1$};
    \node[oldpt] at (2,2.75) {};
    \node[lab, anchor=south east] at (1.93,2.82) {$u^2$};
    \node[tiny, text=ForestGreen!65!black, anchor=south] at (5.05,2.85)
      {$\Hmin(x_{12})$};
  }

  \begin{scope}
    \node[panellab] at (2.175,5.85) {(a) \(v^\ell\) in one sector};
    \begin{scope}
      \clip (-2.0,-0.8) rectangle (6.35,5.7);
      \drawmaxline{maxold}{0,4}
      \drawmaxline{maxold}{2,2.75}
      \drawmaxline{maxnew}{-1,2.35}
      \drawminline
    \end{scope}
    \drawcommon
    \node[newpt] at (-1,2.35) {};
    \node[lab, text=RoyalBlue!80!black, anchor=north east] at (-1.07,2.28)
      {$v^\ell$};
  \end{scope}

  \begin{scope}[shift={(9.45,0)}]
    \node[panellab] at (2.175,5.85) {(b) \(w^\ell\) in another sector};
    \begin{scope}
      \clip (-2.0,-0.8) rectangle (6.35,5.7);
      \drawmaxline{maxold}{0,4}
      \drawmaxline{maxold}{2,2.75}
      \drawmaxline{maxnew}{1.1,1.1}
      \drawminline
    \end{scope}
    \drawcommon
    \node[newpt] at (1.1,1.1) {};
    \node[lab, text=RoyalBlue!80!black, anchor=south east] at (1.03,1.17)
      {$w^\ell$};
  \end{scope}
\end{tikzpicture}
  \caption{The separating \(\ell\)-th extension points in the proof of
  \cref{lem:unsigned-extension}; the other \(k-1\) extension points are omitted.}
  \label{fig:apex-induction-step}
\end{figure}

\subsection{Point sequences with many extensions}\label{sec:basic-config-unsigned}

We now construct a sequence \(\a\) of points in \(\TPfin^d\) such that \(\mathcal B(\a)\) has many full-dimensional cells. In the affine setting, points in general position suffice (see Appendix~\ref{a:mistake-GP}), whereas the tropical setting requires a specific construction. The key difference lies in the way tropical hyperplanes can overlap. Even when \(\a\) is in max-general position, the hyperplanes of \(\mathcal B(\a)\) can overlap along facets; see \cref{fig:auxiliary-arrangement}.\footnote{Spanned hyperplanes in the affine setting also exhibit forced incidences, but distinct affine lines cannot share a ray.} Already in \(\TPfin^2\), two tropical lines can share a ray and meet a third line at the same point, and such overlap is unavoidable in auxiliary arrangements (\cref{fig:auxiliary-overlap}). This can be viewed as degenerations in which two-dimensional cells collapse onto segments or points.

\begin{figure}[h]
  \centering
  \begin{tikzpicture}[
    x=0.53cm, y=0.53cm,
    line cap=round, line join=round,
    font=\small,
    auxiliary/.style={black!40, dashed, line width=0.65pt},
    minline/.style={ForestGreen!65!black, line width=1.05pt},
    secantB/.style={RoyalBlue!50, line width=0.45pt},
    secantR/.style={BrickRed!55!white, line width=0.45pt},
    apexB/.style={circle, fill=RoyalBlue!85!black, inner sep=0pt, minimum size=2.0pt},
    apexR/.style={circle, fill=BrickRed!85!black, inner sep=0pt, minimum size=2.0pt},
    lab/.style={font=\footnotesize, fill=white, inner sep=1pt, outer sep=1pt},
    intersectionpt/.style={circle, fill=black!75, inner sep=0pt, minimum size=3pt}
  ]
  \begin{scope}
    \begin{scope}
      \clip (-1.4,-1.05) rectangle (7.8,5.7);
      \foreach \x/\y in {0/0,0.75/4.75,2/0,4/4.75,5.75/1.5} {
        \draw[auxiliary] (\x,\y) -- ++(0,20);
      }
      \foreach \x/\y in {0/0,0/2.75,0.75/4.75,5.75/1.5} {
        \draw[auxiliary] (\x,\y) -- ++(20,0);
      }
      \foreach \x/\y in {0/2.75,0.75/4.75,0/0,4/4.75,2/0,5.75/1.5} {
        \draw[auxiliary] (\x,\y) -- ++(-20,-20);
      }
      \foreach \x/\y in {0/0,2/0} {
        \draw[minline] (\x,\y) -- ++(0,20);
        \draw[minline] (\x,\y) -- ++(-20,-20);
      }
      \draw[minline] (0,0) -- ++(20,0);
      \draw[minline] (5.75,1.5) -- ++(0,20);
      \draw[minline] (5.75,1.5) -- ++(20,0);
      \draw[minline] (5.75,1.5) -- ++(-20,-20);
    \end{scope}
    \foreach \x/\y/\i in {0/4/1,2/2.75/2,5.75/4.75/3,4.25/0/4} {
      \node[intersectionpt] at (\x,\y) {};
      \node[lab, anchor=south east]
        at ({\x-0.07},{\y+0.07}) {$u^{\i}$};
    }
  \end{scope}

  \begin{scope}[shift={(10.4,0)}]
    \begin{scope}
      \clip (-1.4,-1.05) rectangle (7.8,5.7);
      \foreach \x/\y in {-0.45/0,0.15/0,0.82/0,1.08/0,2.2/0,2.87/0,3.15/0} {
        \draw[secantB] (\x,\y) -- ++(0,20);
        \draw[secantB] (\x,\y) -- ++(-20,-20);
      }
      \draw[secantB] (-0.45,0) -- ++(20,0);
      \foreach \x/\y in {4.62/0.62,5.16/1.16,6.12/2.12,
                         6.38/2.38,7.24/3.24} {
        \draw[secantR] (\x,\y) -- ++(0,20);
        \draw[secantR] (\x,\y) -- ++(20,0);
      }
      \draw[secantR] (7.24,3.24) -- ++(-20,-20);
    \end{scope}
    \foreach \x/\y in {-0.45/0,0.15/0,0.82/0,1.08/0,2.2/0,2.87/0,3.15/0} {
      \node[apexB] at (\x,\y) {};
    }
    \foreach \x/\y in {4.62/0.62,5.16/1.16,6.12/2.12,
                       6.38/2.38,7.24/3.24} {
      \node[apexR] at (\x,\y) {};
    }
    \node[font=\small,fill=white,inner sep=1pt,
          text=RoyalBlue!85!black] at (1.65,0) {$\cdots$};
    \node[font=\small,fill=white,inner sep=1pt,
          text=BrickRed!85!black,rotate=45] at (5.64,1.64) {$\cdots$};
    \node[intersectionpt] at (4,0) {};
  \end{scope}
\end{tikzpicture}
  \caption{Overlaps of min-tropical lines in \(\TPfin^2\).
  Left: overlaps in the auxiliary arrangement \(\mathcal B(\u)\).
  Right: larger overlapping families.}
  \label{fig:auxiliary-overlap}
\end{figure}
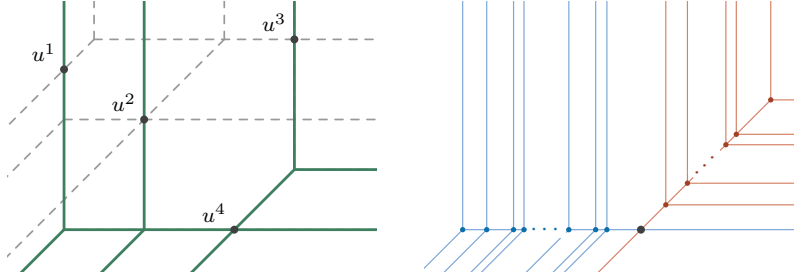

\medskip

Note that in \(\TPfin^{d}\), the points \((x_1,\ldots,x_{d+1})\) and \((y_1,\ldots,y_{d+1})\) are the same if and only if the values $x_i-y_i$ are all equal for $i \in [d+1]$. We can therefore identify \(\TPfin^{d}\) with \(\R^d\) by the chart \(x_{d+1}=0\): the vector \(x\in\R^d\) stands for the class of \((x_1,\ldots,x_d,0)\) in \(\TPfin^{d}\). For every vector \(v\in\R^d\) we write \(v_{d+1}\coloneqq 0\), so that \(v_i\) is defined for all \(i\in[d+1]\). Here is our construction (illustrated in \Cref{fig:lower-cell-count-unsigned} for $d=2$):

\begin{proposition}\label{prop:unsigned-local-product-grid}
  For every \(m\geq1\) there exist \(m\)-element sets \(A_s^j\subset\R^{d}\), for \(j,s\in[d]\), such that:
  \begin{itemize}
  \item[(i)] $\bigcup_{j,s\in[d]} A_s^j$ is a set of \(d^2m\) points in max-general position,
  \item[(ii)] for every \(j\in[d]\) and every \(\p=(p^1,\ldots,p^d)\in A_1^j\times\cdots\times A_d^j\), there exists some \(c_j(\p)\in (-1,1)\) such that the min-tropical hyperplane spanned by \(\p\) intersects $(-1,1)^d$ in $\left\{x\in(-1,1)^d\colon x_j=c_j(\p)\right\}$,
  \item[(iii)] for every \(j\in[d]\), as \(\p\) ranges over \(A_1^j\times\cdots\times A_d^j\), the \(m^d\) numbers \(c_j(\p)\) are pairwise distinct.
  \end{itemize}
\end{proposition}

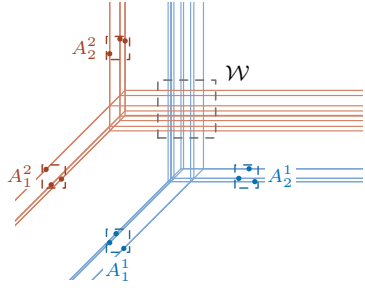
\begin{figure}[t]
  \centering
  \begin{tikzpicture}[
    x=0.85cm, y=0.85cm,
    >=latex,
    line cap=round, line join=round,
    font=\small,
    secantB/.style={RoyalBlue!50,           line width=0.45pt},
    secantR/.style={BrickRed!55!white,      line width=0.45pt},
    clusterB/.style={RoyalBlue!80!black, dashed, line width=0.55pt},
    clusterR/.style={BrickRed!85!black, dashed, line width=0.55pt},
    apexB/.style={circle, fill=RoyalBlue!85!black, inner sep=0pt, minimum size=2.0pt},
    apexR/.style={circle, fill=BrickRed!85!black,  inner sep=0pt, minimum size=2.0pt},
    ptlabel/.style={font=\footnotesize, fill=white, inner sep=1pt, outer sep=2pt},
    clusterlabel/.style={font=\scriptsize, fill=white, inner sep=0.7pt, outer sep=1pt},
    zoombox/.style={dashed, black!55, line width=0.6pt},
  ]

  \begin{scope}
    \clip (-2.6,-2.6) rectangle (2.8,1.7);

    \foreach \ax/\ay in {0.91/2.12, 1.13/2.03, 1.03/1.89} {
      \foreach \bx/\by in {-1.13/1.08, -0.88/1.03, -1.04/0.88} {
        \pgfmathsetmacro{\ux}{-(\by + \ax - \ay)}
        \pgfmathsetmacro{\uy}{-\by}
        \draw[secantB] (\ux,\uy) -- (3.6,\uy);
        \draw[secantB] (\ux,\uy) -- (\ux,3.4);
        \draw[secantB] (\ux,\uy) -- ({\ux-2.5},{\uy-2.5});
      }
    }

    \foreach \ax/\ay in {2.12/0.89, 1.88/1.04, 2.04/1.13} {
      \foreach \bx/\by in {0.89/-1.11, 1.13/-0.91, 0.97/-1.14} {
        \pgfmathsetmacro{\ux}{-\bx}
        \pgfmathsetmacro{\uy}{-(\bx - \ax + \ay)}
        \draw[secantR] (\ux,\uy) -- (3.6,\uy);
        \draw[secantR] (\ux,\uy) -- (\ux,3.4);
        \draw[secantR] (\ux,\uy) -- ({\ux-2.5},{\uy-2.5});
      }
    }
  \end{scope}

  \draw[clusterB] (-1.18,-2.18) rectangle (-0.82,-1.82);
  \foreach \ax/\ay in {-0.91/-2.12, -1.13/-2.03, -1.03/-1.89}
    \node[apexB] at (\ax,\ay) {};
  \node[RoyalBlue, clusterlabel, anchor=north] at (-1, -2.22) {$A_1^1$};

  \draw[clusterB] (0.82,-1.18) rectangle (1.18,-0.82);
  \foreach \bx/\by in {1.13/-1.08, 0.88/-1.03, 1.04/-0.88}
    \node[apexB] at (\bx,\by) {};
  \node[RoyalBlue, clusterlabel, anchor=west] at (1.25, -1) {$A_2^1$};

  \draw[clusterR] (-2.18,-1.18) rectangle (-1.82,-0.82);
  \foreach \ax/\ay in {-2.12/-0.89, -1.88/-1.04, -2.04/-1.13}
    \node[apexR] at (\ax,\ay) {};
  \node[BrickRed!85!black, clusterlabel, anchor=east] at (-2.25, -1) {$A_1^2$};

  \draw[clusterR] (-1.18,0.82) rectangle (-0.82,1.18);
  \foreach \bx/\by in {-0.89/1.11, -1.13/0.91, -0.97/1.14}
    \node[apexR] at (\bx,\by) {};
  \node[BrickRed!85!black, clusterlabel, anchor=east] at (-1.25, 1) {$A_2^2$};

  \draw[zoombox] (-0.38,-0.40) rectangle (0.52,0.50);
  \node[ptlabel, anchor=south west] at (0.55,0.38) {$\mathcal{W}$};

\end{tikzpicture}
  \caption{The product grid of \cref{prop:unsigned-local-product-grid} for
  \(d=2\) and \(m=3\). The finite sets \(A_1^1,A_2^1\) (blue) and
  \(A_1^2,A_2^2\) (red) form one pair for each coordinate; the lines spanned
  by selections from \(A_1^1\times A_2^1\) and from \(A_1^2\times A_2^2\)
  intersect \((-1,1)^2\) in \(m^2\) vertical and \(m^2\) horizontal segments,
  forming the product grid.}
  \label{fig:lower-cell-count-unsigned}
\end{figure}

Before proving the proposition, we derive the following consequence.

\begin{corollary}\label{cor:cell-count-unsigned}
  Let \(d\geq1\) be fixed. For every \(t\geq d^2\), there is a \(t\)-point
  configuration \(\a\) in max-general position in \(\TPfin^d\) whose auxiliary
  arrangement \(\mathcal B(\a)\) has \(\Omega(t^{d^2})\) full-dimensional
  cells.
\end{corollary}
\begin{proof}
  Let us first consider the case where $t$ writes \(t=d^2m\) for some integer \(m\geq1\). Let \(A_s^j\) be the sets constructed in \cref{prop:unsigned-local-product-grid}. For every \(j\in[d]\), the \(m^d\) selections \(\p\in A_1^j\times\cdots\times A_d^j\) determine \(m^d\) min-tropical hyperplanes that cut $(-1,1)^d$ along \(m^d\) parallel affine hyperplanes \(x_j=c_j(\p)\). These cuts for all $j \in [d]$ decompose $(-1,1)^d$ into \((m^d+1)^d\) boxes.

  Let \(\a\) be any ordering of the \(t\) points of \(\bigcup_{j,s\in[d]}A_s^j\) and note that $\a$ is in max-general position. Any two boxes are separated by a min-tropical hyperplane spanned by some $d$ points from $\a$, and therefore are contained in different cells of the auxiliary arrangement \(\mathcal B(\a)\). It follows that \(\mathcal B(\a)\) has at least $(m^d+1)^d\geq m^{d^2}=\Omega(t^{d^2})$  full-dimensional cells.

  Now let \(t\geq d^2\) be arbitrary, and put \(m\coloneqq\lfloor
  t/d^2\rfloor\), so that \(t/2\leq t_0=d^2m\leq t\). Let \(\a_0\) be the
  \(t_0\)-point configuration constructed above for this \(m\). By
  \cref{lem:generic-extension-existence}, we may extend \(\a_0\) to a
  \(t\)-point configuration \(\a\) in max-general position. Since
  \(\mathcal B(\a)\) contains all hyperplanes of \(\mathcal B(\a_0)\), it
  still refines the grid and has \(\Omega(t_0^{d^2})=\Omega(t^{d^2})\)
  full-dimensional cells.
\end{proof}

\subsubsection[Geometric intuition of the proof of the local product-grid proposition]{Geometric intuition of the proof of \cref{prop:unsigned-local-product-grid}}

For the sake of the presentation we put $r=d+1$ throughout.

\medskip

We construct the sets \(A_1^j,\ldots,A_d^j\) independently for each $j \in [d]$. To this end, we note that near a point of the relative interior of the facet \(\Fmin_{jr}(u)\) of some min-tropical hyperplane \(\Hmin(u)\), that tropical hyperplane coincides with the affine hyperplane \(x_j-x_r=u_j-u_r\); see \cref{fig:facet-local}. In the chart \(x_r=0\), this is the affine hyperplane \(x_j=u_j-u_r\), whose position determines the apex difference \(u_j-u_r\) and vice-versa. This difference, in turn, can be controlled by the spanning points (suitably chosen).  A point \(p\) on the facet \(\Fmin_{ab}(u)\) satisfies \(p_a-u_a=p_b-u_b\), so it fixes the difference \(u_a-u_b=p_a-p_b\).

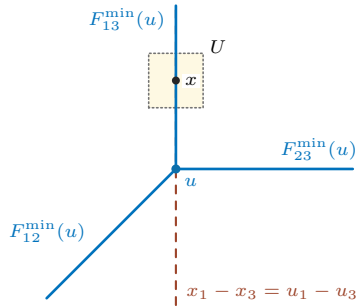
\begin{figure}[h]
  \centering
  \begin{tikzpicture}[
    x=0.9cm, y=0.9cm,
    >=latex,
    line cap=round, line join=round,
    font=\small,
    lineB/.style={RoyalBlue, line width=1pt},
    affine/.style={BrickRed!80!black, line width=0.8pt, dashed},
    apex/.style={circle, fill=RoyalBlue!85!black, inner sep=0pt, minimum size=3.4pt},
    pt/.style={circle, fill=black!85, inner sep=0pt, minimum size=3pt},
    nbhd/.style={draw=black!55, line width=0.6pt, densely dotted,
                 fill=Goldenrod!15},
    lab/.style={font=\scriptsize, fill=white, inner sep=1pt, outer sep=1pt},
  ]

  \draw[nbhd] (-0.4,0.9) rectangle (0.4,1.7);

  \draw[affine] (0,-2.0) -- (0,2.4);

  \draw[lineB] (0,0) -- (0,2.4);
  \draw[lineB] (0,0) -- (2.6,0);
  \draw[lineB] (0,0) -- (-1.9,-1.9);

  \node[apex] at (0,0) {};
  \node[RoyalBlue, lab, anchor=north west] at (0.05,-0.05) {$u$};

  \node[pt] at (0,1.3) {};
  \node[lab, anchor=west] at (0.06,1.3) {$x$};
  \node[lab, anchor=south west] at (0.42,1.62) {$U$};

  \node[RoyalBlue, lab, anchor=east] at (-0.08,2.2) {$\Fmin_{13}(u)$};
  \node[RoyalBlue, lab, anchor=south] at (2.1,0.05) {$\Fmin_{23}(u)$};
  \node[RoyalBlue, lab, anchor=south east] at (-1.25,-1.15) {$\Fmin_{12}(u)$};

  \node[BrickRed!80!black, lab, anchor=west] at (0.08,-1.8)
    {$x_1-x_3=u_1-u_3$};

\end{tikzpicture}
  \caption{In a small neighborhood \(U\) of a point \(x\) of the relative
  interior of \(\Fmin_{13}(u)\), the min-tropical line \(\Hmin(u)\)
  coincides with the affine line \(x_1-x_3=u_1-u_3\) (dashed).}
  \label{fig:facet-local}
\end{figure}

\medskip

So we fix some point $\bar u^j$ such that the facet $\Fmin_{jr}(\bar u^j)$ intersects $(-1,1)^d$. We then fix some arbitrary ordering \(j=i_0,i_1,\ldots,i_d=r\) of \([r]\), that is, a path from \(j\) to \(r\) through all indices, and put \(p^s\) on the facet \(\Fmin_{i_{s-1}i_s}(\bar u^j)\) for \(s\in[d]\). For $s \in [d]$, the point \(p^s\) fixes one difference along the path $\bar u^j_{i_{s-1}}-\bar u^j_{i_s}=p^s_{i_{s-1}}-p^s_{i_s}$. Summing these \(d\) equations, the left-hand sides telescope to \(\bar u^j_{i_0}-\bar u^j_{i_d}=\bar u^j_j-\bar u^j_r\), and we obtain
\begin{equation}\label{eq:path-telescope}
  \bar u^j_j- \bar u^j_r=\sum_{s=1}^{d}\bigl(p^s_{i_{s-1}}-p^s_{i_s}\bigr).
\end{equation}
The right-hand side of~\eqref{eq:path-telescope} is a sum of \(d\) terms, and the \(s\)-th term depends only on \(p^s\). So if each \(p^s\) is perturbed into \(m\) candidates for that point, the \(m^d\) choices of \(\p\) give, locally, \(m^d\) affine hyperplanes of $\R^d$ parallel to $\{x_j = 0\}$. When the candidates are chosen generically, these affine hyperplanes are distinct.

\subsubsection[Proof of the local product-grid proposition]{Proof of \cref{prop:unsigned-local-product-grid}}

Recall that for the sake of the presentation we put $r=d+1$ throughout. Let \(O, e_1,\ldots,e_d\) be the standard orthonormal frame of \(\R^d\), and let \(e_r\coloneqq-(e_1+\cdots+e_d) \in \R^d \).

\medskip

\begin{remark}\label{r:affproj}
Throughout the proof, we identify $\R^d \simeq \{x_{d+1} = 0\} \subset\TPfin^d$ and all points in $\{x_{d+1} = 0\} \subset\TPfin^d$ are considered through their homogeneous coordinates with $x_{d+1}=0$, which are unique over the tropical semiring. So, for instance, $e_r$ represents the point \((0,\ldots,0,1)\) of \(\TPfin^d\).
\end{remark}

\medskip

For every \(j\in[d]\) we define $\bar u^j \coloneqq 4(e_j+e_r)$. Note that the origin $O$ lies in the relative interior of \(\Fmin_{jr}(\bar u^j)\) since \((O-\bar u^j)_j=(O-\bar u^j)_r=0\) and \((O-\bar u^j)_k=4\) for every \(k\notin\{j,r\}\). (Here we use Remark~\ref{r:affproj}.) In other words, locally around $O$ the min-tropical hyperplanes $\Hmin(\bar u^1), \ldots, \Hmin( \bar u^d)$ intersect like the coordinate hyperplanes of $\R^d$.

\medskip

Let us fix \(j\in[d]\) for now. We fix an ordering \(j=i_0,i_1,\ldots,i_d=r\) of \([r]\) (which depends on $j$ but we omit this dependency for the sake of readability) and for $s \in [d]$ we define $\bar p^{j,s}\coloneqq\bar u^j-4e_{i_{s-1}}-4e_{i_s}$. Put \(\rho\coloneqq1/(2r)\) and let
\(
  \mathcal V_s^j\coloneqq
  \{q\in\R^d\colon\|q-\bar p^{j,s}\|_\infty<\rho\},
\)
where \(\|\cdot\|_\infty\) denotes the sup-norm on \(\R^d\).

\begin{lemma}\label{lem:unsigned-local-coordinate-slices}
  Any sequence \(\p=(p^1,\ldots,p^d)\in\mathcal V_1^j\times\cdots\times\mathcal V_d^j\) in max-general position spans a min-tropical hyperplane whose intersection with \((-1,1)^d\) is nonempty and equal to $\left\{x\in(-1,1)^d \colon x_j=c_j(\p)\right\}$, where $c_j(\p)\coloneqq\sum_{s=1}^{d}\pth{p^s_{i_{s-1}}-p^s_{i_s}}$.
\end{lemma}

\noindent
\Cref{fig:lower-base-config-unsigned} illustrates the lemma for \(d=2\) and \(j=1\).

\begin{figure}[t]
  \centering
  \begin{tikzpicture}[
    x=0.95cm, y=0.95cm,
    >=latex,
    line cap=round, line join=round,
    font=\small,
    lineB/.style={RoyalBlue, line width=1pt},
    faintB/.style={RoyalBlue!35, line width=0.6pt},
    localB/.style={RoyalBlue, line width=1.1pt},
    faintlocalB/.style={RoyalBlue!40, line width=0.8pt},
    neighB/.style={draw=RoyalBlue!70!black, fill=RoyalBlue!10,
                   line width=0.55pt},
    apex/.style={circle, fill=RoyalBlue!85!black, inner sep=0pt, minimum size=3.4pt},
    pointpt/.style={circle, fill=black!85, inner sep=0pt, minimum size=2.9pt},
    altpt/.style={circle, draw=black!85, fill=white, line width=0.5pt,
                  inner sep=0pt, minimum size=2.9pt},
    intpt/.style={circle, fill=Goldenrod, draw=black, line width=0.4pt,
                  inner sep=0pt, minimum size=4.2pt},
    ptlabel/.style={font=\scriptsize, fill=white, inner sep=1pt,
                    outer sep=1pt},
    panellab/.style={font=\itshape, anchor=south},
    boxB/.style={draw=black!55, line width=0.7pt, dashed},
    axis/.style={->, black!45, line width=0.45pt},
  ]

  \begin{scope}
    \node[panellab] at (0.35, 1.9) {(a) two choices for each point};

    \draw[neighB] (-1.3,-2.3) rectangle (-0.7,-1.7);
    \draw[neighB] (0.7,-1.3) rectangle (1.3,-0.7);

    \begin{scope}
      \clip (-2.7,-2.7) rectangle (3.7,1.6);
      \foreach \dd/\ee in {-0.2/-0.07, -0.2/0.07, 0.2/-0.07} {
        \pgfmathsetmacro{\ux}{\dd+\ee}
        \pgfmathsetmacro{\uy}{-1+\ee}
        \draw[faintB] (\ux,\uy) -- (3.7,\uy);
        \draw[faintB] (\ux,\uy) -- (\ux,1.6);
        \draw[faintB] (\ux,\uy) -- ({\ux-3},{\uy-3});
      }
      \draw[lineB] (0.27,-0.93) -- (3.7,-0.93);
      \draw[lineB] (0.27,-0.93) -- (0.27,1.6);
      \draw[lineB] (0.27,-0.93) -- (-2.73,-3.93);
    \end{scope}

    \draw[boxB] (-0.42,-0.42) rectangle (0.42,0.42);
    \node[ptlabel, anchor=south west] at (0.45,0.43) {$\mathcal{W}$};

    \node[intpt] at (0,0) {};
    \node[ptlabel, anchor=south east] at (-0.05,0.03) {$0$};

    \node[apex] at (0.27,-0.93) {};
    \node[RoyalBlue, ptlabel, anchor=north west] at (0.32,-0.98) {$u$};

    \node[altpt] at (-1.2,-2) {};
    \node[pointpt] at (-0.8,-2) {};
    \node[ptlabel, anchor=north] at (-1.00,-2.32) {$\mathcal{V}_1^1\ni p^1$};

    \node[altpt] at (1,-1.07) {};
    \node[pointpt] at (1,-0.93) {};
    \node[ptlabel, anchor=north] at (1.00,-1.32) {$\mathcal{V}_2^1\ni p^2$};

    \node[RoyalBlue, ptlabel, anchor=west] at (0.45,1.15) {$H^{\min}(u)$};
  \end{scope}

  \begin{scope}[shift={(7.5,-0.7)}]
    \node[panellab] at (0, 2.15) {(b) four slices inside $\mathcal{W}$};

    \draw[boxB] (-1.3,-1.3) rectangle (1.3,1.3);
    \node[ptlabel, anchor=north west] at (1.32,-1.3) {$\mathcal{W}$};

    \draw[axis] (-1.6,0) -- (1.65,0) node[anchor=west] {$x_1$};
    \draw[axis] (0,-1.55) -- (0,1.6) node[anchor=south] {$x_2$};

    \foreach \c in {-0.836,-0.402,0.402}
      \draw[faintlocalB] (\c,-1.3) -- (\c,1.3);
    \draw[localB] (0.836,-1.3) -- (0.836,1.3);

    \node[RoyalBlue, ptlabel, anchor=south] at (0.836,1.3) {$x_1=c_1(\p)$};
  \end{scope}

\end{tikzpicture}
  \caption{\Cref{lem:unsigned-local-coordinate-slices} for \(d=2\) and
  \(j=1\). (a)~Two choices of \(p^1\in\mathcal{V}_1^1\) and of
  \(p^2\in\mathcal{V}_2^1\), around the basic points \(\bar p^{1,1},\bar p^{1,2}\), and
  the four min-tropical lines they span; one of them, \(\Hmin(u)\), is drawn
  solid. (b)~These lines meet \(\mathcal{W}\) in four vertical segments
  \(\{x\in\mathcal W\colon x_1=c_1(\p)\}\), one for each choice of \((p^1,p^2)\).}
  \label{fig:lower-base-config-unsigned}
\end{figure}

\begin{proof}
  Let $\W \coloneqq (-1,1)^d$. Let \(u\) be the vector determined by \(u_r=0\) and, for $s\in[d]$, $u_{i_{s-1}}-u_{i_s}=p^s_{i_{s-1}}-p^s_{i_s}$.

\medskip
  
We first show that \(u\) is close to \(\bar u^j\). By the definitions of \(\bar u^j\) and \(\bar p^{j,s}\), we have $\bar u^j_{i_{s-1}}-\bar u^j_{i_s}=\bar p^{j,s}_{i_{s-1}}-\bar p^{j,s}_{i_s}$, for $s\in[d]$. Subtracting gives $(u-\bar u^j)_{i_{s-1}}-(u-\bar u^j)_{i_s} =(p^s-\bar p^{j,s})_{i_{s-1}}-(p^s-\bar p^{j,s})_{i_s}$. For every \(k\in[d]\), summing the preceding identities over
\(s=k,\ldots,d\) yields $(u-\bar u^j)_{i_{k-1}} - (u-\bar u^j)_{r} = \sum_{s=k}^{d}\left[(p^s-\bar p^{j,s})_{i_{s-1}}-(p^s-\bar p^{j,s})_{i_s}\right]$. Since \((u-\bar u^j)_r=0\), we obtain
\(
  |(u-\bar u^j)_{i_{k-1}}|\le 2d\rho
\) for every \(k\in[d]\).
Consequently, \(\|u-\bar u^j\|_\infty\le 2d\rho<1\).

\medskip

Next, we claim that the min-tropical hyperplane spanned by \(\p\) is \(\Hmin(u)\). To see this, let us rewrite $p^s-u=(p^s-\bar p^{j,s})+(\bar p^{j,s}-\bar u^j)+(\bar u^j-u)$. Now, observe that \(\bar p^{j,s}\) lies in the relative interior of \(\Fmin_{i_{s-1}i_s}(\bar u^j)\) since the \(i_{s-1}\)th and \(i_s\)th coordinates of \(\bar p^{j,s}-\bar u^j\) are equal and smaller by \(4\) than every other coordinate of that vector. (Here we again use Remark~\ref{r:affproj}.) Since \(\|(p^s-\bar p^{j,s})+(\bar u^j-u)\|_\infty\leq\|p^s-\bar p^{j,s}\|_\infty+\|\bar u^j-u\|_\infty<(2d+1)\rho<1\), the \(i_{s-1}\)th and \(i_s\)th coordinates of \(p^s-u\) are also smaller than all its other coordinates. Moreover, these two coordinates are equal, since \(p^s_{i_{s-1}}-u_{i_{s-1}}=p^s_{i_s}-u_{i_s}\). Thus \(p^s\in\Hmin(u)\). Since \(\p\) is in max-general position, the min-tropical hyperplane containing \(\p\) is unique, so it is \(\Hmin(u)\).

\medskip

  Finally, we show that \(\Hmin(u)\cap\mathcal W=\{x\in\mathcal W:x_j=c_j(\p)\}\). Let \(x\in\mathcal W\), so \(\|x\|_\infty<1\). Then $x-u=(O-\bar u^j)+x+(\bar u^j-u)$, and \(\|x+(\bar u^j-u)\|_\infty\leq\|x\|_\infty+\|\bar u^j-u\|_\infty<1+1=2\). The $j$th and $r$th coordinates of \(O-\bar u^j\) are smaller by \(4\) than
  all its other coordinates, so the minimum coordinate of \(x-u\) is the \(j\)th one or
  \(r\)th one. Consequently, \(x\in \Hmin(u)\) if and only if \(x_j-u_j=x_r-u_r\). Since
  \(x_r=0\), summing~$u_{i_{s-1}}-u_{i_s}=p^s_{i_{s-1}}-p^s_{i_s}$ gives $x_j=u_j-u_r =\sum_{s=1}^{d}\pth{p^s_{i_{s-1}}-p^s_{i_s}} =c_j(\p)$. As \(\bar u^j=4(e_j+e_r)\) and \((e_r)_j=-1\), we have \(\bar u^j_j=0\). Together
  with \(u_r=0\) and \(\|u-\bar u^j\|_\infty<1\), $|c_j(\p)|=|u_j|=|(u-\bar u^j)_j|<1$. Hence, the intersection $\W \cap \{x_j=c_j(\p)\}$ is non-empty.
\end{proof}

\begin{lemma}\label{lem:unsigned-distinct-slices}
  For every \(m\geq1\), there exist \(m\)-element sets
  \(A_s\subset\mathcal V_s^j\), \(s\in[d]\), such that
  the values \(c_j(\p)\), as \(\p\) ranges over
  \(A_1\times\cdots\times A_d\), are pairwise distinct.
\end{lemma}

\begin{proof}
  Label the \(m\) candidate points in each \(\mathcal V_s^j\).
  For two distinct selections \(\p,\p'\), the equality
  \(c_j(\p)=c_j(\p')\) is a nontrivial linear equation
  in the candidate coordinates and therefore defines
  a proper hyperplane.
  Excluding these finitely many hyperplanes from the
  space of candidate points
  \(\prod_{s=1}^d(\mathcal V_s^j)^m\) leaves a nonempty open
  set. Any choice in this set gives the required sets
  \(A_1,\ldots,A_d\).
\end{proof}

We now finally conclude the proof that there exist arbitrarily large point sets whose auxiliary arrangement contains large grids.

\begin{proof}[Proof of \cref{prop:unsigned-local-product-grid}]
  For each \(j\in[d]\), we let \(\{\mathcal V_s^j\}_{s\in[d]}\) denote the sets constructed for \cref{lem:unsigned-local-coordinate-slices}, and pick in each $\mathcal V_s^j$ an $m$-element subset $A_s^j$ such that the values \(c_j(\p)\), as \(\p\) ranges over \(A_1^j\times\cdots\times A_d^j\), are pairwise distinct (here we use \cref{lem:unsigned-distinct-slices}). Sufficiently small perturbations of the points of \(\cup_{j, s\in [d]}A_s^j\) keep them in their sets \(\mathcal V_s^j\) and preserve, for each \(j\), the pairwise distinctness
  of the values \(c_j(\p)\), as \(\p\) ranges over \(A_1^j\times\cdots\times A_d^j\). By \cref{lem:generic-extension-existence}, applied to the empty sequence, we may therefore perturb all \(d^2m\) points into max-general position. This gives~(i) and~(iii), and~(ii) follows from \cref{lem:unsigned-local-coordinate-slices}.
\end{proof}

\subsection{Completing the proof of Theorem~\ref{thm:main-unsigned}}

We can now prove the lower bounds of Theorem~\ref{thm:main-unsigned}.

\begin{proposition}\label{prop:unsigned-lower-bound}
  As $n \to \infty$, for every fixed $d\geq 1$,
  \[\log_2 \#\CTT(n,d) \ge d^2\,n\log_2 n -d^2\,n\log_2\log_2 n -\Oh(n).\]
\end{proposition}

\begin{proof}
  For sufficiently large \(n\), set $t=\left\lfloor\frac{n}{\log_2 n}\right\rfloor$. By \cref{cor:cell-count-unsigned}, there is a \(t\)-point configuration \(\a\) in max-general position in \(\TPfin^d\) whose auxiliary
  arrangement \(\mathcal B(\a)\) has a set \(\mathcal C\) of at least
  \(c_d\,t^{d^2}\) full-dimensional cells, for some constant \(c_d>0\)
  depending only on \(d\). For each sequence of cells \(C_1,\ldots,C_{n-t}\in\mathcal C\), we choose \((v^1,\ldots,v^{n-t})\) with \(v^\ell\in C_\ell\) so that \((\a,v^1,\ldots,v^{n-t})\) is in max-general position. Distinct cells of \(\mathcal B(\a)\) give distinct types, so distinct cell sequences give distinct tuples \((T^{\mathcal B(\a)}(v^1),\ldots,T^{\mathcal B(\a)}(v^{n-t}))\). By \cref{lem:unsigned-extension}, the corresponding extensions have distinct labeled combinatorial types, so we obtain at least
  \(\card{\mathcal C}^{\,n-t}\) distinct types. Therefore $\card{\CTT(n,d)} \;\geq\; \card{\mathcal C}^{\,n-t} \;\geq\; \bigl(c_d\,t^{d^2}\bigr)^{n-t}$, and taking logarithms, using \(t=n/\log_2 n-\Oh(1)\), gives $\log_2\card{\CTT(n,d)} \geq d^2\,n\log_2 n -d^2\,n\log_2\log_2 n-\Oh(n)$ as claimed.
\end{proof}

\begin{corollary}\label{cor:lower-cmf-rt}
  For every fixed $r \geq 2$, as $n \to \infty$,
  \[\log_2 \#\CMFT(n,r) \ge (r-1)^2\,n\log_2 n - (r-1)^2\,n\log_2\log_2 n -\Oh(n).\]
\end{corollary}
\begin{proof}
  By \eqref{eq:ct-rt-cmf} and \(\CMFTbullet(n,r)\subseteq\CMFT(n,r)\), we have $\card{\CMFT(n,r)} \;\geq\; \card{\CMFTbullet(n,r)} \;=\; \card{\CTT(n-r,r-1)}$. We then apply \cref{prop:unsigned-lower-bound} with \(d=r-1\) and \(n-r\) in place of \(n\). Since \(r\) is fixed, this shift is absorbed in the linear error term.
\end{proof}

\section{Signed tropicalization}\label{sec:signed-}

We now introduce some signed tropical geometry in preparation for the proof of Theorem~\ref{thm:main-signed} in Sections~\ref{sec:upper-signed} and~\ref{sec:lower-signed}. We start with some background on signed tropicalization (Section~\ref{sec:signed-chirotopes}), needed for the proof of the upper bound of Theorem~\ref{thm:main-signed} (Section~\ref{sec:upper-signed}). In particular, we recall properties of signed tropical numbers and signed tropical determinants, allowing us to define signed tropical chirotopes. We also recall the definition of signed tropical hyperplanes and explain how to visualize them (Section~\ref{sec:viro-hyperplanes}). Determinants of signed tropical matrices capture log-limits of determinants of parametric families of real matrices. We outline in Section~\ref{sec:signed-real-limits} how this principle allows to envision signed tropical chirotopes as logarithmic limits of affine chirotopes.

\subsection{Signed tropical numbers and signed tropicalization of chirotopes}\label{sec:signed-chirotopes}

The set of \emph{signed tropical numbers} is $\strop \coloneqq (\{\pm 1\} \times
\R) \cup \{(0,-\infty)\}$. We identify $(+1,a)$ with $a$, $(-1,a)$
with $\ominus a$, and $(0,-\infty)$ with $-\infty$. An element
$(\varepsilon,a) \in \strop$ has \emph{sign} $\sgn((\varepsilon,a)) =
\varepsilon$ and \emph{modulus} $|(\varepsilon,a)| = a \in \R \cup
\{-\infty\}$. We write \(\tzero=(0,-\infty)\) and \(\tone=(+1,0)\)
for the tropical zero and multiplicative unit. We use the following
conventions for signed tropical sums and products:
\begin{itemize}
 \item The \emph{tropical sum} \(x_1\oplus\cdots\oplus x_n\) of a 
finite family of signed tropical numbers equals any term of maximal modulus,
provided all such terms have the same sign; otherwise,
the sum is \emph{undefined}. We say that a
signed tropical sum $x_1 \oplus x_2 \oplus \ldots \oplus x_n$ is
\emph{balanced}, which we write $x_1 \oplus x_2 \oplus \ldots \oplus
x_n \nabla \tzero$, when \(\max_{1\le i\le n}|x_i|\) is attained by terms of both signs,
or when all \(x_i=\tzero\).
\item The \emph{tropical product} of signed tropical numbers is defined by $(\varepsilon,a)\odot(\delta,b)\coloneqq(\varepsilon\delta,a+b)$, where the signs are multiplied ordinarily and the moduli are added, with the convention that \(a+(-\infty)=-\infty\).
\end{itemize}

\noindent
For a nonzero signed tropical number, the multiplicative inverse is \((\varepsilon,a)^{\odot-1}=(\varepsilon,-a)\). Inversion of vectors with nonzero coordinates is taken coordinatewise. The terminology ``balanced'' originates from~\cite{maxplus90b,AkianGaubertGuterman2014}, where it is shown
that the set of signed tropical numbers can be completed by a set of quasi-zeros, called \emph{balanced} elements,
to constitute a bona-fide semiring (with an everywhere defined addition), the \emph{symmetrized tropical semiring}.
Alternatively, the signed tropical number can be equipped with a multi-valued operation,
following a construction of Krasner~\cite{Krasner1983}, leading to a hyperfield structure~\cite{Viro2010}. Both approaches
are equivalent~\cite{AkianGaubertRowen2024}.

\medskip

The \emph{signed tropical determinant} of a signed tropical matrix $A \in \strop^{q \times q}$
is given by the signed tropical sum
\begin{align}
  \tdet(A) \coloneqq \bigoplus_{\sigma \in
  \Sym_q} \bigl(\sgn(\sigma),0\bigr) \odot \bigodot_{i=1}^q
  A_{i,\sigma(i)}\label{e-def-sdet}
  \end{align}
where $\Sym_q$ denotes the set of permutations of $[q]$ and
\(\sgn(\sigma)\in\{\pm1\}\) is the ordinary sign of \(\sigma\).
The coefficient \((\sgn(\sigma),0)\) equals \(\tone\) for an even
permutation and \(\ominus\tone\) for an odd permutation. We say that a square matrix \(A\in\strop^{q\times q}\) is \emph{singular} if \(\tdet(A)\) is balanced,
and \emph{non-singular} if
\(\tdet(A)\in\strop\setminus\{\tzero\}\). We say that a matrix
\(A\in\strop^{n\times r}\) is \emph{generic} if every
maximal square submatrix, that is, every \(r\times r\) submatrix, is
non-singular.

\begin{remark}
If the unsigned tropical matrix $|A|$ obtained by applying
the modulus map to every entry of a matrix $A$ is non-singular
in the sense of~\cref{sec:pointed-cmf}, then the signed tropical
matrix \(A\) is non-singular, but the converse implication
does not hold. Indeed, when $A$ is square, non-singularity in the unsigned
sense means that there is only one term of maximal modulus in the
sum~\eqref{e-def-sdet} whereas non-singularity in the signed sense
allows several terms of maximal modulus provided they all share the same sign.
\end{remark}

\medskip

Given a sequence \(\q=(q^1,\dots,q^n)\) with \(q^i\in\strop^d\), let \(\bar Q\in\strop^{n\times(d+1)}\) be the matrix whose \(i\)-th row is \((q^i,\tone)\). We call \(\bar Q\) the \emph{homogenized matrix} of \(\q\).
More generally, a \emph{matrix of} \(\q\) is any matrix whose \(i\)-th row
is a nonzero signed tropical scalar multiple of \((q^i,\tone)\). We say that \(\q\) is \emph{in min-general position} if every square
submatrix of \(\bar Q\) is non-singular. 
Unlike genericity of \(\bar Q\), this condition includes square submatrices of every order. In particular,
the \(1\times1\) submatrices force every coordinate of every \(q^i\) to be
nonzero. This condition is independent of the chosen row representatives and of
the order of the points.

\medskip

The \emph{signed tropical chirotope} of a sequence \(\q=(q^1,\ldots,q^n)\) of points in \(\strop^d\) in min-general position is the map \(\chi_{\q}\) that associates to every \((d+1)\)-tuple \(I=(i_1,\ldots,i_{d+1})\) of distinct indices in \([n]\) the sign
\[ \chi_{\q}(I)=\chi_{\q}(i_1, \ldots, i_{d+1}) = \sgn \pth{\tdet \pth{\begin{array}{cc} q^{i_1} & \tone\\ \vdots & \vdots\\ q^{i_{d+1}} & \tone \end{array}}}.\]
Since \(\tdet\) changes sign when two rows are swapped, \(\chi_{\q}\) is alternating, hence determined by its values on subsets \(I\in\binom{[n]}{d+1}\); here and below, we identify a subset of \([n]\) with the tuple of its elements in increasing order.
For a tuple \(I=(i_1,\ldots,i_k)\) of distinct indices in \([n]\), we write \(\q^I\coloneqq(q^{i_1},\ldots,q^{i_k})\).
We denote by \(\OTT(n,d)\) the set of signed tropical chirotopes, also called \emph{tropical labeled order types}, of \(n\)-point sequences in min-general position in \(\strop^d\).
Their interpretation as limits of real chirotopes is established in \cref{prop:tropical-real-realizability}.

\subsection{Signed tropical space and signed tropical hyperplanes}
\label{sec:viro-hyperplanes}

The multiplicative group \(\strop^\times=\strop\setminus\{\tzero\}\) acts on \(\strop^{d+1}\) by signed tropical multiplication \(\odot\).  We can thus define the \emph{$d$-dimensional projective signed tropical space} as \(\TP^d_{\pm}\coloneqq (\strop^{d+1}\setminus\{(\tzero,\ldots,\tzero)\})/\strop^\times\).  We also work in the finite part of the projective space, that is \(\TPpmfin^d\coloneqq(\strop^\times)^{d+1}/\strop^\times\). We can further identify each point of the affine chart \(x_{d+1}\ne\tzero\)  with its unique representative \(x=(x_1,\ldots,x_d,\tone)\), where \(\tone=(+1,0)\).

\medskip

\begin{figure}[h]
  \centering
  \begin{tikzpicture}[
    x=0.42cm, y=0.42cm,
    >=latex,
    line cap=round, line join=round,
    font=\small,
    axis/.style={<->, black!60, thin},
    bdry/.style={line width=1.7pt, black!82, line cap=round},
    thinedge/.style={black!30, thin},
    inflabel/.style={font=\scriptsize, fill=white, inner sep=1pt, text=black!55},
    ptlabel/.style={font=\footnotesize, fill=white, inner sep=1pt, outer sep=1pt},
    chartlabel/.style={font=\normalsize, text=black!80},
    axislabel/.style={font=\footnotesize, fill=white, inner sep=1.2pt, text=black!60},
    panel/.style={font=\itshape, text=black!75},
  ]

  \node[panel, anchor=south] at (0,5.7) {The four sign charts};

  \fill[RoyalBlue!8]   (0.7,0.7) rectangle (4.7,4.7);
  \draw[thinedge]      (0.7,0.7) rectangle (4.7,4.7);
  \draw[bdry] (4.7,0.7) -- (0.7,0.7) -- (0.7,4.7);
  \node[chartlabel] at (2.9,2.9) {$(+,+)$};
  \node[ptlabel] at (0.7,0.7) {$\tzero$};

  \fill[BrickRed!7]    (-4.7,0.7) rectangle (-0.7,4.7);
  \draw[thinedge]      (-4.7,0.7) rectangle (-0.7,4.7);
  \draw[bdry] (-4.7,0.7) -- (-0.7,0.7) -- (-0.7,4.7);
  \node[chartlabel] at (-2.9,2.9) {$(-,+)$};
  \node[ptlabel] at (-0.7,0.7) {$\tzero$};

  \fill[Goldenrod!13]  (-4.7,-4.7) rectangle (-0.7,-0.7);
  \draw[thinedge]      (-4.7,-4.7) rectangle (-0.7,-0.7);
  \draw[bdry] (-4.7,-0.7) -- (-0.7,-0.7) -- (-0.7,-4.7);
  \node[chartlabel] at (-2.9,-2.9) {$(-,-)$};
  \node[ptlabel] at (-0.7,-0.7) {$\tzero$};

  \fill[ForestGreen!8] (0.7,-4.7) rectangle (4.7,-0.7);
  \draw[thinedge]      (0.7,-4.7) rectangle (4.7,-0.7);
  \draw[bdry] (4.7,-0.7) -- (0.7,-0.7) -- (0.7,-4.7);
  \node[chartlabel] at (2.9,-2.9) {$(+,-)$};
  \node[ptlabel] at (0.7,-0.7) {$\tzero$};

  \draw[->, line width=1pt, black!70] (6.2,0) -- (9.3,0);
  \node[panel, anchor=south, text=black!70] at (7.75,0.35) {glue};
  \node[font=\scriptsize, text=black!60, anchor=north] at (7.75,-0.35)
    {$x_1{=}\tzero,\ x_2{=}\tzero$};

  \begin{scope}[shift={(16.2,0)}]
    \node[panel, anchor=south] at (0,5.7) {The signed tropical plane};

    \fill[RoyalBlue!8]   (0,0) rectangle (5.0,5.0);
    \fill[BrickRed!7]    (0,0) rectangle (-5.0,5.0);
    \fill[Goldenrod!13]  (0,0) rectangle (-5.0,-5.0);
    \fill[ForestGreen!8] (0,0) rectangle (5.0,-5.0);

    \draw[bdry] (-5.0,0) -- (5.0,0);
    \draw[bdry] (0,-5.0) -- (0,5.0);
    \draw[axis] (5.0,0) -- (5.7,0);
    \draw[axis] (-5.0,0) -- (-5.7,0);
    \draw[axis] (0,5.0) -- (0,5.7);
    \draw[axis] (0,-5.0) -- (0,-5.7);
    \node[inflabel] at (5.2,0.6)  {$\oplus\infty$};
    \node[inflabel] at (-5.2,0.6) {$\ominus\infty$};
    \node[inflabel] at (1.05,5.1)  {$\oplus\infty$};
    \node[inflabel] at (1.05,-5.0) {$\ominus\infty$};
    \node[ptlabel] at (0,0) {$\tzero$};

    \node[chartlabel] at (2.8,2.8)   {$(+,+)$};
    \node[chartlabel] at (-2.8,2.8)  {$(-,+)$};
    \node[chartlabel] at (-2.8,-2.8) {$(-,-)$};
    \node[chartlabel] at (2.8,-2.8)  {$(+,-)$};

    \node[axislabel, rotate=90, anchor=south] at (-0.42,2.0) {$x_1=\tzero$};
    \node[axislabel, anchor=north] at (2.2,-0.3) {$x_2=\tzero$};
  \end{scope}
\end{tikzpicture}
  \caption{The \emph{signed tropical space} $\strop^2$.
  \emph{Left:} four copies of \(\T^2\), indexed by
\(s\in\{\pm1\}^2\), drawn separately. The bold edges correspond to \(z_i=-\infty\), or equivalently
\(x_i=\tzero\). \emph{Right:} identifying boundary points according to \(\sim\) gives the signed tropical plane. All four copies meet at \((\tzero,\tzero)\).}
\label{fig:viro-plane}
\end{figure}
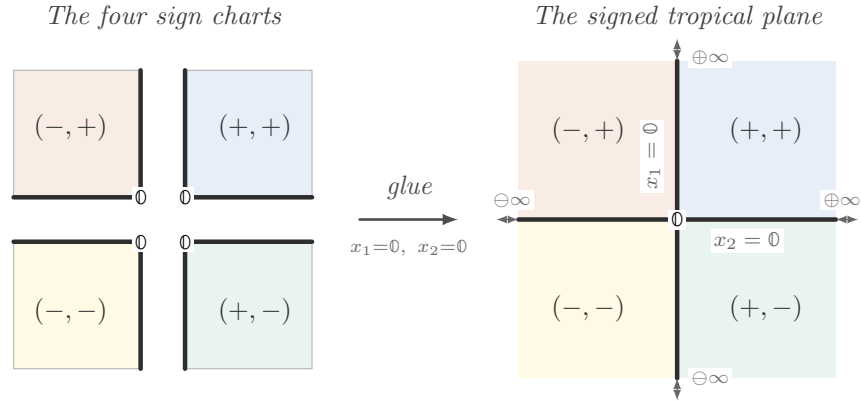

Omitting the last coordinate identifies our affine chart with
the \emph{signed tropical space} \(\strop^d\).
Consider~\(2^d\) copies of \(\T^d\), indexed by
\(s\in\{\pm1\}^d\). A labelled point \((s,z)\) represents
\(x\in\strop^d\) with coordinates
\[
  x_i=
  \begin{cases}
    (s_i,z_i), & z_i\in\R,\\
    \tzero,   & z_i=-\infty,
  \end{cases}
  \qquad i\in[d].
\]
When \(z_i=-\infty\), both choices \(s_i=+1\) and \(s_i=-1\)
give the same coordinate \(x_i=\tzero\).
We therefore glue these copies by
\[
  (s,z)\sim(t,z')
  \quad\Longleftrightarrow\quad
  z=z'
  \ \text{and}\
  s_i=t_i \text{ whenever } z_i\ne-\infty,
\]
giving $\pth{\{\pm1\}^d\times\T^d}/{\sim} \;\simeq\;\strop^d$. The case \(d=2\) is illustrated in \cref{fig:viro-plane}. For each \(s\in\{\pm1\}^d\), the \emph{sign chart} $s$ of \(\TPpmfin^d\) is \(\mathcal C_s\coloneqq
  \{x\in\strop^d:\sgn(x_i)=s_i\text{ for all }i\in[d]\} \simeq\R^d\),
where the identification is by taking moduli.
Thus the signed tropical space consists of the sign charts together with
the zero-coordinate loci \(x_i=\tzero\), \(i\in[d]\),
to which the charts are attached.

\medskip

The signed tropical space being introduced, we next present the signed counterpart of the material presented in Section~\ref{sec:unsigned-geometry}.

\medskip

Let \(w\in\TPpmfin^d\). The \emph{signed max-tropical hyperplane}
with apex \(w\) is
\[
  \Hsmax(w)\coloneqq
  \Bigl\{x\in\TP^d_{\pm}:
    \bigoplus_{j=1}^{r}w_j^{\odot-1}\odot x_j
    \;\nabla\;\tzero\Bigr\},
\]
that is, the locus where the tropical sum is balanced.
Equivalently, \(x\in\Hsmax(w)\) if and only if there exist
distinct indices \(i,j\in[r]\) satisfying
\[ |x_i|-|w_i| =|x_j|-|w_j|  =\max_{k\in[r]}(|x_k|-|w_k|) \quad \text{and} \quad \sgn(w_i)\sgn(x_i) =-\sgn(w_j)\sgn(x_j).\]
These conditions are independent of the chosen representatives:
rescaling \(w\) or \(x\) adds a common constant to all modulus
differences and multiplies all sign products by a common sign.

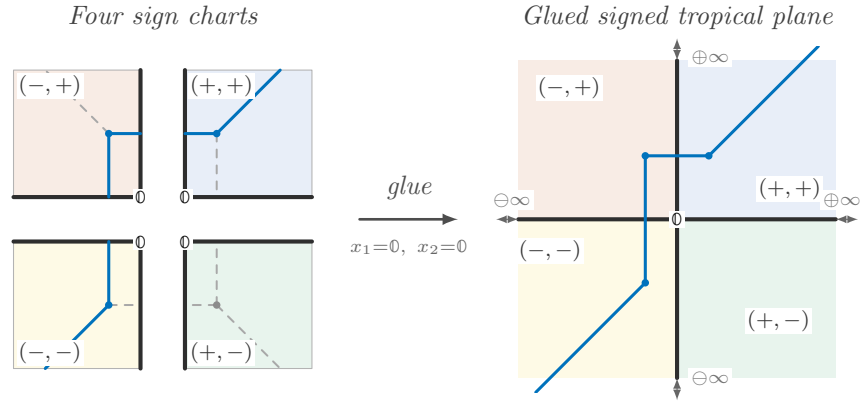
\begin{figure}[h]
  \centering
  \begin{tikzpicture}[
    x=0.42cm, y=0.42cm,
    >=latex, line cap=round, line join=round, font=\small,
    axis/.style={<->, black!60, thin},
    bdry/.style={line width=1.55pt, black!82, line cap=round},
    thinedge/.style={black!30, thin},
    lineB/.style={RoyalBlue, line width=1.15pt},
    ghost/.style={black!34, line width=0.75pt, dashed},
    bendpt/.style={circle, fill=RoyalBlue, inner sep=0pt, minimum size=2.8pt},
    inflabel/.style={font=\scriptsize, fill=white, inner sep=1pt, text=black!55},
    ptlabel/.style={font=\scriptsize, fill=white, inner sep=0.8pt, outer sep=1pt},
    chartlabel/.style={font=\footnotesize, fill=white, inner sep=0.8pt,
      text=black!80},
    panel/.style={font=\itshape, text=black!75},
  ]

  \node[panel, anchor=south] at (0,5.7) {Four sign charts};

  \fill[RoyalBlue!8] (0.7,0.7) rectangle (4.7,4.7);
  \draw[thinedge] (0.7,0.7) rectangle (4.7,4.7);
  \draw[bdry] (4.7,0.7) -- (0.7,0.7) -- (0.7,4.7);
  \draw[lineB] (1.7,2.7) -- (3.7,4.7);
  \draw[lineB] (1.7,2.7) -- (0.7,2.7);
  \draw[ghost] (1.7,2.7) -- (1.7,0.7);
  \node[bendpt] at (1.7,2.7) {};
  \node[chartlabel] at (1.85,4.2) {$(+,+)$};
  \node[ptlabel] at (0.7,0.7) {$\tzero$};

  \fill[BrickRed!7] (-4.7,0.7) rectangle (-0.7,4.7);
  \draw[thinedge] (-4.7,0.7) rectangle (-0.7,4.7);
  \draw[bdry] (-4.7,0.7) -- (-0.7,0.7) -- (-0.7,4.7);
  \draw[ghost] (-1.7,2.7) -- (-3.7,4.7);
  \draw[lineB] (-1.7,2.7) -- (-0.7,2.7);
  \draw[lineB] (-1.7,2.7) -- (-1.7,0.7);
  \node[bendpt] at (-1.7,2.7) {};
  \node[chartlabel] at (-3.55,4.2) {$(-,+)$};
  \node[ptlabel] at (-0.7,0.7) {$\tzero$};

  \fill[Goldenrod!13] (-4.7,-4.7) rectangle (-0.7,-0.7);
  \draw[thinedge] (-4.7,-4.7) rectangle (-0.7,-0.7);
  \draw[bdry] (-4.7,-0.7) -- (-0.7,-0.7) -- (-0.7,-4.7);
  \draw[lineB] (-1.7,-2.7) -- (-3.7,-4.7);
  \draw[lineB] (-1.7,-2.7) -- (-1.7,-0.7);
  \draw[ghost] (-1.7,-2.7) -- (-0.7,-2.7);
  \node[bendpt] at (-1.7,-2.7) {};
  \node[chartlabel] at (-3.55,-4.2) {$(-,-)$};
  \node[ptlabel] at (-0.7,-0.7) {$\tzero$};

  \fill[ForestGreen!8] (0.7,-4.7) rectangle (4.7,-0.7);
  \draw[thinedge] (0.7,-4.7) rectangle (4.7,-0.7);
  \draw[bdry] (4.7,-0.7) -- (0.7,-0.7) -- (0.7,-4.7);
  \draw[ghost] (1.7,-2.7) -- (3.7,-4.7);
  \draw[ghost] (1.7,-2.7) -- (0.7,-2.7);
  \draw[ghost] (1.7,-2.7) -- (1.7,-0.7);
  \node[bendpt, fill=black!45] at (1.7,-2.7) {};
  \node[chartlabel] at (1.85,-4.2) {$(+,-)$};
  \node[ptlabel] at (0.7,-0.7) {$\tzero$};

  \draw[->, line width=1pt, black!70] (6.2,0) -- (9.3,0);
  \node[panel, anchor=south, text=black!70] at (7.75,0.35) {glue};
  \node[font=\scriptsize, text=black!60, anchor=north] at (7.75,-0.35)
    {$x_1{=}\tzero,\ x_2{=}\tzero$};

  \begin{scope}[shift={(16.2,0)}]
    \node[panel, anchor=south] at (0,5.7) {Glued signed tropical plane};

    \fill[RoyalBlue!8]   (0,0) rectangle (5.0,5.0);
    \fill[BrickRed!7]    (0,0) rectangle (-5.0,5.0);
    \fill[Goldenrod!13]  (0,0) rectangle (-5.0,-5.0);
    \fill[ForestGreen!8] (0,0) rectangle (5.0,-5.0);

    \draw[bdry] (-5.0,0) -- (5.0,0);
    \draw[bdry] (0,-5.0) -- (0,5.0);
    \draw[axis] (5.0,0) -- (5.7,0);
    \draw[axis] (-5.0,0) -- (-5.7,0);
    \draw[axis] (0,5.0) -- (0,5.7);
    \draw[axis] (0,-5.0) -- (0,-5.7);
    \node[inflabel] at (5.2,0.6)  {$\oplus\infty$};
    \node[inflabel] at (-5.2,0.6) {$\ominus\infty$};
    \node[inflabel] at (1.05,5.1)  {$\oplus\infty$};
    \node[inflabel] at (1.05,-5.0) {$\ominus\infty$};
    \node[ptlabel] at (0,0) {$\tzero$};

    \begin{scope}
      \clip (-5.45,-5.45) rectangle (5.45,5.45);
      \draw[lineB] (-1,2) -- (1,2);
      \draw[lineB] (-1,-2) -- (-1,2);
      \draw[lineB] (1,2) -- (4.45,5.45);
      \draw[lineB] (-1,-2) -- (-4.45,-5.45);
    \end{scope}

    \node[bendpt] at (1,2) {};
    \node[bendpt] at (-1,2) {};
    \node[bendpt] at (-1,-2) {};

    \node[chartlabel] at (3.55,0.95) {$(+,+)$};
    \node[chartlabel] at (-3.4,4.15) {$(-,+)$};
    \node[chartlabel] at (-4.0,-1.0) {$(-,-)$};
    \node[chartlabel] at (3.2,-3.2) {$(+,-)$};
  \end{scope}
\end{tikzpicture}
  \caption{The signed max-tropical line of
  \cref{ex:signed-max-tropical-line} (schematic).
  \emph{Left:} retained facets are solid blue; discarded facets are dashed.
  \emph{Right:} the glued line in the signed tropical plane.}
  \label{fig:signed-hyperplane-chart}
\end{figure}

\medskip

A signed max-tropical hyperplane in the signed tropical space can be visualized thanks to the following method. We consider each chart indexed by \(s\in\{\pm1\}^d\) in turn. By our convention, the last homogeneous coordinate is normalized to $\tone$, so we set $s_r=+1$. In the chart of sign \(s\in\{\pm1\}^d\), we draw the set of points~$x$ such that $\sgn(x) = s$ and $|x|$ belongs to the unsigned max-tropical hyperplane with apex \(|w|\). We then only keep the facets \(\Fmax_{ab}(|w|)\) defined in~\eqref{eq:tropical-facet} precisely when \(\sgn(w_a)s_a=-\sgn(w_b)s_b\). The facet equalities and inequalities extend to the chart boundaries by allowing modulus \(-\infty\) for zero coordinates, which cannot
attain the maximum. Gluing the retained pieces along these boundaries gives the affine part of \(\Hsmax(w)\) in the signed tropical space.

\begin{example}\label{ex:signed-max-tropical-line}
  The signed max-tropical line $\left\{ x\in\TP^2_{\pm}: \pth{(\ominus 1)\odot x_1} \ \oplus \ x_2 \ \oplus \ \pth{(\ominus 2)\odot x_3} \;\nabla\;\tzero \right\}$ has apex \(w=(\ominus 1,2,\ominus\tone)\), with modulus vector \(|w|=(1,2,0)\). In the finite part of chart \((s_1,s_2)\), the three terms have signs \((-s_1,s_2,-1)\) and the retained facets are 

\begin{center}
  \begin{tabular}{ccc}
    \toprule
    Chart & Term signs & Retained pairs \\
    \midrule
    \((+,+)\) & \((-,+,-)\) & \(\{1,2\},\ \{2,3\}\) \\
    \((-,+)\) & \((+,+,-)\) & \(\{1,3\},\ \{2,3\}\) \\
    \bottomrule
  \end{tabular}
  \qquad 
  \begin{tabular}{ccc}
    \toprule
    Chart & Term signs & Retained pairs \\
    \midrule
    \((-,-)\) & \((+,-,-)\) & \(\{1,2\},\ \{1,3\}\) \\
    \((+,-)\) & \((-,-,-)\) & None \\
    \bottomrule
  \end{tabular}
\end{center}
In \((+,-)\), all nonzero terms have the same sign, so the signed line
is empty there, even at the unsigned apex.
\Cref{fig:signed-hyperplane-chart} shows the facet selection in
all four charts and the line obtained after gluing.
\end{example}

\medskip

Let \(A\in\strop^{d\times(d+1)}\). For $j \in [d+1]$ let \(A_{\widehat j}\) be the square matrix obtained by deleting column \(j\).  Suppose that every \(A_{\widehat j}\) is non-singular.  The \emph{signed Cramer vector} \(\operatorname{cr}_{\pm}(A)\in(\strop^\times)^{d+1}\) is the vector with $j$th coordinate $[\operatorname{cr}_{\pm}(A)]_j\coloneqq ((-1)^{d+1+j},0)\odot\tdet(A_{\widehat j})$. The signed tropical Cramer rule~\cite[Cor.~4.19]{AkianGaubertGuterman2014} implies that
$\operatorname{cr}_{\pm}(A)$ is the unique nonzero solution
in $\strop^{d+1}$ of the homogeneous balance system
\(
  \left\{
    \bigoplus_{j=1}^{d+1} A_{ij}\odot y_j
    \;\nabla\;\tzero
  \right\}_{i\in[d]},
\)
up to multiplication by a nonzero signed tropical scalar. Hence, if we interpret the rows of $A$ as the coordinates of a point sequence \(\q=(q^1,\ldots,q^d)\) in \(\TPpmfin^d\), then the projective class of \(\operatorname{cr}_{\pm}(A)\) only depends on that point sequence. We call this projective class of \(\operatorname{cr}_{\pm}(A)\), that is this point in \(\TPpmfin^d\), the \emph{signed Cramer point} of $\q$ and denote it by \(\operatorname{cr}_{\pm}(\q)\in\TPpmfin^d\). The signed tropical Cramer rule therefore implies that $d$ points
in min-general position lie on a unique signed max-tropical hyperplane:

\begin{lemma}\label{lem:signed-cramer-intersection}
Let \(\q=(q^1,\ldots,q^d)\) be a point sequence in min-general
position in \(\TPpmfin^d\). Then
\(
  \Hsmax(\operatorname{cr}_{\pm}(\q)^{\odot-1})
\)
is the unique signed max-tropical hyperplane containing all points of \(\q\).
\end{lemma}

\begin{proof}
Let \(Q\) have the points of \(\q\) as rows.
An apex \(w\in\TPpmfin^d\) defines a hyperplane containing all points of
\(\q\) precisely when $\bigoplus_{j=1}^{d+1}Q_{ij}\odot w_j^{\odot-1} \;\nabla\;\tzero$ for every $i\in[d]$. min-general position makes every \(d\times d\) submatrix of \(Q\)
non-singular. The signed tropical Cramer rule therefore gives the unique
projective solution \(w^{\odot-1}=\operatorname{cr}_{\pm}(\q)\), with
all coordinates nonzero, so the unique apex is
\(w=\operatorname{cr}_{\pm}(\q)^{\odot-1}\).
\end{proof}

\medskip

As in the unsigned setting, we call \(\Hsmax(\operatorname{cr}_{\pm}(\q)^{\odot-1})\) the \emph{signed max-tropical hyperplane spanned by} \(\q\).

\subsection{Logarithmic limit and the geometry of signed tropical chirotopes}\label{sec:signed-real-limits}

We finally explain how the signed tropical space provides a geometric interpretation of signed tropical chirotopes in terms of positions of points relative to induced hyperplanes. Before that, we outline how the signed tropical geometry can be understood as a logarithmic limit as illustrated in \cref{fig:signed-puiseux-to-tropical}.

\subsubsection{Signed tropicalization as a logarithmic limit}

Our purpose here is to clarify  how signed tropical chirotopes relate to real ones, even though we do not use this relation in our proofs.

\medskip

The \emph{monomial lift} of a signed tropical number
$a=(\varepsilon,\alpha)$ is
$a(t)\coloneqq\varepsilon t^\alpha$, with the convention that $t^{-\infty}$ is the null monomial. Note that for $a, b \in \strop$ we have $(a \odot b)(t) = a(t)*b(t)$ and, whenever $a\oplus b$ is defined and nonzero, it records the sign of the leading coefficient and the leading exponent of $a(t)+b(t)$. 
For $t>1$, this monomial interpretation can be considered throught the {\em signed
logarithmic map}:
\[
  \operatorname{Log}^{\pm}_t(x) =
  \begin{cases}
    \bigl(\sgn(x),\log_t|x|\bigr),&x\neq0,\\ \tzero,&x=0.
  \end{cases}
\]
Applying this map coordinatewise to a point of $\R^d$ maps it to the signed tropical space of \cref{sec:viro-hyperplanes}. Now, for $j \in [d+1]$ let $c_j(t)=\delta_jt^{b_j}$ for with $\delta_j\in\{\pm1\},\quad b_j\in\R$ and write $c_j=(\delta_j,b_j)$ for the corresponding signed tropical number. For each \(t>1\), the real coefficients \(c_j(t)\) define the affine hyperplane $\mathcal H_t(c)=\left\{x\in\R^d: \sum_{j=1}^d c_j(t)x_j+c_{d+1}(t)=0 \right\}$. The signed tropical hyperplane
\[\mathcal H^{\mathrm{trop}}(c)= \left\{x\in\strop^d: \left(\bigoplus_{j=1}^d c_j\odot x_j\right) \oplus c_{d+1}\;\nabla\;\tzero \right\} =\Hsmax(c^{\odot-1})\]
is the limit of \(\operatorname{Log}^{\pm}_t(\mathcal H_t(c))\) as \(t\to\infty\)~\cite{Tabera2015,Alessandrini2013}. This is illustrated in \cref{fig:signed-puiseux-to-tropical} for $d=2$, $c^1(t)=(-t^2,1,-t^4)$, $c^2(t)=(-1,t^3,-t^4)$, and $c^3(t)=(1,1,-t^5)$.

\medskip

It turns out that the real chirotope of a monomial lift stabilizes to the signed tropical chirotope defined in \cref{sec:signed-chirotopes}.

\begin{proposition}\label{prop:tropical-real-realizability}
  Let \(\q\in(\strop^d)^n\) be in min-general position, and let
\(\q(t)\in(\R^d)^n\) be its coordinatewise monomial lift.
There exists \(T>1\) such that, for every \(t>T\), the real
configuration \(\q(t)\) is in general position and its chirotope $\chi_{\q(t)}$ equals the signed tropical chirotope $\chi_{\q}$.
\end{proposition}

\begin{proof}
Let \(\bar Q\) be the homogenized matrix of \(\q\); its monomial
lift \(\bar Q(t)\) is the homogenized matrix of \(\q(t)\).
Consider a \(k\times k\) submatrix \(A\) of \(\bar Q\).
Its entries are nonzero by min-general position, so write
\(A_{ij}=(\varepsilon_{ij},a_{ij})\). The determinant expansion is
\[
  \det A(t)=\sum_{\sigma\in\Sym_k}
    \sgn(\sigma)
    \left(\prod_{i=1}^k\varepsilon_{i,\sigma(i)}\right)
    t^{\sum_{i=1}^k a_{i,\sigma(i)}}.
\]
Since \(A\) is non-singular, the terms of largest exponent have
the same sign and cannot cancel one another. They therefore determine
the eventual sign of the determinant, so
\(\sgn\det A(t)=\sgn\tdet(A)\) for all sufficiently large \(t\). There are finitely many square submatrices, so one threshold \(T\)
works for all of them. Thus every square minor of \(\bar Q(t)\) is
nonzero for \(t>T\), and \(\q(t)\) is in general position.
In particular, the maximal minors have the signs prescribed by
\(\chi_{\q}\), proving \(\chi_{\q(t)}=\chi_{\q}\).
\end{proof}

As a consequence, we have \(\OTT(n,d)\subseteq\OTR(n,d)\) for every $n > d \ge 1$.

\subsubsection{Signed tropical chirotopes and induced hyperplanes}\label{s:signed chirotopes and hyperplanes}

In the real setting, let \(\p=(p^1,\ldots,p^n)\) be a point sequence
in \(\R^d\) such that any \(d+1\) of its points are affinely independent.
Its \emph{chirotope} is the map \(\chi_{\p}\) that
associates to every \((d+1)\)-tuple \(I=(i_1,\ldots,i_{d+1})\) of
distinct indices in \([n]\) the sign
\[
  \chi_{\p}(I)=\chi_{\p}(i_1,\ldots,i_{d+1})=\sgn\det
  \begin{pmatrix}
    p^{i_1} & 1\\
    \vdots & \vdots\\
    p^{i_{d+1}} & 1
  \end{pmatrix}\in\{\pm1\}.
\]
Geometrically, the chirotope of \(\p\) records the side of each
point relative to the hyperplanes spanned by \(d\) affinely independent
points of \(\p\), with signs fixed by their order.

\begin{figure}[h]
  \centering
  \begin{tikzpicture}[
    line cap=round, line join=round, font=\small,
    spannedline/.style={black!80, line width=0.9pt},
    markedpoint/.style={circle, fill=black, inner sep=0pt,
      minimum size=3.8pt},
    pointname/.style={inner sep=1pt, outer sep=1pt},
    panelname/.style={font=\small\itshape, anchor=south},
  ]

  \begin{scope}[shift={(0,-1.7cm)},x=0.55cm,y=0.4cm]
    \node[panelname] at (4.5,11.75) {Real};
    \draw[spannedline] (0,{12/7}) -- (9,{30/7});

    \node[markedpoint] at (1,2) {};
    \node[pointname,anchor=south] at (1,2.4) {$q^1$};
    \node[markedpoint] at (8,4) {};
    \node[pointname,anchor=north] at (8,3.6) {$q^2$};
    \node[markedpoint] at (2,8) {};
    \node[pointname,anchor=west] at (2.25,8) {$x^+$};
    \node[markedpoint] at (2,0.5) {};
    \node[pointname,anchor=west] at (2.25,0.5) {$x^-$};
  \end{scope}

  \begin{scope}[shift={(9cm,0)},x=0.38cm,y=0.30cm]
    \pgfmathsetmacro{\virooffset}{1/ln(2)}
    \pgfmathsetmacro{\virobendx}{2+\virooffset}
    \pgfmathsetmacro{\virobendy}{4+\virooffset}
    \node[panelname] at (0,10) {The signed tropical plane};

    \draw[spannedline]
      ({\virobendx+3},{\virobendy+3})
      -- (\virobendx,\virobendy)
      -- ({-\virobendx},\virobendy)
      -- ({-\virobendx},{-\virobendy})
      -- ({-\virobendx-3},{-\virobendy-3});

    \node[markedpoint] at (\virooffset,\virobendy) {};
    \node[pointname,anchor=south]
      at (\virooffset,{\virobendy+0.45}) {$q^1$};
    \node[markedpoint] at ({3+\virooffset},{5+\virooffset}) {};
    \node[pointname,anchor=west]
      at ({3.4+\virooffset},{4.7+\virooffset}) {$q^2$};
    \node[markedpoint] at ({1+\virooffset},{6+\virooffset}) {};
    \node[pointname,anchor=west]
      at ({1.4+\virooffset},{6+\virooffset}) {$x^+$};
    \node[markedpoint] at ({1+\virooffset},{2+\virooffset}) {};
    \node[pointname,anchor=west]
      at ({1.4+\virooffset},{2+\virooffset}) {$x^-$};
  \end{scope}
\end{tikzpicture}
  \caption{Lines spanned by \(q^1,q^2\) in the real plane (left) and
  the signed tropical plane (right), with \(x^\pm\) on opposite sides.}
  \label{fig:chirotope-sides}
\end{figure}
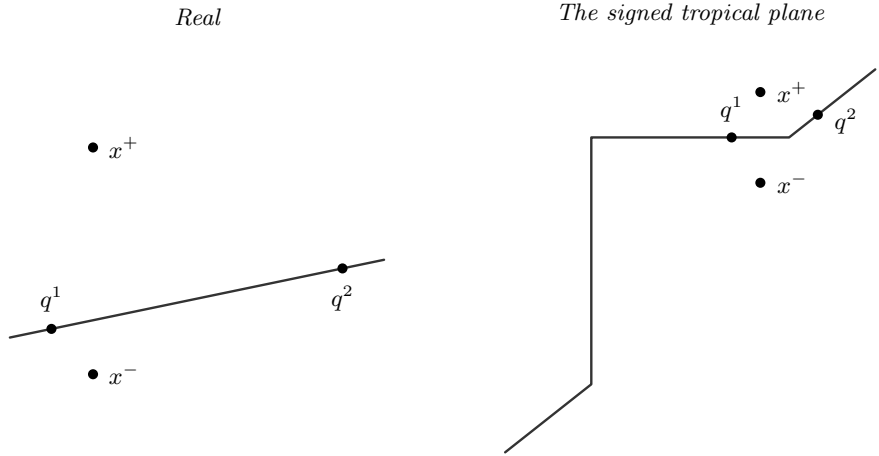

Let \(\q=(q^1,\ldots,q^n)\) be a point sequence in min-general
position in \(\strop^d\), and let \(I=(i_1,\ldots,i_d)\) be a tuple of distinct indices in \([n]\).
For a point \(x\in\strop^d\), write
\(\tdet(\q^I,x)\) for the signed tropical
determinant of the matrix with rows
\((q^{i_1},\tone),\ldots,(q^{i_d},\tone),(x,\tone)\).
Expansion along the last row gives
\begin{equation}\label{eq:determinant-expansion}
  \tdet(\q^I,x)
  =\bigoplus_{j=1}^{d}c_j\odot x_j\oplus c_{d+1},
\end{equation}
where \(c_j\) is the signed tropical cofactor of the \(j\)th entry of the last row. By \cref{lem:signed-cramer-intersection}, the signed tropical determinant is balanced exactly when \(x\) lies on the signed max-tropical hyperplane spanned by \(\q^I\). This hyperplane divides \(\strop^d\) into two connected components. In fact, both components are the interior of (affine) tropical half-spaces; they are tropically convex and in particular, they are connected. (See e.g.~\cite{SGKatz} for more information on half-spaces and tropical convexity in the affine setting.)
Moreover,
\(\tdet(\q^I,x)\) has sign \(+1\) on one of them and \(-1\) on the other.
Thus, as in the real setting, the signed tropical chirotope of \(\q\) records the side of each point relative to the signed max-tropical hyperplanes spanned by \(d\) points of \(\q\). (See \Cref{fig:chirotope-sides}.)

\section{The upper bound of Theorem~\ref{thm:main-signed}}\label{sec:upper-signed}

We now prove our upper bound on $\#\OTT(n,d)$. The proof hinges on the elementary observation that every signed tropical chirotope can be recovered from a coherent matching field and the signs of the matrix entries used by its matchings.

\medskip

We define the \emph{sign} of a perfect matching $M\subset I\times[r]$, with $I=\{i_1<\cdots<i_r\}$, as $\sgn(M)\coloneqq\sgn(\sigma)\in\{\pm1\}$, where $\sigma\in\Sym_r$ is the unique permutation with $(i_\ell,\sigma(\ell))\in M$ for all $\ell\in[r]$. For $I\in\binom{[n]}{d+1}$, let $\bar Q_I$ denote the submatrix of $\bar Q$
with rows indexed by $I$, in increasing order. The following proposition is elementary.

\begin{proposition}\label{prop:ot-from-cmf-signs}
  Let $\q\in(\strop^d)^n$ be in min-general position, with
  homogenized matrix $\bar Q$.
  There exists a coherent matching field
  $\Lambda=(M_I)\in\CMFT(n,d+1)$ such that, for every
  $I\in\binom{[n]}{d+1}$, $M_I$ selects a term of maximal
  modulus in $\tdet(\bar Q_I)$ and $\chi_{\q}(I)=\sgn(M_I) \prod_{(i,k)\in M_I}\sgn(\bar Q_{i,k})$.
\end{proposition}
\begin{proof}
  Since every $1\times1$ submatrix of $\bar Q$ is
  non-singular, $W=|\bar Q|$ has finite entries.
  Choose a perturbation $W'\in\R^{n\times(d+1)}$ satisfying:
  \begin{itemize}
  \item $W'$ is generic in the sense of \cref{sec:pointed-cmf};
  \item every strict inequality between matching weights at $W$
  is preserved at $W'$.
  \end{itemize}
  There are finitely many weight comparisons, so these conditions
  can be achieved by a sufficiently small perturbation avoiding
  the corresponding equality hyperplanes.
  For each $I\in\binom{[n]}{d+1}$, let $M_I$ be the unique
  maximizing matching for $\tper(W'_I)$.
  Then $\Lambda\coloneqq(M_I)_I$ is coherent, realized by $W'$,
  and the second condition ensures that each $M_I$ also
  maximizes $\tper(W_I)$. Now, fix $I\in\binom{[n]}{d+1}$. Since $\q$ is in min-general position,
  $\tdet(\bar Q_I)$ is not balanced. Thus all its terms of maximal
  modulus, including the permutation sign, have the same nonzero sign.
  The term selected by $M_I$ therefore has sign $\chi_{\q}(I)$,
  which gives the stated formula.
\end{proof}

Now, the announced upper bound follows from~\cref{prop:ot-from-cmf-signs} and the upper bound of \cref{thm:main-unsigned}.

\begin{proposition}
  For every fixed $d \geq 1$, as $n \to \infty$,
  \[\log_2 \card{\OTT(n,d)} \;\leq\; d^2\,n\log_2 n \;-\; \bigl((d+1)\log_2(d!) - d\bigr)\,n \;+\; \Oh(1).\]
\end{proposition}
\begin{proof}
  Fix $\Lambda\in\CMFT(n,d+1)$, and let
  $\operatorname{supp}(\Lambda)\coloneqq\bigcup_I M_I$
  be the set of positions used by its matchings.
  By \cite[Corollary~3.4]{SturmfelsZelevinski93},
  this support contains exactly $n-d$ positions in each column.
  The homogenizing column has sign $+1$, so there are at most
  $2^{d(n-d)}$ sign patterns on the remaining used positions.
  By \cref{prop:ot-from-cmf-signs}, these signs together with
  $\Lambda$ determine the chirotope. Consequently, $\card{\OTT(n,d)}
    \leq 2^{d(n-d)}\,\card{\CMFT(n,d+1)}$. Taking logarithms and applying \cref{cor:cmf-log-upper-bound}
  with $r=d+1$ gives
  \begin{align*}
    \log_2\card{\OTT(n,d)}
    &\leq d(n-d)+\log_2\card{\CMFT(n,d+1)}\\
    &\leq d^2n\log_2 n
      -\bigl((d+1)\log_2(d!)-d\bigr)n+\Oh(1).
  \end{align*}
  Here the constant term $-d^2$ is absorbed into $\Oh(1)$,
  since $d$ is fixed.
\end{proof}

\section{The lower bound of Theorem~\ref{thm:main-signed}}\label{sec:lower-signed}

In this section we prove the lower bound on $\#\OTT(n,d)$. We use the geometric interpretation of signed tropical chirotopes presented in \Cref{s:signed chirotopes and hyperplanes} and proceed as in the unsigned setting: the key step is to construct a point sequence that admits many 1-point extensions with distinct signed tropical chirotopes. Here also we control these extensions through a \emph{signed auxiliary arrangement} (\cref{sec:signed-extension-step}): extension points with distinct sign vectors relative to it yield distinct chirotopes.

\subsection{The signed auxiliary arrangement and the extension step}
\label{sec:signed-extension-step}

Let $t \geq d$, and let \(\q=(q^1,\ldots,q^t)\) be a point sequence in min-general position in \(\strop^d\). For \(I=\{i_1<\cdots<i_d\}\in\binom{[t]}d\), let \(H_{\q^I}\) be
the unique signed max-tropical hyperplane spanned by
\(\q^I\) (\cref{lem:signed-cramer-intersection}).
We define the \emph{signed auxiliary arrangement} of \(\q\) by $\mathcal B=\mathcal B^{\pm}(\q)
\coloneqq \bigl(H_{\q^I}\bigr)_{I\in\binom{[t]}d}$. The \emph{sign vector} relative to $\mathcal B$ of a point \(x\) not on any of these hyperplanes is
  \[\sigma^{\mathcal B}(x) \coloneqq \bigl(\sgn\tdet(\q^I,x)\bigr)_{I\in\binom{[t]}d}\in\{\pm1\}^{\binom td}.\]
  As explained in \Cref{s:signed chirotopes and hyperplanes}, the entries of $\sigma^{\mathcal B}(x)$ record the position of \(x\) with respect to the hyperplanes spanned by every \(d\) points of \(\q\). The sign vector plays the role that the type \(T^{\mathcal B}\) plays in \cref{sec:unsigned-extension-step}.

\medskip

A \emph{generic extension} of \(\q\) is a sequence \(\v=(q^1,\ldots,q^t,v^1,\ldots,v^k)\), for some \(k\geq1\), that is in min-general position. Here is the signed analogue of Lemma~\ref{lem:generic-extension-existence}:

\begin{lemma}\label{lem:signed-generic-extension-existence}
Let \(\q\) be a (possibly empty) point sequence in min-general
position in \(\strop^d\). Let $k \geq 1$, \(s^1,\ldots,s^k\in\{\pm1\}^d\), and \(O \coloneqq O_1\times\cdots\times O_k \subset\mathcal C_{s^1}\times\cdots\times\mathcal C_{s^k}\) a nonempty open set. There exists a generic extension \((\q,v^1,\ldots,v^k)\)
with \((v^1,\ldots,v^k)\in O\).
\end{lemma}
\begin{proof}
For \(k=1\), expand each new tropical minor along the row
corresponding to \(v^1\). The cofactors in each expansion
are nonzero by the min-general position of \(\q\). The tropical minor is
balanced only if two expansion terms attain
the same maximal modulus. The equality of their moduli
defines a proper affine hyperplane in modulus coordinates.
Thus, when $v^1$ is chosen outside a finite union of such hyperplanes, every
new minor is non-singular. Since this union cannot
cover \(O\), the required extension exists.
For general \(k\), choose
\(v^i\in O_i\) successively using the case \(k=1\).
\end{proof}

\begin{lemma}
\label{lem:signed-extension}
  Let \(\q\) be in min-general position, and let
  \(\mathcal B=\mathcal B^{\pm}(\q)\). Let \(\v\) and \(\w\) be generic
  extensions obtained by appending \((v^1,\ldots,v^k)\) and
  \((w^1,\ldots,w^k)\) to $\q$, respectively. If $\bigl(\sigma^{\mathcal B}(v^1),\ldots,\sigma^{\mathcal B}(v^k)\bigr) \ne \bigl(\sigma^{\mathcal B}(w^1),\ldots,\sigma^{\mathcal B}(w^k)\bigr)$, then \(\v\) and \(\w\) have distinct signed tropical chirotopes.
\end{lemma}
\begin{proof}
  Let $\ell \in [k]$ denote some index such that $\sigma^{\mathcal B}(v^\ell) \neq \sigma^{\mathcal B}(w^\ell)$. This means that there exists \(I=\{i_1<\cdots<i_d\}\in\binom{[t]}{d}\) such that $\bigl(\sigma^{\mathcal B}(v^\ell)\bigr)_I \ne\bigl(\sigma^{\mathcal B}(w^\ell)\bigr)_I$. In particular
  \[
    \chi_{\v}(i_1,\ldots,i_d,t+\ell)
    =\bigl(\sigma^{\mathcal B}(v^\ell)\bigr)_I
    \ne\bigl(\sigma^{\mathcal B}(w^\ell)\bigr)_I
    =\chi_{\w}(i_1,\ldots,i_d,t+\ell)
  \]
  and $\chi_{\v} \neq \chi_{\w}$.
\end{proof}

\subsection{Point sequences with many extensions}\label{sec:basic-config-signed}

We now construct a point sequence whose signed auxiliary arrangement,
as defined in \cref{sec:signed-extension-step}, realizes many sign
vectors. As in \cref{sec:basic-config-unsigned}, we construct a product
grid from spanned hyperplanes.

\medskip

Let \(\mathcal C_1\) and \(\mathcal C_2\) be the sign charts indexed by \((+1,\ldots,+1)\) and \(\bigl((-1)^{d+1},\ldots,(-1)^{d+1}\bigr)\) in \(\{\pm1\}^d\), respectively. (They coincide when \(d\) is odd.) We place our points in the sign chart \(\mathcal C_1\) so that their signed auxiliary arrangement realizes many sign vectors in
\(
  \mathcal W\coloneqq
  \{x\in\mathcal C_2:|x|\in(-1,1)^d\}.
\)
\cref{prop:unsigned-local-product-grid} has the following signed analogue:

\begin{proposition}\label{prop:signed-local-product-grid}
  For every \(m\geq1\) there exist \(m\)-element sets \(A_s^j\subset\mathcal C_1\), for \(j,s\in[d]\), such that:
  \begin{itemize}
  \item[(i)] $\bigcup_{j,s\in[d]} A_s^j$ is a set of \(d^2m\) points in min-general position,
  \item[(ii)] for every \(j\in[d]\) and every \(\q=(q^1,\ldots,q^d)\in A_1^j\times\cdots\times A_d^j\), there exists some \(c_j(\q)\in(-1,1)\) such that the signed max-tropical hyperplane spanned by \(\q\) intersects \(\mathcal W\) in \(\{x\in\mathcal W:|x_j|=c_j(\q)\}\),
  \item[(iii)] for every \(j\in[d]\), as \(\q\) ranges over \(A_1^j\times\cdots\times A_d^j\), the \(m^d\) numbers \(c_j(\q)\) are pairwise distinct.
  \end{itemize}
\end{proposition}

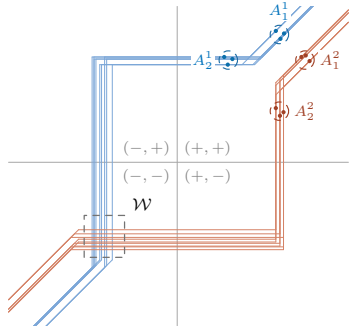
\begin{figure}[htbp]
  \centering
  \resizebox{!}{4.3cm}{\begin{tikzpicture}[
    x=0.85cm, y=0.85cm,
    >=latex,
    line cap=round, line join=round,
    font=\small,
    axis/.style={black!35, line width=0.45pt},
    chartlabel/.style={font=\scriptsize, text=black!50, fill=white, inner sep=0.7pt},
    secantB/.style={RoyalBlue!50,           line width=0.45pt},
    secantR/.style={BrickRed!55!white,      line width=0.45pt},
    clusterB/.style={RoyalBlue!80!black, dashed, line width=0.55pt},
    clusterR/.style={BrickRed!85!black, dashed, line width=0.55pt},
    apexB/.style={circle, fill=RoyalBlue!85!black, inner sep=0pt, minimum size=2.0pt},
    apexR/.style={circle, fill=BrickRed!85!black,  inner sep=0pt, minimum size=2.0pt},
    ptlabel/.style={font=\footnotesize, fill=white, inner sep=1pt, outer sep=2pt},
    clusterlabel/.style={font=\scriptsize, fill=white, inner sep=0.7pt, outer sep=1pt},
    zoombox/.style={dashed, black!55, line width=0.6pt},
  ]

  \begin{scope}
    \clip (-3.3,-3.2) rectangle (3.35,3.05);

    \draw[axis] (-3.3,0) -- (3.35,0);
    \draw[axis] (0,-3.2) -- (0,3.05);
    \node[chartlabel, anchor=south west] at (0.10,0.10) {$(+,+)$};
    \node[chartlabel, anchor=south east] at (-0.10,0.10) {$(-,+)$};
    \node[chartlabel, anchor=north west] at (0.10,-0.10) {$(+,-)$};
    \node[chartlabel, anchor=north east] at (-0.10,-0.10) {$(-,-)$};

    \foreach \ax/\ay in {1.93/2.57, 2.09/2.52, 2.02/2.415} {
      \foreach \bx/\by in {0.92/2.06, 1.08/2.02, 0.99/1.91} {
        \pgfmathsetmacro{\ux}{\by + \ax - \ay}
        \pgfmathsetmacro{\uy}{\by}
        \draw[secantB] (\ux,\uy) -- (3.35,{\uy+3.35-\ux});
        \draw[secantB] (-\ux,\uy) -- (\ux,\uy);
        \draw[secantB] (-\ux,-\uy) -- (-\ux,\uy);
        \draw[secantB] (-\ux,-\uy) -- (-3.3,{-\uy-(3.3-\ux)});
      }
    }

    \foreach \ax/\ay in {2.43/2.06, 2.59/1.98, 2.515/2.09} {
      \foreach \bx/\by in {1.93/1.06, 2.08/0.98, 2.01/0.91} {
        \pgfmathsetmacro{\ux}{\bx}
        \pgfmathsetmacro{\uy}{\bx - \ax + \ay}
        \draw[secantR] (\ux,\uy) -- (3.35,{\uy+3.35-\ux});
        \draw[secantR] (\ux,-\uy) -- (\ux,\uy);
        \draw[secantR] (-\ux,-\uy) -- (\ux,-\uy);
        \draw[secantR] (-\ux,-\uy) -- (-3.3,{-\uy-(3.3-\ux)});
      }
    }
  \end{scope}

  \draw[clusterB] (2,2.5) circle[radius=0.18];
  \foreach \ax/\ay in {1.93/2.57, 2.09/2.52, 2.02/2.415}
    \node[apexB] at (\ax,\ay) {};
  \node[RoyalBlue, clusterlabel, anchor=south] at (2,2.72) {$A_1^1$};

  \draw[clusterB] (1,2) circle[radius=0.18];
  \foreach \bx/\by in {0.92/2.06, 1.08/2.02, 0.99/1.91}
    \node[apexB] at (\bx,\by) {};
  \node[RoyalBlue, clusterlabel, anchor=east] at (0.78,2) {$A_2^1$};

  \draw[clusterR] (2.5,2) circle[radius=0.18];
  \foreach \ax/\ay in {2.43/2.06, 2.59/1.98, 2.515/2.09}
    \node[apexR] at (\ax,\ay) {};
  \node[BrickRed!85!black, clusterlabel, anchor=west] at (2.72,2) {$A_1^2$};

  \draw[clusterR] (2,1) circle[radius=0.18];
  \foreach \bx/\by in {1.93/1.06, 2.08/0.98, 2.01/0.91}
    \node[apexR] at (\bx,\by) {};
  \node[BrickRed!85!black, clusterlabel, anchor=west] at (2.22,1) {$A_2^2$};

  \draw[zoombox] (-1.82,-1.86) rectangle (-1.03,-1.04);
  \node[ptlabel, anchor=south west] at (-1.00,-1.05) {$\mathcal{W}$};

\end{tikzpicture}}
  \caption{The signed product grid of \cref{prop:signed-local-product-grid}
  for \(d=2\) and \(m=3\). The sets \(A_1^1,A_2^1\) (blue) and
  \(A_1^2,A_2^2\) (red) lie in the positive chart \(\mathcal C_1\). The
  signed max-tropical lines spanned by pairs in \(A_1^1\times A_2^1\) and in
  \(A_1^2\times A_2^2\) meet the box \(\mathcal W\subset\mathcal C_2\) in
  \(m^2\) vertical and \(m^2\) horizontal segments, forming the product grid.}
  \label{fig:signed-product-grid}
\end{figure}

Before proving the proposition, we derive its counting consequence.

\begin{corollary}\label{cor:signed-sign-vector-count}
  Let \(d\geq1\) be fixed. For every \(t\geq d^2\), there is a \(t\)-point
  sequence \(\a\) in min-general position in \(\strop^d\) whose signed
  auxiliary arrangement \(\mathcal B^{\pm}(\a)\) has
  \(\Omega(t^{d^2})\) sign vectors.
\end{corollary}
\begin{proof}
  First suppose that \(t=d^2m\) for an integer \(m\geq1\). Let \(A_s^j\) be the sets constructed in \cref{prop:signed-local-product-grid}. For every \(j\in[d]\), the \(m^d\) selections \(\q\in A_1^j\times\cdots\times A_d^j\) determine \(m^d\) signed max-tropical hyperplanes that cut \(\mathcal W\) along \(m^d\) parallel affine hyperplanes \(|x_j|=c_j(\q)\). These cuts for all $j \in [d]$ decompose \(\mathcal W\) into \((m^d+1)^d\) boxes.

Let \(\a\) be any ordering of the \(t\) points of
\(\bigcup_{j,s\in[d]}A_s^j\); it is in min-general position.
Any two boxes lie on different sides of some hyperplane of \(\mathcal B^{\pm}(\a)\).
By \Cref{s:signed chirotopes and hyperplanes}, points in different boxes
therefore have distinct sign vectors relative to \(\mathcal B^{\pm}(\a)\).
Thus \(\mathcal B^{\pm}(\a)\) realizes at least
\((m^d+1)^d\geq m^{d^2}=\Omega(t^{d^2})\)
distinct sign vectors.

  Now let \(t\geq d^2\) be arbitrary, and put \(m\coloneqq\lfloor
  t/d^2\rfloor\) and \(t_0\coloneqq d^2m\), so that \(t/2\leq t_0\leq t\). Let \(\a_0\) be the
  \(t_0\)-point sequence constructed above for this \(m\). If \(t>t_0\), then by
  \cref{lem:signed-generic-extension-existence} we may extend \(\a_0\) to a
  \(t\)-point sequence \(\a\) in min-general position; otherwise we put \(\a\coloneqq\a_0\). Since
  \(\mathcal B^{\pm}(\a)\) contains all hyperplanes of
  \(\mathcal B^{\pm}(\a_0)\), it realizes at least
  \(\Omega(t_0^{d^2})=\Omega(t^{d^2})\) sign vectors.
\end{proof}

The proof of \cref{prop:signed-local-product-grid} follows the same approach as \cref{sec:basic-config-unsigned}. The unsigned construction uses min-tropical hyperplanes, whereas \(\mathcal B^{\pm}\) consists of signed max-tropical hyperplanes.\footnote{The reader may wonder why our auxiliary arrangements use min-hyperplanes in the unsigned setting and max-hyperplanes in the signed setting. Recall that in the unsigned setting we work with combinatorial types of sequences of max-hyperplanes whereas in the signed setting we work with combinatorial types of sequences of points (which are easier to define from scratch). A signed point sequence \(\q=(q^1,\ldots,q^t)\) is dual to a sequence of signed max-hyperplanes with apices \((q^1)^{\odot-1},\ldots,(q^t)^{\odot-1}\). If we also started from max-hyperplane sequences in the signed setting, the two constructions could be presented in a unified way.}
Negating modulus coordinates turns the minimum conditions into maximum conditions, so we reflect the unsigned basic configuration and reuse its geometric estimates.

\medskip

As in \cref{sec:basic-config-unsigned}, we put \(r\coloneqq d+1\) and represent every point \(x\in\TPpmfin^d\) by its homogeneous representative with \(x_r=\tone\).

\medskip

Fix \(j\in[d]\) and let \(j=i_0,i_1,\ldots,i_d=r\) be the ordering of \([r]\) fixed in \cref{sec:basic-config-unsigned}. With the same vectors \(e_1,\ldots,e_d\) and \(e_r=-(e_1+\cdots+e_d)\), we define \(\bar u^j\coloneqq-4(e_j+e_r)\) and, for every \(s\in[d]\), \(\bar p^{j,s}\coloneqq\bar u^j+4e_{i_{s-1}}+4e_{i_s}\). Let \(\bar q^{j,s}\in\mathcal C_1\) have modulus vector \(\bar p^{j,s}\). Put \(\rho\coloneqq1/(2r)\) and \(\mathcal V_s^j\coloneqq\{q\in\mathcal C_1:\|\,|q|-\bar p^{j,s}\,\|_\infty<\rho\}\). For
\(\q=(q^1,\ldots,q^d)\in\mathcal V_1^j\times\cdots\times\mathcal V_d^j\),
write \(p^s\coloneqq|q^s|\) (so \(p^s_r=0\)), and put
\begin{equation}\label{eq:signed-telescope}
  c_j(\q)\coloneqq\sum_{s=1}^{d}
     \pth{p^s_{i_{s-1}}-p^s_{i_s}}.
\end{equation}

\begin{lemma}
\label{lem:signed-local-coordinate-slices}
  Any sequence \(\q=(q^1,\ldots,q^d)\in\mathcal V_1^j\times\cdots\times\mathcal V_d^j\) in min-general position spans a signed max-tropical hyperplane whose intersection with \(\mathcal W\) is nonempty and equal to \(\{x\in\mathcal W:|x_j|=c_j(\q)\}\).
\end{lemma}

\begin{proof}
  Let \(u\in\R^r\) be the vector determined by \(u_r=0\) and \(u_{i_{s-1}}-u_{i_s}=p^s_{i_{s-1}}-p^s_{i_s}\) for \(s\in[d]\); summing these equations gives \(u_j=c_j(\q)\).
  The vectors \(-\bar u^j\), \(-\bar p^{j,s}\), \(-p^s\) and \(-u\) play the roles of \(\bar u^j\), \(\bar p^{j,s}\), \(p^s\) and \(u\) in the proof of \cref{lem:unsigned-local-coordinate-slices}, so negating the estimates in that proof gives:
  \begin{enumerate}[label=(\alph*)]
    \item\label{it:signed-slice-points} for every \(s\in[d]\), \(\operatorname*{arg\,max}_{k\in[r]}(p^s_k-u_k)=\{i_{s-1},i_s\}\);
    \item\label{it:signed-slice-apex} \(u_j\in(-1,1)\);
    \item\label{it:signed-slice-window} for every \(x\in\mathcal W\), the maximum of \(|x_k|-u_k\) over \(k\in[r]\) is attained only at indices \(k\in\{j,r\}\).
  \end{enumerate}

  Let \(w\in\TPpmfin^d\) be given by \(w_{i_s}=\bigl((-1)^{d-s},u_{i_s}\bigr)\), \(s=0,\ldots,d\). Then \(q^s\in\Hsmax(w)\) for every \(s\in[d]\): by~\ref{it:signed-slice-points}, the maximum of \(|q^s_k|-|w_k|\) over \(k\in[r]\) is attained precisely at \(i_{s-1}\) and \(i_s\), and since \(q^s\in\mathcal C_1\), the corresponding term signs \(\sgn(w_{i_{s-1}})\sgn(q^s_{i_{s-1}})=(-1)^{d-s+1}\) and \(\sgn(w_{i_s})\sgn(q^s_{i_s})=(-1)^{d-s}\) are opposite. Since \(\q\) is in min-general position, \cref{lem:signed-cramer-intersection} shows that \(\Hsmax(w)\) is the unique hyperplane spanned by \(\q\).

  Now let \(x\in\mathcal W\). By~\ref{it:signed-slice-window}, only the terms indexed by \(j\) and \(r\) can have maximal modulus in \(\bigoplus_{k=1}^r w_k^{\odot-1}\odot x_k\). Since \(x\in\mathcal C_2\), these two terms are \(\bigl(-1,|x_j|-u_j\bigr)\) and \(\tone\). Their signs are opposite, so the sum is balanced exactly when \(|x_j|=u_j=c_j(\q)\). Hence
  \[
    \Hsmax(w)\cap\mathcal W=\{x\in\mathcal W:|x_j|=c_j(\q)\},
  \]
  which is nonempty by~\ref{it:signed-slice-apex}.
\end{proof}

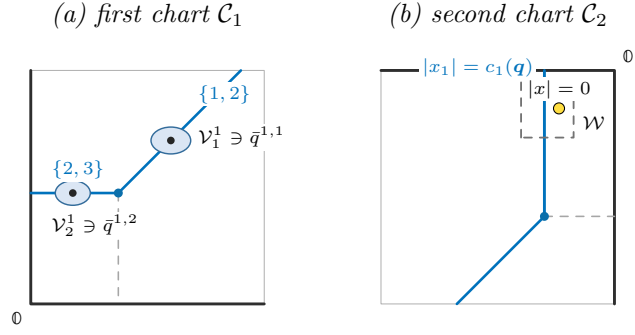
\begin{figure}[t]
  \centering
  \resizebox{!}{4.4cm}{\begin{tikzpicture}[
    x=0.78cm, y=0.78cm,
    >=latex,
    line cap=round, line join=round,
    font=\small,
    lineB/.style={RoyalBlue, line width=1pt},
    localB/.style={RoyalBlue, line width=1.1pt},
    neighB/.style={draw=RoyalBlue!70!black, fill=RoyalBlue!12, line width=0.55pt},
    ghost/.style={black!35, line width=0.6pt, dashed},
    apex/.style={circle, fill=RoyalBlue!85!black, inner sep=0pt, minimum size=3.4pt},
    pointpt/.style={circle, fill=black!85, inner sep=0pt, minimum size=2.9pt},
    intpt/.style={circle, fill=Goldenrod, draw=black, line width=0.4pt,
                  inner sep=0pt, minimum size=4.2pt},
    bdry/.style={line width=1.1pt, black!82, line cap=round},
    chartedge/.style={draw=black!30, thin},
    vaxis/.style={<->, black!50, line width=0.6pt},
    boxB/.style={draw=black!55, line width=0.7pt, dashed},
    ptlabel/.style={font=\scriptsize, fill=white, inner sep=1pt, outer sep=1pt},
    chartlabel/.style={font=\scriptsize, text=black!65},
    panellab/.style={font=\itshape, anchor=south},
  ]

  \begin{scope}
    \node[panellab] at (2,4.55) {(a) first chart \(\mathcal C_1\)};
    \draw[chartedge] (0,0) rectangle (4,4);
    \draw[bdry] (4,0) -- (0,0) -- (0,4);
    \node[ptlabel, anchor=north east] at (-0.05,-0.05) {$\tzero$};

    \coordinate (a) at (1.5,1.9);
    \draw[ghost] (a) -- (1.5,0.12);
    \draw[lineB] (a) -- (3.6,4.0);
    \draw[lineB] (a) -- (0.0,1.9);
    \node[apex] at (a) {};

    \draw[neighB] (2.4,2.8) ellipse [x radius=0.34, y radius=0.23];
    \node[pointpt] at (2.4,2.8) {};
    \node[ptlabel, anchor=west] at (2.78,2.85) {$\mathcal V_1^1\ni\bar q^{1,1}$};

    \draw[neighB] (0.72,1.9) ellipse [x radius=0.3, y radius=0.21];
    \node[pointpt] at (0.72,1.9) {};
    \node[ptlabel, anchor=north] at (1.10,1.62) {$\mathcal V_2^1\ni\bar q^{1,2}$};

    \node[RoyalBlue, ptlabel, anchor=south west] at (2.78,3.35)
      {$\{1,2\}$};
    \node[RoyalBlue, ptlabel, anchor=south east] at (1.35,2.05)
      {$\{2,3\}$};
  \end{scope}

  \begin{scope}[shift={(6.0,0)}]
    \node[panellab] at (2,4.55) {(b) second chart \(\mathcal C_2\)};
    \draw[chartedge] (0,0) rectangle (4,4);
    \draw[bdry] (0,4) -- (4,4) -- (4,0);
    \node[ptlabel, anchor=south west] at (4.05,4.05) {$\tzero$};

    \coordinate (b) at (2.8,1.5);
    \draw[ghost] (b) -- (4.0,1.5);
    \draw[localB] (b) -- (2.8,4.0);
    \draw[lineB] (b) -- (1.3,0.0);
    \node[apex] at (b) {};

    \draw[boxB] (2.4,2.85) rectangle (3.3,3.8);
    \node[ptlabel, anchor=west] at (3.36,3.10) {$\mathcal W$};
    \node[intpt] at (3.05,3.35) {};
    \node[ptlabel, anchor=south] at (3.05,3.43) {$|x|=0$};
    \node[RoyalBlue, ptlabel, anchor=south east] at (2.7,3.80) {$|x_1|=c_1(\q)$};
  \end{scope}

\end{tikzpicture}}
  \caption{\Cref{lem:signed-local-coordinate-slices} for \(d=2\), \(j=1\):
  spanning points in \(\mathcal C_1\), and the signed line they span, which
  meets \(\mathcal W\subset\mathcal C_2\) in the segment
  \(\{x\in\mathcal W:|x_1|=c_1(\q)\}\).}
  \label{fig:signed-local-coordinate-slice}
\end{figure}

\begin{proof}[Proof of \cref{prop:signed-local-product-grid}]
  For each \(j\in[d]\), we let \(\{\mathcal V_s^j\}_{s\in[d]}\) denote the sets constructed for \cref{lem:signed-local-coordinate-slices}, and pick in each \(\mathcal V_s^j\) an \(m\)-element subset \(A_s^j\) such that the values \(c_j(\q)\), as \(\q\) ranges over \(A_1^j\times\cdots\times A_d^j\), are pairwise distinct (here we use the proof of \cref{lem:unsigned-distinct-slices}, which applies verbatim since \(c_j(\q)\) is given by the same linear formula~\eqref{eq:signed-telescope} in the moduli of the points). Sufficiently small perturbations of the points of \(\cup_{j,s\in[d]}A_s^j\) keep them in their sets \(\mathcal V_s^j\) and preserve, for each \(j\), the pairwise distinctness
  of the values \(c_j(\q)\), as \(\q\) ranges over \(A_1^j\times\cdots\times A_d^j\). By \cref{lem:signed-generic-extension-existence}, applied to the empty sequence, we may therefore perturb all \(d^2m\) points into min-general position. This gives~(i) and~(iii), and~(ii) follows from \cref{lem:signed-local-coordinate-slices}.
\end{proof}

\subsection{Completing the proof of Theorem~\ref{thm:main-signed}}

We can now prove the lower bound of Theorem~\ref{thm:main-signed}.

\begin{proposition}\label{prop:signed-lower-bound}
  As $n \to \infty$, for every fixed $d\geq 1$,
  \[\log_2 \#\OTT(n,d) \ge d^2\,n\log_2 n -d^2\,n\log_2\log_2 n -\Oh(n).\]
\end{proposition}
\begin{proof}
  For sufficiently large \(n\), set $t=\left\lfloor\frac{n}{\log_2 n}\right\rfloor$. By \cref{cor:signed-sign-vector-count}, there is a \(t\)-point sequence \(\a\) in min-general position in \(\strop^d\) whose signed auxiliary
  arrangement \(\mathcal B=\mathcal B^{\pm}(\a)\) realizes a set \(\Sigma\) of at least
  \(c_d\,t^{d^2}\) sign vectors, for some constant \(c_d>0\)
  depending only on \(d\). For each sequence of sign vectors \(\sigma_1,\ldots,\sigma_{n-t}\in\Sigma\), we choose \((v^1,\ldots,v^{n-t})\) with \(\sigma^{\mathcal B}(v^\ell)=\sigma_\ell\) so that \((\a,v^1,\ldots,v^{n-t})\) is in min-general position; this is possible by \cref{lem:signed-generic-extension-existence}, since each sign vector in \(\Sigma\) is attained on a nonempty open subset of a sign chart. By \cref{lem:signed-extension}, distinct sign-vector sequences give extensions with distinct signed tropical chirotopes, so we obtain at least
  \(\card{\Sigma}^{\,n-t}\) distinct chirotopes. Therefore $\card{\OTT(n,d)} \;\geq\; \card{\Sigma}^{\,n-t} \;\geq\; \bigl(c_d\,t^{d^2}\bigr)^{n-t}$, and taking logarithms, using \(t=n/\log_2 n-\Oh(1)\), gives $\log_2\card{\OTT(n,d)} \geq d^2\,n\log_2 n -d^2\,n\log_2\log_2 n-\Oh(n)$ as claimed.
\end{proof}

\section{Parity coherent matching fields and parity labeled order types}\label{sec:parity}

In this section, we introduce parity coherent matching fields and parity
labeled order types. These are the coherent matching fields realized by weight
matrices with entries of the form \(\eta s^w\), where \(\eta\in\{\pm1\}\),
\(w\in\Z\) and \(s\) is sufficiently large, and the signed tropical 
labeled order types of point sequences whose coordinates have moduli of the
same form. We then prove \cref{thm:main-parity}.

\medskip

As explained in \cref{sec:signed-real-limits}, signed tropical numbers
describe the asymptotic behavior of signed monomials~\(\varepsilon t^a\). Signed parity numbers similarly describe the
behavior of double exponentials: a coordinate
\((\varepsilon,\eta s^w)\) stands for the real number
\(\varepsilon t^{\eta s^w}\) with \(t\gg s\gg1\), and two successive
tropicalizations recover its data. Letting \(t\to\infty\) with \(s\)
fixed gives the signed tropical number \((\varepsilon,\eta s^w)\),
and then letting \(s\to\infty\) records the parity \(\eta\) and the
priority \(w\) of its modulus. We call this \emph{double tropicalization}.

\subsection{Parity numbers and comparison of weights}\label{sec:parity-numbers}

The set of \emph{parity numbers} is $\P\coloneqq\{0\}\cup\{\eta s^w:\eta\in\{\pm1\},\ w\in\Z\}$, where \(\eta s^w\) denotes the function \(s\mapsto\eta s^w\) on \((1,\infty)\)
and \(0\) denotes the zero function. For a nonzero parity number, we call
\(\eta\) its \emph{parity}\footnote{The terminology is motivated by parity games, where the standard encoding
as mean-payoff games assigns the payment \((-1)^w s^w\)
to an integer priority \(w\), for sufficiently large \(s\).
Here we allow \(\eta\) and \(w\) to be chosen independently,
but retain the term ``parity'' for \(\eta\).}
and \(w\) its \emph{priority}, which is
the exponent of \(s\). For each fixed \(s>1\), a parity number evaluates to a real number.

\medskip

Throughout this section, \(\sum\) denotes ordinary addition, that is,
the tropical product \(\odot\). A sum of parity numbers need not be a
parity number, so we work in the ring \(\Z[s^{\pm1}]\supseteq\P\) of finite sums
of integer powers of \(s\) with integer coefficients (Laurent
polynomials). For a nonzero \(A\in\Z[s^{\pm1}]\), let \(\sgn A\) be the
sign of its leading coefficient; it is the sign of \(A(s)\) for all
sufficiently large \(s\). We order \(\Z[s^{\pm1}]\) by \(A<B\) if
\(\sgn(B-A)=+1\), equivalently if \(A(s)<B(s)\) for all sufficiently
large \(s\). This order is total. On parity numbers it reads
\begin{equation}\label{eq:parity-order}
  \eta s^w<\eta' s^{w'}
  \quad\Longleftrightarrow\quad
  \begin{cases}
    \eta=-1,\ \eta'=+1, &\text{or}\\
    \eta=\eta'=+1,\ w<w', &\text{or}\\
    \eta=\eta'=-1,\ w>w',
  \end{cases}
\end{equation}
and \(0\) lies between the parity numbers of parity \(-1\) and those of
parity \(+1\).

\medskip

We will use the following simple lemma.

\begin{lemma}\label{lem:parity-comparison}
Let \(E\) be a finite index set, and let \((a_e)_{e\in E}\) be a family of nonzero parity numbers,
with \(a_e=\eta_e s^{w_e}\) and pairwise distinct priorities \(w_e\). For distinct subsets
\(F,F'\subseteq E\), let \(e^\star\) be the element of
\(F\mathbin{\triangle}F'\) with largest priority. Then
\begin{equation}\label{eq:parity-comparison}
  \sgn\Bigl(\sum_{e\in F}a_e-\sum_{e\in F'}a_e\Bigr)
  =\eta_{e^\star}\bigl(\mathbf{1}_{F}(e^\star)-\mathbf{1}_{F'}(e^\star)\bigr).
\end{equation}
In particular, these subset sums are pairwise distinct, and there is
\(s_0\) such that, for every \(s>s_0\), their values at \(s\) are ordered
as in \(\Z[s^{\pm1}]\).
\end{lemma}
\begin{proof}
The terms indexed by \(F\cap F'\) cancel. The remaining terms have
pairwise distinct exponents, and the largest one is \(w_{e^\star}\), with
coefficient the right-hand side of \eqref{eq:parity-comparison}. There
are only finitely many pairs of subsets, so one threshold \(s_0\)
suffices for the last assertion.
\end{proof}

\subsection{Parity coherent matching fields}\label{sec:parity-cmf}

An \emph{\(n\times r\) parity matrix} is a matrix
\(A=(a_{i,k})\in\P^{n\times r}\) with nonzero entries
\(a_{i,k}=\eta_{i,k}s^{w_{i,k}}\) whose \(nr\) priorities \(w_{i,k}\)
are pairwise distinct.
For \(I\in\binom{[n]}r\) and \(M\in\mathcal M_I\), we define the weight of \(M\) with respect to \(A\) as the Laurent polynomial $\omega_A(M)\coloneqq\sum_{(i,k)\in M}a_{i,k}$. (Here we use the matching notation of \cref{sec:contributions}.) By \cref{lem:parity-comparison},
distinct matchings in \(\mathcal M_I\) have distinct weights, so there
is a unique matching \(M_I\) of largest weight (in the order on
Laurent polynomials). For a parity matrix \(A\), let
\(\Lambda(A)\coloneqq(M_I)_{I\in\binom{[n]}r}\), where \(M_I\) is the
matching of largest weight \(\omega_A\) in \(\mathcal M_I\). A
matching field of type \((n,r)\) is \emph{parity coherent} if it
equals \(\Lambda(A)\) for some \(n\times r\) parity matrix \(A\). We
write \(\CMFP(n,r)\) for the set of parity coherent matching fields of
type~\((n,r)\).

\medskip

 For fixed \(s>1\), evaluating the entries of \(A\) gives a real matrix
 \(A(s)\). For each \(I\in\binom{[n]}r\), the matchings in
 \(\mathcal M_I\) maximizing \(\sum_{(i,k)\in M}a_{i,k}(s)\) need not be
 unique, and they may change as \(s\) grows. However, by the last
 assertion of \cref{lem:parity-comparison}, they stabilize for all
sufficiently large \(s\) and every \(I\), and coincide with the matching
 \(M_I\) defined by the order on Laurent polynomials. Thus \(\Lambda(A)\) is a coherent matching field, realized by
\(A(s)\) for all sufficiently large \(s\). Consequently,
\(\CMFP(n,r)\subseteq\CMFT(n,r)\) for every \(n\ge r\ge2\). It turns out that for every fixed \(r\ge3\), as \(n\to\infty\), the parity coherent
matching fields of type \((n,r)\) are actually a vanishingly small
subset of the coherent matching fields of type \((n,r)\). More precisely:
 
\begin{proposition}\label{prop:parity-cmf}
For all \(n\ge r\ge2\),
\begin{equation}\label{eq:parity-cmf-bounds}
  \bigl((n-r+1)!\bigr)^{r-1}
  \le\card{\CMFP(n,r)}
  \le r^n\bigl((r-1)n\bigr)!\,2^{(r-1)n}.
\end{equation}
\end{proposition}

\paragraph{The upper bound.}
The upper bound follows by identifying the data that determine all
maximizing matchings.

\begin{lemma}\label{lem:parity-minimal-data}
Let \(A\) be an \(n\times r\) parity matrix. For each row \(i\in[n]\), let
\(k_i\) be the column index of the entry of lowest priority in row \(i\),
that is, \(w_{i,k_i}<w_{i,k}\) for all \(k\neq k_i\). Let
\(E^\circ\coloneqq\{(i,k)\in[n]\times[r]:k\neq k_i\}\) be the set of all
other entries. Then \(\Lambda(A)\) is determined by the indices
\((k_i)_{i\in[n]}\), the order of the priorities on \(E^\circ\), and the
parities on \(E^\circ\).
\end{lemma}
\begin{proof}
Fix \(I\in\binom{[n]}{r}\) and distinct matchings \(M,M'\in\mathcal M_I\),
and let \(e=(i,k)\) be the edge of largest priority in
\(M\mathbin{\triangle}M'\). Since \(M\) and \(M'\) each contain exactly one
edge in row \(i\), the other one, \((i,k')\), also lies in
\(M\mathbin{\triangle}M'\); hence \(w_{i,k}>w_{i,k'}\ge w_{i,k_i}\), so
\(e\in E^\circ\).

Thus \(e\) is the edge of largest priority in
\((M\mathbin{\triangle}M')\cap E^\circ\), so it is determined by the
priority order on \(E^\circ\) alone. Its parity then tells us which of
\(M\) and \(M'\) has larger weight, by \cref{lem:parity-comparison}. We can therefore compare any two
matchings using only the stated data, and hence determine the unique
maximizing matching \(M_I\). Applying this argument to every row set
\(I\) determines the entire matching field \(\Lambda(A)\).
\end{proof}

\begin{proof}[Proof of the upper bound in \cref{prop:parity-cmf}]
There are \(r^n\) choices for \((k_i)_{i\in[n]}\). For each choice,
\(E^\circ\) has \((r-1)n\) elements, with at most
\(\bigl((r-1)n\bigr)!\) priority orders and \(2^{(r-1)n}\) parity
patterns. By \cref{lem:parity-minimal-data}, each such choice determines
at most one matching field. Multiplying gives the upper bound in
\eqref{eq:parity-cmf-bounds}.
\end{proof}

\paragraph{The lower bound.}
For the lower bound, we construct parity matrices from tuples of
permutations and recover these permutations from the resulting matching
fields. In this way, we show that \(\card{\CMFP(n,r)}\) is at least the
number of these tuples.

\begin{proof}[Proof of the lower bound in \cref{prop:parity-cmf}]
Put \(m=n-r+1\), so that \(n=(r-1)+m\), and let
\(\sigma_1,\ldots,\sigma_{r-1}\in\Sym_m\). Let \(A\) be an \(n\times r\)
parity matrix with all parities \(+1\) whose priorities satisfy:
\begin{itemize}
  \item[(a)] the diagonal entries \((j,j)\), \(j\in[r-1]\), have larger
    priorities than all other entries;\footnote{The first \(r-1\) rows
    of \(A\) play the role of the first \(r-1\) rows of the tropical
    identity \(I_r\) in the pointed augmentation of
    \cref{sec:pointed-cmf}.}
  \item[(b)] for each \(k\in[r-1]\) and \(a,b\in[m]\),
    \(w_{r-1+a,k}>w_{r-1+b,k}\) if and only if \(\sigma_k(a)>\sigma_k(b)\);
  \item[(c)] for each \(i\in[m]\) and \(k\in[r-1]\), the entry
    \((r-1+i,k)\) has larger priority than every entry in column \(r\).
\end{itemize}
For instance, we can take \(A=(s^{W_{i,k}})_{(i,k) \in [n]\times[r]}\) with
\[
  W=\begin{pNiceArray}{cccc|c}[first-row,margin]
    \rule{0pt}{5ex} & & & & \\
    rm+1 & \ast & \cdots & \ast & \ast\\
    \ast & rm+2 & \cdots & \ast & \ast\\
    \vdots & \vdots & \ddots & \vdots & \vdots\\
    \ast & \ast & \cdots & rm+r-1 & \ast\\
    \hline
    m+\sigma_1(1) & 2m+\sigma_2(1) & \cdots & (r-1)m+\sigma_{r-1}(1) & -1\\
    \vdots & \vdots & & \vdots & \vdots\\
    m+\sigma_1(m) & 2m+\sigma_2(m) & \cdots & (r-1)m+\sigma_{r-1}(m) & -m
    \CodeAfter
      \begin{tikzpicture}[decoration={brace,amplitude=4pt}]
        \draw[decorate] ([xshift=1mm,yshift=3.5mm]1-|1)
          -- ([xshift=-1mm,yshift=3.5mm]1-|5) node[midway,above=4pt] {\(r-1\)};
        \draw[decorate] ([xshift=1mm,yshift=3.5mm]1-|5)
          -- ([xshift=-1mm,yshift=3.5mm]1-|6) node[midway,above=4pt] {\(1\)};
        \draw[decorate] ([xshift=5mm,yshift=-0.5mm]1-|6)
          -- ([xshift=5mm,yshift=0.5mm]5-|6) node[midway,right=5pt] {\(r-1\)};
        \draw[decorate] ([xshift=5mm,yshift=-0.5mm]5-|6)
          -- ([xshift=5mm,yshift=0.5mm]8-|6) node[midway,right=5pt] {\(m\)};
      \end{tikzpicture}
  \end{pNiceArray}\hspace{3.5em},
\]
where the entries \(\ast\) are pairwise distinct integers smaller than
\(-m\).

\medskip

Now fix \(k\in[r-1]\) and \(1\le a<b\le m\), and consider the row set $I=I^k_{a,b}\coloneqq\bigl([r-1]\setminus\{k\}\bigr)\cup\{r-1+a,\,r-1+b\}$. By~(a) and \cref{lem:parity-comparison}, the matching \(M_I\) of
\(\Lambda(A)\) contains every diagonal entry \((j,j)\) with \(j\neq k\),
so it matches the rows \(r-1+a\) and \(r-1+b\) to the columns \(k\) and
\(r\). By~(c), the comparison of the two possible matchings is decided
by their entries in column \(k\), and by~(b), the entry \((r-1+a,k)\)
has the larger priority there exactly when \(\sigma_k(a)>\sigma_k(b)\).
Hence, by \cref{lem:parity-comparison}, \(M_I\) matches row \(r-1+a\) to
column \(k\) if and only if \(\sigma_k(a)>\sigma_k(b)\). Choosing
\(I^k_{a,b}\) for all pairs \(a<b\), the corresponding matchings
\(M_{I^k_{a,b}}\) determine \(\sigma_k\), and doing so for all
\(k\in[r-1]\) determines the tuple of permutations. The map
\[
  \Sym_m^{\,r-1}\longrightarrow\CMFP(n,r),
  \qquad(\sigma_1,\ldots,\sigma_{r-1})\longmapsto\Lambda(A)
\]
is therefore injective, which proves the lower bound.
\end{proof}

\begin{proof}[Proof of \cref{thm:main-parity}~(i)]
Fix \(r\ge2\). Since \(k!=k^k\,2^{-\Oh(k)}\) by Stirling's formula, the
bounds in \eqref{eq:parity-cmf-bounds} satisfy
\(\bigl((n-r+1)!\bigr)^{r-1}=n^{(r-1)n}\,2^{-\Oh(n)}\) and
\(r^n\bigl((r-1)n\bigr)!\,2^{(r-1)n}=n^{(r-1)n}\,2^{\Oh(n)}\).
\end{proof}

\subsection{Parity labeled order types}\label{sec:parity-ot}

We now pass from matching fields to labeled order types. As in the signed
tropical case of \cref{sec:signed-chirotopes}, we introduce signed
parity numbers, define the chirotope of a parity point sequence
through signed tropical determinants, and then relate it to signed
tropical chirotopes by evaluating at large \(s\).

\medskip

The set of \emph{signed parity numbers} is \(\Ppm\coloneqq\{\pm1\}\times\P\).
An element \((\varepsilon,a)\in\Ppm\) has \emph{outer sign} \(\varepsilon\) and
\emph{modulus}~\(a\). For \(a=\eta s^w\), it represents
\(\varepsilon t^{a}=\varepsilon t^{\eta s^w}\) and evaluates at \(s>1\) to
the signed tropical number \((\varepsilon,a(s))\in\strop\). We write
\(\tone=(+1,0)\) for the element of \(\Ppm\),
representing \(t^0=1\). We use the signed tropical arithmetic of
\cref{sec:signed-chirotopes}, with moduli in \(\Z[s^{\pm1}]\) under the
order of Laurent polynomials (see \cref{sec:parity-numbers}), and signed
tropical determinants are defined for matrices with entries in \(\Ppm\).

\medskip

A \emph{parity point sequence} is a sequence \(\q=(q^1,\ldots,q^n)\) of
points \(q^i\in\Ppm^d\) whose coordinates \(q^i_k\) have nonzero moduli and
pairwise distinct priorities \(w_{i,k}\), for \(i\in[n]\) and \(k\in[d]\).
Let \(\bar Q\) be the homogenized
matrix of \(\q\), whose \(i\)th row is \((q^i,\tone)\), and let
\(\varepsilon_{i,k}\) and \(a_{i,k}\) be the outer sign and the modulus of
\(q^i_k\). Notice that the moduli of \(\bar Q\) do not form a parity
matrix, since the entries in its last column are the same.
For \(I\in\binom{[n]}{d+1}\) and \(M\in\mathcal M_I\), the term of
\(\tdet(\bar Q_I)\) indexed by \(M\) has modulus equal to the weight
\[  \omega_{\q}(M)\coloneqq\sum_{(i,k)\in M,\ k\le d}a_{i,k}.\]
Since the entries \(a_{i,k}\) have pairwise distinct
priorities, there is a unique matching
\(M_I\) of largest weight, and \(\tdet(\bar Q_I)\) is not balanced, with
\begin{equation}\label{eq:parity-chirotope}
  \sgn\tdet(\bar Q_I)=\sgn(M_I)\prod_{(i,k)\in M_I,\ k\le d}\varepsilon_{i,k}.
\end{equation}
The same argument applies to every square submatrix of \(\bar Q\). The \emph{parity chirotope} of \(\q\) is the map \(\chi_{\q}\) that
associates to every \((d+1)\)-tuple \(I=(i_1,\ldots,i_{d+1})\) of
distinct indices in \([n]\) the sign
\[
  \chi_{\q}(I)=\sgn\tdet\begin{pmatrix}
    q^{i_1} & \tone\\
    \vdots & \vdots\\
    q^{i_{d+1}} & \tone
  \end{pmatrix}\in\{\pm1\}.
\]
Two parity point sequences are said to {\em have the same parity labeled order type} if and only if they have the same parity chirotope. We write \(\OTP(n,d)\) for the set of parity labeled order types. (Note that by definition, in a parity point sequence all coordinates have pairwise distinct priorities.)

\medskip

For fixed \(s>1\), the evaluated point sequence \(\q(s)\in(\strop^d)^n\)
need not be in min-general position, and its chirotope may change as
\(s\) grows. However, for all sufficiently large \(s\), each square
submatrix of \(|\bar Q(s)|\) has the same unique matching of largest
weight as the corresponding submatrix of \(|\bar Q|\). Hence, for all
sufficiently large \(s\), \(\q(s)\) is in min-general position and
\(\chi_{\q(s)}=\chi_{\q}\). Together with
\cref{prop:tropical-real-realizability}, this gives
\(\OTP(n,d)\subseteq\OTT(n,d)\subseteq\OTR(n,d)\) for every \(n>d\ge1\). Again, for every fixed \(d\ge2\) and as \(n\to\infty\), the labeled parity order types
of type \((n,d)\) are actually a vanishingly small subset of the
labeled tropical order types of type \((n,d)\). More precisely:

\begin{proposition}\label{prop:parity-ot}
For all \(d\ge1\) and \(n\ge d+1\),
\begin{equation}\label{eq:parity-ot-bounds}
  \bigl((n-d)!\bigr)^d
  \le\card{\OTP(n,d)}
  \le(dn)!\,4^{dn}.
\end{equation}
\end{proposition}

\paragraph{The upper bound.}
As for the parity matching fields, the upper bound follows by identifying the data
that determine all maximizing matchings, now together with the outer
signs.

\begin{proof}[Proof of the upper bound in \cref{prop:parity-ot}]
By \eqref{eq:parity-chirotope}, \(\chi_{\q}\) is determined by the
matchings \(M_I\) and the outer signs \(\varepsilon_{i,k}\). The matchings \(M_I\) are determined by the
priority order and the parities of the \(dn\) entries \(a_{i,k}\).
Hence \(\chi_{\q}\) is determined by the priority order, the parities
and the outer signs of the affine coordinates. There are at most \((dn)!\)
priority orders, \(2^{dn}\) parity patterns and \(2^{dn}\) outer-sign
patterns, which gives the upper bound in \eqref{eq:parity-ot-bounds}.
\end{proof}

\paragraph{The lower bound.}
We use the permutation construction for parity matching fields and
show that the signs of its maximizing matchings on selected row sets
suffice to recover all the permutations.

\begin{proof}[Proof of the lower bound in \cref{prop:parity-ot}]
Put \(m=n-d\ge1\), and fix
\(\sigma=(\sigma_1,\ldots,\sigma_d)\in\Sym_m^d\).
Let \(A=(a_{i,k})\) be the matrix constructed in the proof of the lower
bound in \cref{prop:parity-cmf}, with \(r=d+1\). Define
\(\q=(q^1,\ldots,q^n)\) by \(q^i_k=(+1,a_{i,k})\) for \(i\in[n]\) and
\(k\in[d]\). The moduli \(a_{i,k}\), for \(i\in[n]\) and \(k\in[d]\), have pairwise distinct priorities, so \(\q\) is a parity point sequence. Write
\(\chi^\sigma\coloneqq\chi_{\q}\). The matrix \(|\bar Q|\) is obtained
from \(A\) by replacing its last column with zeros.

\medskip

Fix \(k\in[d]\) and \(1\le a<b\le m\). To compare \(\sigma_k(a)\) and
\(\sigma_k(b)\), consider the row set $I=I^k_{a,b}\coloneqq\bigl([d]\setminus\{k\}\bigr)\cup\{d+a,\,d+b\}$. By construction, the diagonal entries \(a_{j,j}\),
\(j\in[d]\setminus\{k\}\), have the largest priorities and parity \(+1\). 
Thus, \(M_I\) uses all of them. There are only two ways to complete these \(d-1\) edges to a matching in \(\mathcal{M}_I\).
Let \(M_a\) be the one that matches row \(d+a\) to column \(k\) and row \(d+b\) to column
\(d+1\), and let \(M_b\) be the one that exchanges these two assignments. The common
diagonal entries cancel and the last
column contributes zero. Thus $\omega_{\q}(M_a)-\omega_{\q}(M_b)=a_{d+a,k}-a_{d+b,k}$. The entries \(a_{d+a,k}\) and \(a_{d+b,k}\) have parity \(+1\), and their priority order agrees with the order of \(\sigma_k(a)\) and
\(\sigma_k(b)\). Hence \(M_I=M_a\) if and only if \(\sigma_k(a)>\sigma_k(b)\).

\medskip

The matchings \(M_a\) and \(M_b\) differ by exchanging columns \(k\)
and \(d+1\), so their signs are opposite. Since all outer signs of
\(\q\) are \(+1\), \eqref{eq:parity-chirotope} gives
\(\chi^\sigma(I)=\sgn(M_I)\). Thus \(\chi^\sigma(I)\) determines
whether \(M_I=M_a\) or \(M_I=M_b\), and hence whether
\(\sigma_k(a)>\sigma_k(b)\). Taking all pairs \(a<b\) recovers
\(\sigma_k \in\Sym_m\), and doing so for every \(k\in[d]\) recovers \(\sigma\).
Consequently, \(\sigma\mapsto\chi^\sigma\) is injective, giving
\(\card{\OTP(n,d)}\ge(m!)^d=((n-d)!)^d\).
\end{proof}

\begin{proof}[Proof of \cref{thm:main-parity}~(ii)]
Fix \(d\ge1\). As in the proof of~(i), the bounds in
\eqref{eq:parity-ot-bounds} satisfy
\(\bigl((n-d)!\bigr)^d=n^{dn}\,2^{-\Oh(n)}\) and
\((dn)!\,4^{dn}=n^{dn}\,2^{\Oh(n)}\).
\end{proof}

\appendix

\section{The cell count in the Goodman--Pollack extension argument}
\label{a:mistake-GP}

Goodman and Pollack~\cite[\S~5(i)]{GoodmanPollack1986} stated that the arrangement of the $\binom{n}{d}$ hyperplanes spanned by $n$ points in general position in $\R^d$ has
  \[
    \sum_{k=0}^{d}\binom{\binom{n}{d}}{k}
    \;-\; n\binom{\binom{n-1}{d-1}-1}{d}
  \]
  full-dimensional cells. For $d=10$ and $n=12$, this expression equals $-32{,}550{,}951{,}444$, so it cannot be an exact cell-count formula.

\medskip

A coarser count nevertheless suffices to recover the lower bound $n^{d^2n-\Oh(n/\log n)}$ of Goodman and Pollack~\cite[\S~5(i)]{GoodmanPollack1986} for the number of realizable uniform chirotopes of size $n$ and rank $d+1$. Fix some vector $\vec{u}$ and let $\h$ denote the set of the $\binom{n}d$ hyperplanes. Suppose that the $n$ initial points are in suitable general position that (i) no edge of the arrangement of $\h$ is contained in a hyperplane orthogonal to $\vec{u}$, and (ii) any $d$ hyperplanes from $\h$ spanned by {\em disjoint} $d$-tuples of points intersect in a vertex that lies on no other hyperplane from $\h$; call these vertices {\em red}. The number of red vertices is at least $C n^{d^2}$ for all $n\geq d^2$, for some constant $C>0$ that depends only on $d$. Now, let us charge every cell of the arrangement of $\h$ bounded in the $\vec{u}$ direction to its extreme vertex in that direction. Every red vertex, being on exactly $d$ hyperplanes from $\h$, must be charged at least once. There are therefore at least $C n^{d^2}$ cells defined by $\h$. Keeping these general-position conditions at each insertion, the number of realizable uniform chirotopes of size $n$ and rank $d+1$ is therefore at least
  \[
    \prod_{m=d^2}^{n-1}Cm^{d^2}
    =n^{d^2n-\Oh(n/\log n)}.
  \]

\section{Comparison to another approaches for the upper bound of the coherent matching fields}\label{app:upper-bound-comparison}

A natural question is why not bound \(\card{\CTT(n,r-1)}\) by the classical
methods used for upper bounds on realizable order types, rather than by
the polyhedral row-counting bound of \cref{sec:upper-unsigned}.
Despite that this approache looks quite simple, it yields the strongest of the bounds 
compared below.
Classically, Goodman and Pollack~\cite{GoodmanPollack1986} bound the number
of labeled order types by applying
Milnor's theorem~\cite[Theorem~3]{Milnor1964} to the product of the
chirotope determinants; Alon~\cite{Alon1995} recasts the same count via
Warren's theorem on the signs of those determinants~\cite[Theorem~3]{Warren1968}.

The same sign-counting strategy applies here to coherent matching fields.
We spell out the resulting bound from Warren's theorem and revisit
Santos's supporting-hyperplane count~\cite{Santos2004}, which replaces
each tropical hyperplane by the ordinary hyperplanes supporting its facets.
Both estimates have the leading term \((r-1)^2\,n\log_2 n\), but larger
linear coefficients than the row-counting bound of
\cref{cor:unsigned-upper-bound}. This comparison explains the advantage
of the approach used in \cref{sec:upper-unsigned}.

\paragraph{Counting coherent matching fields with Warren's theorem.}

Put \(q=n+r\). By \eqref{eq:ct-rt-cmf}, we have
\(\card{\CTT(n,r-1)}=\card{\CMFTbullet(q,r)}
\leq\card{\CMFT(q,r)}\), so it suffices to count coherent matching
fields of type \((q,r)\).
Consider a \(q\times r\) weight matrix \(W\).  For
\(I\in\binom{[q]}{r}\) and two distinct bijections
\(\sigma,\tau\colon I\to[r]\), set
\[
  p_{I,\sigma,\tau}(W)
  \coloneqq
  \sum_{i\in I}W_{i,\sigma(i)}
  -\sum_{i\in I}W_{i,\tau(i)}.
\]
The coherent matching field of \(W\) is determined by the signs of these
comparison forms. Adding constants to rows or columns does not change
the signs, so after normalization the parameter space has dimension
\(D=(q-1)(r-1)\). There are at most
\(m=\binom{q}{r}\binom{r!}{2}\) such forms.
Warren's theorem~\cite[Theorem~3]{Warren1968} gives
\(\card{\CMFT(q,r)}\leq(4em/D)^D\), and hence, for fixed \(r\),
\begin{equation}\label{eq:warren-tropical-bound}
  \log_2\card{\CTT(n,r-1)}
  \leq
  (r-1)^2\,n\log_2 n
  +(r-1)n\log_2\!\left(\frac{2e(r!-1)}{r-1}\right)
  +\Oh(\log n).
\end{equation}

\paragraph{Santos's supporting-hyperplane count.}

Santos~\cite[Theorem~5.4]{Santos2004} replaces every tropical hyperplane by the
\(\binom{r}{2}\) ordinary affine hyperplanes supporting its facets.  This
produces a labeled arrangement of \(n\binom{r}{2}\) ordinary hyperplanes
whose chirotope refines the corresponding element of
\(\CTT(n,r-1)\). Each chirotope determinant, indexed by an
\(r\)-tuple of supporting hyperplanes, is a linear form in the apex
coordinates. After normalizing row and column translations, the parameter
space has dimension \(D=(n-1)(r-1)\), and there are at most
\(m=\binom{n\binom{r}{2}}{r}\) nonzero linear determinant forms.
Their zero sets stratify the parameter
space into cells on which the labeled chirotope, and hence the corresponding
element of \(\CTT(n,r-1)\), is constant.  By \eqref{eq:ct-rt-cmf}, the
maximal number of cells of all dimensions in such a hyperplane arrangement
gives
\[
  \card{\CTT(n,r-1)}
  =
  \card{\RT(n,r)}
  \leq
  \sum_{i=0}^{D}2^i\binom{m}{i}.
\]
Estimating this sum by \(m^D\) for large \(n\), together with a coarse
estimate of \(m\), gives the coefficient \(r(r-1)\) in the upper bound
displayed by Santos. If we keep the binomial coefficient instead, the
last summand dominates and Stirling's formula gives, for fixed \(r\),
\begin{equation}\label{eq:santos-sharp-bound}
  \log_2\card{\CTT(n,r-1)}
  \leq
  (r-1)^2\,n\log_2 n
  +(r-1)n\log_2\!\left(
    \frac{2e\binom{r}{2}^{r}}{r!(r-1)}
  \right)
  +\Oh(\log n).
\end{equation}
Thus the sharper evaluation of Santos's own cell count already replaces
\(r(r-1)\) by \((r-1)^2\) in the leading term.

\paragraph{Comparison of the bounds.}

For fixed \(r\), the bounds in \eqref{eq:warren-tropical-bound},
\eqref{eq:santos-sharp-bound}, and \cref{cor:unsigned-upper-bound} have
the same leading term \((r-1)^2\,n\log_2 n\), but different linear
coefficients. The first two coefficients are positive, whereas the
coefficient \(-r\log_2((r-1)!)\) in our row-counting bound is zero for
\(r=2\) and negative for \(r\geq3\). For example, when \(r=3\), the
first two coefficients are approximately \(7.529\) and \(7.225\),
respectively, compared with \(-3\) in \cref{cor:unsigned-upper-bound}.

\printbibliography

\end{document}